\documentclass[a4paper,reqno, 11pt]{amsart}   	
\usepackage[DIV=12, oneside]{typearea}			
\usepackage[utf8]{inputenc}
\usepackage[T1]{fontenc}
\usepackage[english]{babel}
\usepackage[centertags]{amsmath}
\usepackage{amstext,amssymb,amsopn,amsthm}
\usepackage{mathrsfs}
\usepackage{dsfont}
\usepackage{bbm}
\usepackage{thmtools}
\usepackage{graphicx}
\usepackage[titletoc,title]{appendix}
\usepackage{etexcmds}

\usepackage[backgroundcolor=white, bordercolor=blue,
linecolor=blue]{todonotes}

\usepackage{multirow}
\usepackage{tikz}
\usetikzlibrary{patterns}

\usepackage{pgfplots}
\pgfplotsset{compat=newest}
\usetikzlibrary{calc, patterns,angles,quotes}
\usepgfplotslibrary{fillbetween}
\definecolor{c595959}{RGB}{89,89,89}
\definecolor{c5a5a5a}{RGB}{90,90,90}

\usepackage{enumitem}
\setlist[enumerate]{itemsep=0mm}

\addto\extrasenglish{}
\addto\extrasenglish{}
\addto\extrasenglish{}

\usepackage[colorlinks=true, linkcolor=black, citecolor=black]{hyperref} 

\usepackage[backend=biber, maxbibnames=5, style=alphabetic, doi=false,isbn=false]{biblatex} 

\theoremstyle{plain}
\declaretheorem[title=Theorem, parent=section]{Theorem}
\declaretheorem[title=Lemma,sibling=Theorem]{Lemma}
\declaretheorem[title=Proposition,sibling=Theorem]{Proposition}
\declaretheorem[title=Corollary,sibling=Theorem]{Corollary}

\theoremstyle{definition}
\declaretheorem[title=Definition,sibling=Theorem]{Definition}
\declaretheorem[title=Remark,sibling=Theorem]{Remark}
\declaretheorem[title=Remark, numbered=no]{Remark*}

\declaretheorem[title=Assumption, numbered=no]{Assumption*}

\numberwithin{equation}{section}

\newcommand{\N}{\mathds{N}}
\newcommand{\R}{\mathds{R}}

\newcommand{\ind}{\mathds{1}}
\newcommand{\intd}{\, \mathrm{d}} 

\DeclareMathOperator{\supp}{supp}
\DeclareMathOperator{\pv}{p.v.}

\makeatletter
\providecommand\@dotsep{5}
\def\listtodoname{List of Todos}
\def\listoftodos{\@starttoc{tdo}\listtodoname}
\makeatother

\makeatletter
\def\namedlabel#1#2{\begingroup
	#2%
	\def\@currentlabel{#2}%
	\phantomsection\label{#1}\endgroup
}
\makeatother

\begin{document}
	\allowdisplaybreaks
	\title[Boundary regularity for non-local operators with vanishing horizon]{Boundary regularity for non-local operators with symmetric kernels and vanishing horizon}
	\author{Philipp Svinger}
	
	\address{Fakult{\"a}t f{\"u}r Mathematik, Universit{\"a}t Bielefeld, Postfach 10 01 31, 33501 Bielefeld, Germany}
	\email{psvinger@math.uni-bielefeld.de}
	
	\makeatletter
	\@namedef{subjclassname@2020}{%
		\textup{2020} Mathematics Subject Classification}
	\makeatother
	
	\subjclass[2020]{47G20, 35B65, 35S15, 35R09}
	
	\keywords{Nonlocal, boundary regularity, censored, regional, fractional Laplace, truncated, vanishing horizon}

\begin{abstract} 
We prove optimal Hölder boundary regularity for a non-local operator with a singular, symmetric kernel that depends on the distance to the boundary of the underlying domain.
Additionally, we prove higher boundary regularity of solutions.
\end{abstract}

	\maketitle

\section{Introduction}
Non-local energy forms have long been used to define function spaces (see, e.g., \cite{Leoni2023}) and have been intensively studied in recent years due to their connections with other fields, such as Markov processes (see \cite{Fukushima2010}) and peridynamics (see \cite{Silling2000}).
The non-local nature of these energy forms leads to challenging questions regarding coercivity estimates and functional inequalities.
In \cite{Dyda2006}, Dyda proved the following comparability result for non-local energy forms defining the Sobolev-Slobodeckij space $H^s(\Omega )$.
Given a bounded Lipschitz domain $\Omega \subset \R^d$, $s\in (0,1)$, and $\sigma \in (0,1]$, there exists some constant $c=c(\Omega ,d,s,\sigma )>0$ such that for any $u\in L^2 (\Omega )$
\begin{equation} \label{eq:ComparabiltiyDydaInequlatiy}
\int _\Omega \int _\Omega \frac{(u(x)-u(y))^2}{|x-y|^{d+2s}} \intd y \intd x \leq c \int _\Omega \int_{B_{ \sigma d_\Omega (x) }(x)} \frac{(u(x)-u(y))^2}{|x-y|^{d+2s}} \intd y \intd x
\end{equation}
where $d_\Omega (x)$ denotes the distance from $x$ to the boundary $\partial \Omega$.
By setting
\begin{align*}
\mathcal{E}_{\Omega ,\sigma} (u,v) &= \frac{1}{2}\int_\Omega \int_{B_{ \sigma d_\Omega (x) }(x)} \frac{(u(x)-u(y))(v(x)-v(y))}{|x-y|^{d+2s}} \intd y\intd x \\
&  = \frac{1}{2}\int_\Omega \int_\Omega \frac{(u(x)-u(y))(v(x)-v(y))}{|x-y|^{d+2s}} \frac{ \mathbbm{1} _{\{ y\in B_{\sigma d_\Omega (x)}(x) \} }  +\mathbbm{1} _{\{ x\in B_{\sigma d_\Omega (y)}(y) \} } }{2} \intd y\intd x ,
\end{align*}
the estimate \eqref{eq:ComparabiltiyDydaInequlatiy} establishes that $[u]_{H^s(\Omega )}$ is comparable to $\mathcal{E}_{\Omega ,\sigma} (u,u)$.
This energy form has an associated integro-differential operator
\begin{equation} \label{eq:DefOperator}
\mathcal{L}_{\Omega ,\sigma } u(x) = \pv \int_\Omega (u(x)-u(y)) \mathcal{K}_{\Omega , \sigma }(x,y) \intd y = \lim _{\epsilon \searrow 0} \int_{\Omega \setminus B_{\epsilon}(x)} (u(x)-u(y)) \mathcal{K}_{\Omega , \sigma }(x,y) \intd y.
\end{equation}
where
\begin{equation} \label{eq:KernelOperator}
\mathcal{K}_{\Omega ,\sigma }(x,y)  = |x-y|^{-d-2s} \frac{ \mathbbm{1} _{\{ y\in B_{\sigma d_\Omega (x)}(x) \} }  +\mathbbm{1} _{\{ x\in B_{\sigma d_\Omega (y)}(y) \} }  }{2}  .
\end{equation}

This paper investigates the boundary regularity of solutions to the Dirichlet problem
\begin{gather} \label{eq:DirichletProblem}
\begin{aligned}
\mathcal{L}_{\Omega , \sigma} u &= f \quad \text{ in }\Omega , \\
u& =0 \quad \text{ on } \partial \Omega .
\end{aligned}
\end{gather}
Note that we need a restriction $s>1/2$ to define a trace of $u$ in the variational setup and impose boundary values $u=0$ on $\partial \Omega$.

\subsection{Main Results}
For the operator $\mathcal{L}_{\Omega , \sigma}$ from \eqref{eq:DefOperator}, we will prove the following boundary regularity result.
\begin{Theorem}\label{Thm:HoelderContinuityWeakSol}
Let $\Omega \subset \R^d$ be a bounded $C^{1,1}$ domain, $s\in (1/2,1)$, and $\sigma \in (0,1]$.
Assume that $u\in H^s_0(\Omega )$ is a weak solution to the Dirichlet problem \eqref{eq:DirichletProblem} for some $f\in L^\infty (\Omega )$.
Then $u\in C^{2s-1 }( \overline{\Omega })$ with
\begin{equation} \label{eq:HoelderBoundaryRegularity}
\| u \| _{C^{2s-1} (\overline{\Omega } )} \leq c   \| f \| _{L^\infty (\Omega )} 
\end{equation} 
for some constant $c=c(\Omega ,d,s,\sigma ) >0$.
\end{Theorem}
In terms of higher boundary regularity, we obtain the following result.
\begin{Theorem}\label{Thm:HigherOrderBoundaryReg}
Let $\Omega \subset \R^d$ be a bounded $C^{1,1}$ domain, $s\in (1/2,1)$, and $\sigma \in (0,1]$.
Assume, that as in \autoref{Thm:HoelderContinuityWeakSol}, $u\in H^s_0(\Omega )$ solves \eqref{eq:DirichletProblem} for some $f\in L^\infty (\Omega )$.
Then there exists some $\alpha = \alpha (d,s,\sigma )\in (0,1)$ such that $u/d_\Omega ^{2s-1} \in C^{\alpha}( \overline{\Omega })$ with
\begin{equation} \label{eq:HigherBoundaryRegularity}
\left\| \frac{u}{d_\Omega ^{2s-1}} \right\| _{C^{\alpha } (\overline{\Omega } )} \leq c   \| f \| _{L^\infty (\Omega )} 
\end{equation} 
for some constant $c=c(\Omega ,d,s,\sigma ) >0$.
\end{Theorem}
\begin{Remark}\label{Rem:CommentsOnMainResults}We start with some comments on these results.
\begin{itemize}
\item The associated operator to the $H^s$ energy $[u]_{H^s(\Omega)}$ is the regional fractional Laplacian (see \eqref{eq:RegionalFractionalLaplacian} for a definition), for which results corresponding to \autoref{Thm:HoelderContinuityWeakSol} and \autoref{Thm:HigherOrderBoundaryReg} are known (see, e.g., \cite{Fall2022}).
Note, that the comparability of energies $\mathcal{E}_{\Omega ,\sigma}(u,u)$ and $[u]_{H^s(\Omega)}$ does not directly imply that \autoref{Thm:HoelderContinuityWeakSol} and \autoref{Thm:HigherOrderBoundaryReg} also hold true for the operator $\mathcal{L}_{\Omega , \sigma}$.
\item Given $f\in L^{2}(\Omega )$, the comparability \eqref{eq:ComparabiltiyDydaInequlatiy} implies the existence and uniqueness of weak solutions $u\in H^s_0( \Omega )$ to the Dirichlet problem \eqref{eq:DirichletProblem}.
\item In the half-space, we can explicitly compute the action of $\mathcal{L}_{\R^d_+,\sigma}$ on the monomials $\varphi (x) = x_d ^{\gamma}$. Indeed, as for the regional fractional Laplacian, the function $\varphi (x) = x^{2s-1}$ is a harmonic function in the half-line, i.e. $\mathcal{L}_{\R _+,\sigma} \varphi =0$ on $\R _+$ (compare \autoref{Prop:CalculationsInHalfspace}).
\item The $C^{2s-1}$ Hölder regularity from \eqref{eq:HoelderBoundaryRegularity} is optimal. Even if $\partial \Omega$ and $f$ would be smooth, in general $u\notin C^\alpha (\overline{\Omega} )$ for any $\alpha >2s-1$.
To see this, note that the weak solution $u\in H^s_0(\Omega )$ to \eqref{eq:DirichletProblem} with $f \equiv 1$ satisfies $u \asymp d_\Omega ^{2s-1}$ due to the bound $|u| \lesssim d_\Omega ^{2s-1}$ (see \eqref{eq:uLeqDist}) and a Hopf type result (compare \autoref{Rem:HopfTypeLemma}). Hence, this solution is not in $ C^\alpha (\overline{\Omega} )$ for any $\alpha >2s-1$.
\item 
Due to the form of $\mathcal{K}_{\Omega ,\sigma}$, the domain of integration of $\mathcal{L}_{\Omega ,\sigma}$ is very different in the cases $\sigma \in (0,1)$ and $\sigma =1$.
If we evaluate $\mathcal{L}_{\Omega ,\sigma} u(x)$ at some point $x\in \Omega$, the operator integrates around $x$ in two domains corresponding to the two summands in \eqref{eq:KernelOperator} (compare with \autoref{fig:SupportKernelSigma23} and \autoref{fig:SupportKernelSigma1} for examples in the case where $\Omega$ is the half-space).
If $\sigma \in (0,1)$, then both of these domains shrink down to the point $x$ as $x$ approaches the boundary.
However, in the case $\sigma =1$, the domain corresponding to the second summand in \eqref{eq:KernelOperator}, will always include points far into the domain.
This phenomena causes differences in the proofs of \autoref{Thm:HoelderContinuityWeakSol} and \autoref{Thm:HigherOrderBoundaryReg} for the cases $\sigma \in (0,1)$ and $\sigma =1$.
\end{itemize}
\end{Remark}
Let us mention two related open problems.
\begin{enumerate}[label=\roman*)]
\item Due to the method of the proof, we have no control over $\alpha$ in \autoref{Thm:HigherOrderBoundaryReg}.
Hence, the optimal higher boundary regularity remains open.
\item In \cite{Kassmann2022}, a non-local operator with visibility constraint was introduced 
\begin{equation}\label{eq:VisibilityConstraintOperator}
Lu(x)= \pv \int _\Omega \frac{u(x)-u(y)}{|x-y|^{d+2s}} \left( \ind _{\Omega _x}(y) + \ind _{\Omega _y}(x) \right) \intd y
\end{equation}
where $\Omega_x= \{ y \mid tx+(1-t)y \in \Omega , \forall t\in (0,1) \}$.
This operator only considers tuples $(x,y) \in \Omega \times \Omega$ if the line segment between $x$ and $y$ is contained in $\Omega$.
With the same methods as in this paper, \autoref{Thm:HoelderContinuityWeakSol} and \autoref{Thm:HigherOrderBoundaryReg} could be proven for $L$.
It would be interesting to find a general class of non-local operators on domains including $\mathcal{L}_{\Omega ,\sigma}$ and $L$ which allow the kernel to vanish, such that \autoref{Thm:HoelderContinuityWeakSol} and \autoref{Thm:HigherOrderBoundaryReg} still holds true.
However, this seems more difficult because in the proof we heavily use that the operator $\mathcal{L}_{\Omega ,\sigma}$ has an associated version $\mathcal{L}_{\R^d_+,\sigma}$ in the half-space, for which explicit computations are possible.
\end{enumerate}

\subsection{Related literature}

Regularity theory for non-local equations has been an active field of research for the past 25 years.
In particular, understanding regularity up to the boundary is of great interest since non-local equations give rise to phenomena that have no counterpart in the local setting.
Indeed, let $u$ solve the Dirichlet problem
\begin{gather}\label{eq:DirichletProblemGeneralOperator}
\begin{aligned}
L u & = f \quad \text{in} \quad \Omega , \\
u& = 0\quad \text{on} \quad  \partial\Omega 
\end{aligned}
\end{gather}
for $L=-\Delta$ on some bounded smooth domain $\Omega \subset \R^d$.
Then, if $f$ is smooth up to the boundary, so is $u$ (see, e.g., \cite[Thm. 2.35]{FernandezReal2022}).
In contrast, if $L$ is the fractional Laplacian
\begin{equation*}
 (-\Delta )^s u(x) = c_{d,s} \pv \int _{\R^d} \frac{u(x)-u(y)}{|x-y|^{d+2s}} \intd y  ,
\end{equation*}
for some $s\in (0,1)$, solutions $u$ to the exterior value problem
\begin{gather}\label{eq:DirichletProblemExteriorValueGeneralOperator}
\begin{aligned}
L u & = f \quad \text{in} \quad \Omega , \\
u& = 0\quad \text{on} \quad  \Omega ^c
\end{aligned}
\end{gather}
are in general only $C^s$ Hölder continuous up to the boundary, even if $f$ is smooth (see, e.g., \cite{RosOton2014}).
Another type of non-local operator studied in the literature is the regional (or sometimes called censored) fractional Laplacian (see, e.g., \cite{Chen2018}, \cite{Fall2022})
\begin{equation}\label{eq:RegionalFractionalLaplacian}
(-\Delta )^s_\Omega u(x) = c_{d,s} \pv \int _{\Omega} \frac{u(x)-u(y)}{|x-y|^{d+2s}} \intd y .
\end{equation}
For $L=(-\Delta )^s_\Omega$ and $s\in (1/2,1)$, solutions $u$ to the Dirichlet problem \eqref{eq:DirichletProblemGeneralOperator} are in general only $C^{2s-1}$ up to the boundary. As mentioned above, the restriction $s>1/2$ is needed to define a trace of $u$ in the variational setup and impose boundary values $u=0$ on $\partial \Omega$.

A typical approach to study boundary regularity of solutions to \eqref{eq:DirichletProblemGeneralOperator} entails the characterization of non-trivial $L$-harmonic functions in the half-line $\R_+$, i.e., solutions to 
\begin{equation} \label{eq:DirichiletProbleHalfLineGeneralOp}
\begin{aligned}
L\varphi &= 0 \quad \text{in} \quad \R_+ , \\
\varphi (0)&=0 .
\end{aligned}
\end{equation}
Indeed, the function $\varphi _1(x)=x^{2s-1}$ solves \eqref{eq:DirichiletProbleHalfLineGeneralOp} for the regional fractional Laplacian as well as for the operator $\mathcal{L}_{\Omega ,\sigma}$ from \eqref{eq:DefOperator}.
For the fractional Laplacian, a harmonic solution in the half-line is given by $\varphi _2(x):= (x_+)^s$.

Closely related to this work is the paper \cite{Chan2023}, where Hölder boundary regularity of the operator 
\begin{equation}\label{eq:ChanOperator}
L_\Omega  u(x) = \pv \int _{B_{d_\Omega (x)}(x)} \frac{u(x)-u(y)}{|x-y|^{d+2s}} \intd y
\end{equation}
is studied.
In comparison to $(-\Delta )_\Omega ^s$, the domain of integration in the integro-differential operator $L_\Omega$ is further reduced to the ball around $x$ inside $\Omega$ with maximal radius.
The operator $L_\Omega$ has the property that $d_\Omega (x)^{2s-2} L_\Omega u (x) \to c_{d,s} (-\Delta )u (x)$ as $x\to \partial \Omega$ for smooth $u$ (see \cite[Lem. A.1]{Chan2023}).
Furthermore, due to the symmetry of the integration domain in the operator \eqref{eq:ChanOperator}, the function $\varphi _3(x)=x$ is $L_{\R_+}$-harmonic in $\R_+$.
Note that $\varphi _3$ approaches the boundary linearly, indicating a different boundary behavior of solutions to the Dirichlet problem compared to the regional fractional Laplacian.
Indeed, given a bounded $C^{1,1}$ domain $\Omega \subset \R^d$ and $s \in (1/2,1)$, it was proven in \cite{Chan2023}, that
\begin{equation}\label{eq:ChanResult}
\left. \begin{matrix}
u \in C^{2s+\epsilon}_{loc}(\Omega )\cap C(\overline{\Omega} ) \text{ solves \eqref{eq:DirichletProblemGeneralOperator} for } L=L_\Omega \\
d_\Omega  ^{2s-2}f\in L^\infty (\Omega )
\end{matrix} \right\} \Rightarrow u\in C^{1,\alpha }(\overline{\Omega}) \text{ for some } \alpha \in (0,1) .
\end{equation}
The assumption $d_\Omega  ^{2s-2}f\in L^\infty (\Omega )$ (compared to $ f\in L^\infty (\Omega )$ as in \autoref{Thm:HoelderContinuityWeakSol}) is a decay assumption on $f$ at the boundary. However, it does not directly interfere with the result $u\in C^{1,\alpha }(\overline{\Omega})$.
By adapting the proof in \cite{Chan2023}, one could also prove \eqref{eq:ChanResult} under the assumption $f\in L^\infty (\Omega )$.

We can rewrite the operator from \eqref{eq:ChanOperator} as $L_\Omega u(x) = \pv \int _\Omega (u(x)-u(y)) K(x,y) \intd y$ where $K$ is the kernel 
\begin{equation*}
K(x,y)=|x-y|^{-d-2s} \mathbbm{1} _{B_{\sigma d_\Omega (x)}(x)} (y) .
\end{equation*}
and $\sigma$ is the fixed constant $\sigma =1$. Note that this kernel is not symmetric.
By symmetrizing the kernel $K$, we obtain the operator $\mathcal{L}_{\Omega ,\sigma}$.

We want to stress the different boundary regularities for solutions to the Dirichlet problem involving the operators $L_\Omega$ and $\mathcal{L}_{\Omega ,\sigma}$.
This effect is mainly due to the symmetry properties of the corresponding kernels.
The operator $L_\Omega$ corresponds to an operator in non-divergence form whereas $\mathcal{L}_{\Omega ,\sigma }$ corresponds to an operator in divergence form.
We also want to mention that the effect that symmetric kernels of non-local operators may lead to different boundary behavior of solutions to \eqref{eq:DirichletProblemExteriorValueGeneralOperator} than corresponding non-symmetric kernels is known in the literature (see, e.g., \cite{Dipierro2022} where boundary regularity results were proven for translational invariant non-local operators in terms of the corresponding Fourier symbols).

There has also been much interest recently in non-local operators of the form
\begin{equation}\label{eq:OperatorScottDu}
Lu(x)= \int _\Omega (u(x)-u(y)) (k(x,y)+k(y,x)) \intd y
\end{equation}
with associated energy $\iint _{\Omega \times \Omega } (u(x)-u(y))^2 k(x,y) \intd y \intd x$ where
\begin{equation*}
k(x,y)=\frac{1}{|x-y|^\beta} \frac{\ind _{\{ y\in B_{\sigma d_\Omega (x)} (x) \}}}{(\sigma d_\Omega (x) )^{d+2-\beta }}
\end{equation*}
for some $\beta \in [0, d+2)$ and $\sigma >0$ (see, e.g., \cite{Tian2017}, \cite{Scott2024}, \cite{Scott2025}, \cite{Du2022}).
This kernel differs from the one studied in this work due to the factor $(d_\Omega (x)) ^{-d-2+\beta }$.
As a result, the kernel not only has a possible singularity on the diagonal but also as $x$ approaches the boundary.
We also want to mention that, due to scaling properties, the action of this operator on monomials in the half-line can be explicitly computed (compare \autoref{Rem:GeneralizeComputationsInHalfspace}).

For the fractional Laplacian, $C^s$ regularity up to the boundary of solutions to \eqref{eq:DirichletProblemExteriorValueGeneralOperator} with right-hand side $f\in L^\infty (\Omega )$ in $C^{1,1}$ domains was proven in \cite{RosOton2014}.
In multiple works (\cite{Grubb2014}, \cite{Grubb2015}, \cite{RosOton2016a}, \cite{RosOton2016b}, \cite{RosOton2017}), this was generalized to a broader class of operators and less regularity assumptions on the boundary of the domain.
In \cite{RosOton2024Book} $C^s$ boundary regularity was proven for non-degenerate $2s$-stable operators in $C^{1,\alpha}$ domains.
Recently, in \cite{Grube2024}, this result was generalized to domains satisfying a $C^{1,\text{dini}}$-condition.

There are also results on boundary regularity for operators
\begin{equation*}
Lu(x)= \pv \int_{\R^d} (u(x)-u(y)) k(x,y) \intd y
\end{equation*}
with a kernel $k$ satisfying $0<\lambda \leq k(x,y) |x-y|^{d+2s} \leq \Lambda$.
If $k$ is translational invariant, i.e., $k(x,y)=\tilde{k}(x-y)$, then it was shown in \cite{RosOton2024}, that solutions to \eqref{eq:DirichletProblemExteriorValueGeneralOperator} with $f\in L^\infty (\Omega )$ are $C^s$ up to the boundary.
An additional assumption on the kernel is needed for the nontranslational invariant case. In \cite{Kim2024}, the optimal $C^s$ Hölder regularity was proven for Hölder continuous kernels, which are not necessarily translational invariant.

The regional fractional Laplacian (see \eqref{eq:RegionalFractionalLaplacian}) is the infinitesimal generator of the censored symmetric stable process, which was constructed in \cite{Bogdan2003}.
For the regional fractional Laplacian, solutions to the Dirichlet problem \eqref{eq:DirichletProblemGeneralOperator} for $s>1/2$ with $f\in L^\infty (\Omega )$ are in general only $C^{2s-1}$ up to the boundary. This was proven in \cite{Chen2018} using the Green function estimates from \cite{Chen2002}.
The result was optimized in \cite{Fall2022} for $C^{1,\alpha}$ domains and more general right-hand sides. Additionally, \cite{Fall2022} not only investigated the Dirichlet problem but also provided results on boundary regularity for the Neumann problem for the regional fractional Laplacian.

In terms of higher boundary regularity for the fractional Laplacian, there are many results which prove Hölder regularity for $u/d_\Omega ^{s}$ by using various regularity assumptions on the right-hand side $f$ (see, e.g., \cite{RosOton2014}, \cite{RosOton2017}, \cite{Grubb2014}, \cite{Grubb2015}, \cite{Abatangelo2020}).
For the regional fractional Laplacian, regularity for $u/d_\Omega ^{2s-1}$ was proven in \cite{Fall2022} and \cite{Fall2022a}.

In \cite{Cho2024}, Green function estimates and Boundary Harnack Principles were proven for operators with kernels decaying at the boundary (see also \cite{Kim2021} and \cite{Kim2023} for results in the half-space).
However, their results do not cover the operator studied in this paper.

As discussed above, the operator $\mathcal{L}_{\Omega ,\sigma}$ can be seen as an operator with visibility constraint that was studied in \cite{Kassmann2022}.
In \cite{Kassmann2022}, energy forms associated with \eqref{eq:VisibilityConstraintOperator} were investigated, and Poincar\'e inequalities were proven.

Another type of operator similar to $\mathcal{L}_{\Omega ,\sigma}$ is studied in \cite{Bellido2021} where the domain of integration in \eqref{eq:RegionalFractionalLaplacian} is changed from $\Omega$ to some ball with fixed radius $B_{r_0}(x)$.
For this operator, $C^s$ boundary regularity of solutions to \eqref{eq:DirichletProblemExteriorValueGeneralOperator} can be obtained by writing the equation in terms of the fractional Laplacian (similar as in the proof of \autoref{Lem:InteriorRegularity}) and using the boundary regularity for $(-\Delta )^s$.

\subsection{Strategy of the proof}
In the proof of \autoref{Thm:HoelderContinuityWeakSol}, we follow the typical approach of combining an estimate $|u| \lesssim d_\Omega ^{2s-1}$ with interior regularity results.
We deduce the interior regularity results from known results about the fractional Laplacian by writing $(-\Delta )^s u = f + (-\Delta )^s u - \mathcal{L}_{\R^d_+, \sigma} u$ and treating the difference $(-\Delta )^s u - \mathcal{L}_{\R^d_+, \sigma} u$ as a right-hand side.

The key ingredient in the proof of the estimate $|u| \lesssim d_\Omega ^{2s-1}$ is a barrier of the form $\varphi = c d_\Omega ^{2s-1} -  d_\Omega ^{2s-1+\epsilon}$.
In the half-space \smash{$\R^d_+$} we can explicitly compute the action of the operator \smash{$\mathcal{L}_{\R^d_+, \sigma}$} on this barrier (see \autoref{Prop:CalculationsInHalfspace}).
In fact, the function \smash{$d_{\R^d_+} ^{2s-1}$} is a \smash{$\mathcal{L}_{\R^d_+, \sigma}$}-harmonic solution in the half-space.
Furthermore,
\begin{equation*}
\mathcal{L}_{\R^d_+, \sigma} d_{\R^d_+} ^{2s-1+\epsilon }= c_1 d_{\R^d_+} ^{-1+\epsilon }
\end{equation*} 
for some $c_1<0$.

In \autoref{sec:Barriers}, we follow the proof of \cite[Appendix B.2]{RosOton2024Book} to conclude that $ \mathcal{L}_{\Omega ,\sigma} \varphi \gtrsim d_\Omega ^{-1+\epsilon}$ close to $\partial \Omega$.
However, herein, difficulties arise since the operator $ \mathcal{L}_{\Omega ,\sigma}$ is not translational invariant and depends on the domain $\Omega$.
A precise analysis of the geometry of our operator in the half-space is needed.

To prove \autoref{Thm:HigherOrderBoundaryReg}, we use a standard blow-up argument together with a Liouville type result in the half-space (see \autoref{Prop:LiouvilleInHalfspace}).
The Liouville type result follows in two steps.
First, using that \smash{$\mathcal{L}_{\R^d_+ , \sigma }$} is translational invariant in the tangential direction $\R^{d-1}$, we obtain that a solution \smash{$\mathcal{L}_{\R^d_+ , \sigma } u=0$} in \smash{$\R^d_+$} is a one-dimensional function, i.e. $u(x)=\tilde{u}(x_d)$.
Next, we prove a boundary Harnack type result where we show that $\tilde{u}(y)/y^{2s-1}$ can be extended continuously up to $0$ and is locally Hölder continuous (see \autoref{Prop:BoundaryHarnack}). Finally, we use this to show that $u$ is a constant multiple of the harmonic solution $x_d^{2s-1}$.

For the proof of the Liouville type result, we mostly follow the proof of \cite{Audrito2023}.
However, also here, difficulties arise.
After we have proved that the harmonic solution $u$ in the half-space $\R^d_+$ is a one-dimensional function, i.e. $u(x)=\tilde{u}(x_d)$, we cannot directly conclude that also $\mathcal{L}_{\R_+,\sigma}\tilde{u}=0$ in $\R_+$. Instead, we must still work with the operator on $\R^d_+$.
This leads to difficulties in proving the boundary Harnack type result since the operator degenerates as we approach the boundary.
Especially in the case $\sigma =1$, this causes a correction factor $x_d^{(d-1)/2}$ throughout \autoref{sec:BoundaryHarnack}.

\subsection{Outline}
In \autoref{sec:Preliminaries}, we introduce some notation and the weak solution concept for our operator. Furthermore, we prove a first interior regularity result and provide explicit calculations in the half-space.
\hyperref[sec:Barriers]{Section 3} contains the construction of the barrier needed to prove Hölder regularity up to the boundary.
An $L^\infty$ estimate for $u$ and the completion of the proof for \autoref{Thm:HoelderContinuityWeakSol} is contained in \autoref{sec:HoelderRegUpToBoundary}.
In \autoref{sec:BoundaryHarnack}, we will prove a boundary Harnack type result in the half-space, which leads to a Louville type result in \autoref{sec:LiouvilleTypeResult}.
Finally, in \autoref{sec:HigherBoundaryReg}, we complete the proof of \autoref{Thm:HigherOrderBoundaryReg}.

\subsection{Acknowledgments}
Financial support by the the Deutsche Forschungsgemeinschaft (DFG, German Research Foundation) – Project-ID 317210226 – SFB 1283 is gratefully acknowledged.
The author is grateful to Moritz Kassmann, Marvin Weidner, and Florian Grube for helpful discussions and for valuable comments on the manuscript.

\section{Preliminaries} \label{sec:Preliminaries}
In this section, we will introduce some notation and define the weak solution concept for \autoref{Thm:HoelderContinuityWeakSol} and \autoref{Thm:HigherOrderBoundaryReg}.
Additionally, we prove a first interior regularity result and present some explicit computations in the half-space.

We will write $a\wedge b := \min \{ a,b\}$ and $a\vee b := \max \{ a,b\}$.
Furthermore, for $a\in \R$ we use the notation $a=a_+ -a_-$ where $a_+:= a\vee 0$ is the positive part and $a_-:=- (0 \wedge a)$ is the negative part.
In addition, we write $a\lesssim b$ if there exists a positive constant $c>0$ such that $a\leq c b$.
We will use the notation $a\asymp b$ if $a\lesssim b$ and $b \lesssim a$.
Throughout the paper, we use the letter $c$ to denote some positive constant.
However, note, that $c$ might change from line to line.
Moreover, if $\Omega \subset \R^d$ is a domain in $\R^d$, then by
\begin{equation*}
d_\Omega (x)  = \inf \{ |x-y| \colon y\in \Omega ^c \} ,
\end{equation*}
we denote the distance from $x$ to $\Omega ^c = \R^d \setminus \Omega$.

By $C^\infty _c (\Omega )$, we denote the space of smooth functions with compact support inside $\Omega \subset \R^d$.
For some $\alpha \in (0,1)$, we will write $C^\alpha (\overline{\Omega } )$ to denote the space of Hölder continuous functions.
This space is equipped with the norm
\begin{equation*}
\| u\| _{C^\alpha (\overline{\Omega } )} = \sup _{x\in \overline{\Omega }} |u(x)| + \sup _{x,y \in \overline{\Omega }} \frac{|u(x)- u(y)|}{|x-y|^\alpha } .
\end{equation*}
For $s\in (0,1)$, we define a seminorm
\begin{equation*}
[u] ^2 _{H^s(\Omega )} = \int _\Omega \int_\Omega \frac{(u(x)-u(y))^2}{|x-y|^{d+2s}} \intd x \intd y .
\end{equation*}
Now, we can define the fractional Sobolev space $H^s(\Omega )$ as the space of $L^2$ functions on $\Omega$ with $[u]_{H^s (\Omega )} < \infty$.
We equip this space with the norm
\begin{equation*}
\| u \| ^2_{H^s(\Omega )} = \| u \| ^2_{L^2(\Omega )} + [u]^2_{H^s (\Omega )} .
\end{equation*}
By $H^s_0 (\Omega )$, we will denote space we  obtain by completion of $C^\infty _c (\Omega )$ under the $H^s( \Omega )$ norm.
Note that if $s\in (1/2, 1)$ and $\Omega$ is a bounded Lipschitz domain, there is an equivalent definition (see, e.g., \cite[Thm. 3.33]{McLean2000})
\begin{equation}\label{eq:ExtendingHs0Functions}
H^s_0 (\Omega ) = \left\{ u\in H^s (\R^d ) \mid u=0 \quad \text{on} \quad \R^d\setminus \Omega \right\} .
\end{equation}
Hence, throughout the paper, we will extend functions $u\in H^s_0( \Omega )$ by $u=0$ on $\Omega ^c$.
By $H^s_{loc} (\Omega )$, we will denote space of functions $u$ such that the product $u\varphi \in H^s (\Omega )$ for all $\varphi \in C^\infty _c (\Omega )$.

Furthermore, to simplify notation throughout this paper, we will occasionally omit the underlying domain $\Omega$ or the parameter $\sigma$ from the notation of the operator $\mathcal{L}_{\Omega ,\sigma}$, the kernel $\mathcal{K}_{\Omega ,\sigma}$ and the bilinear form $\mathcal{E}_{\Omega ,\sigma}$.

\subsection{Weak solution concept}
The operator $\mathcal{L}_{\Omega ,\sigma}$ from \eqref{eq:DefOperator} is defined in terms of the symmetric kernel $\mathcal{K}_{\Omega ,\sigma}$ (see \eqref{eq:KernelOperator}).
Hence, we have an associated energy form 
\begin{equation} \label{eq:EnergyForm}
\mathcal{E}_{\Omega ,\sigma} (u,v) = \frac{1}{2}\int_\Omega \int_\Omega (u(x)-u(y))(v(x)-v(y)) \mathcal{K}_{\Omega ,\sigma}(x,y) \intd y\intd x .
\end{equation}
As mentioned in the beginning of the introduction, this energy is comparable to the $H^s$ energy $[u]_{H^s(\Omega )}$ in the following sense (see \cite[eq. (13)]{Dyda2006}).
\begin{Corollary}\label{Cor:EnergiesComparable}
Let $\Omega \subset \R^d$ be a bounded Lipschitz domain.
Then, there exists a constant $c=c(\Omega ,d , s,\sigma )>0$ such that
\begin{equation*}
\mathcal{E}_{\Omega , \sigma } (u,u) \leq [u] ^2 _{H^s(\Omega )} \leq c \mathcal{E}_{\Omega , \sigma } (u,u) .
\end{equation*}
\end{Corollary}

\begin{Definition}[Weak solution concept]
Fix $s\in (1/2,1)$ and $\sigma \in (0,1]$.
Let $\Omega \subset \R^d$ be a Lipschitz domain and let $\Omega ' \subset \Omega$ be a bounded Lipschitz subdomain.
Given $f\in L^2 (\Omega ')$ we say that a function $u$ satisfying
\begin{equation}\label{eq:AssumptionOnUWeakSol}
\iint _{\Omega \times \Omega \setminus (\Omega '^c \times \Omega '^c )} (u(x)-u(y))^2 \mathcal{K}_{\Omega ,\sigma}(x,y) \intd x \intd y < \infty
\end{equation}
is weak solution to $\mathcal{L}_{\Omega ,\sigma} u =f$ in $\Omega '$ if
\begin{equation}\label{eq:WeakSolConcept}
\mathcal{E}_{\Omega ,\sigma} (u,\varphi ) = \int _\Omega f(x) \varphi (x)  \intd x
\end{equation}
for all $\varphi \in H^s_0(\Omega ')$.
\end{Definition}
\begin{Remark}
Note, that in the definition above, $\mathcal{E}_{\Omega ,\sigma} (u,\varphi )$ is finite for all $\varphi \in H^s_0(\Omega ')$.
\end{Remark}
Now assume, that as in \autoref{Thm:HoelderContinuityWeakSol} and \autoref{Thm:HigherOrderBoundaryReg}, $\Omega = \Omega '$.
Then, using \autoref{Cor:EnergiesComparable}, the assumption in \eqref{eq:AssumptionOnUWeakSol} is equivalent to $[u]_{H^s (\Omega )} < \infty$.
Hence, we can give the following definition for a weak solution to the Dirichlet problem \eqref{eq:DirichletProblem}.
\begin{Definition}
Fix $s\in (1/2,1)$ and $\sigma \in (0,1]$ and let $\Omega\subset \R^d$ be a bounded $C^{1,1}$ domain.
Given $f\in L^\infty (\Omega )$, we say that a function $u\in H^s_0 (\Omega )$ is a weak solution to the Dirichlet problem \eqref{eq:DirichletProblem} if \eqref{eq:WeakSolConcept} is satisfied for all $\varphi \in H^s_0(\Omega )$.
\end{Definition}

\subsection{Interior regularity}
We can rewrite the kernel $\mathcal{K}_{\Omega ,\sigma}$ of our operator $\mathcal{L}_{\Omega ,\sigma}$ from \eqref{eq:DefOperator} as
\begin{equation*}
\mathcal{K}_{\Omega ,\sigma} (x,y) = \frac{\mathcal{B}_{\Omega ,\sigma} (x,y)}{|x-y|^{d+2s}} \quad  \text{where}\quad \mathcal{B}_{\Omega ,\sigma} (x,y) = \frac{\mathbbm{1} _{\{ y\in B_{\sigma d_\Omega (x)}(x) \} }  +\mathbbm{1} _{\{ x\in B_{\sigma d_\Omega (y)}(y) \} } }{2} .
\end{equation*}
Notice that $0\leq \mathcal{B}_{\Omega ,\sigma} (x,y) \leq 1$.
Furthermore, for every $x\in \Omega$ in the interior, we have $\mathcal{B}_{\Omega ,\sigma} (x,y) =1$ for all $y$ in some small neighborhood around $x$. To be more precise:
\begin{equation}\label{eq:BLocally1AroundPoint}
\mathcal{B}_{\Omega ,\sigma} (x,y) =1 \quad \text{for all}\quad y \in B_{cd_\Omega (x)} (x)
\end{equation}
where $c=\sigma /(1+\sigma )$.
Indeed, let $y \in B_{cd_\Omega (x)} (x)$. Then, 
\begin{equation*}
|x-y| \leq c d_\Omega (x) = \sigma \left( d_\Omega (x) - cd_\Omega (x) \right) \leq \sigma d_\Omega (y)
\end{equation*}
which implies $x\in B_{\sigma d_\Omega (y)}(y)$.
Furthermore, $y\in B_{\sigma d_\Omega (x)}(x)$ since $c\leq \sigma$ which shows \eqref{eq:BLocally1AroundPoint}.

We continue by proving our first regularity result for our operator $\mathcal{L}$.
For this, we define the weighted $L^\infty$ norm
\begin{equation*}
\| u\| _{L^\infty _{\kappa} (\Omega )} := \left\| \frac{u(x)}{1+|x|^{\kappa}} \right\| _{L^\infty (\Omega )} .
\end{equation*}
\begin{Lemma}\label{Lem:InteriorRegularity}
Fix $s\in (1/2,1)$, $\sigma \in (0,1]$, $\epsilon \in (0,2s)$ and let $\Omega \subset \R^d$ be a domain with $B_{2+\frac{3}{\sigma}} \subset \Omega$.
Assume that $u$ is a weak solution of $\mathcal{L}_{\Omega , \sigma} u = f$ in $B_1$ with $u\in L^\infty _{2s-\epsilon} (\Omega )$ and $f\in L^{\infty}(B_1)$.
Then $u\in C^{2s} (\overline{ B_{1/2}})$ and
\begin{equation*}
\| u \| _{C^{2s}(\overline{ B_{1/2}})} \leq c \left( \| u\| _{L^\infty _{2s-\epsilon} (\Omega )} + \| f\|_{L^\infty (B_1)} \right)
\end{equation*}
for some constant $c=c(s,d,\epsilon )>0$.
\end{Lemma}
\begin{proof}
Rewriting our operator $\mathcal{L}_{\Omega}$ in terms of the fractional Laplacian, we obtain
\begin{align*}
(-\Delta )^s u(x) & = f(x) + u(x) \int _{\R^d} \frac{1-\mathcal{B}_\Omega (x,y)}{|x-y|^{d+2s}} \intd y   -\int _{\R^d} \frac{u(y) (1-\mathcal{B} _\Omega (x,y))}{|x-y|^{d+2s}} \intd y \\
& = : f(x) + h(x)
\end{align*}
for $x\in B_1$ in the weak sense.

\textbf{Claim:} $\mathcal{B}_\Omega(x,y)=1$ for all $x\in B_1$ and $y\in B_2$.

For $x\in B_1$ and $y\in B_2$ we have $y\in B_3(x) \subset B_{\sigma d_\Omega (x)}(x)$ where the last inclusion follows since 
\begin{equation*}
d_\Omega (x) \geq 2+ \frac{3}{\sigma} -1 \geq\frac{3}{\sigma}
\end{equation*}
using the assumption $B_{2+3 /\sigma }\subset \Omega$.
Furthermore, $x\in B_3(y)\subset   B_{\sigma d_\Omega (y)}(y)$ where we use that $d_\Omega (y) \geq 2+ \frac{3}{\sigma} -2 =\frac{3}{\sigma}$.
This proves the claim from above.

We continue by estimating $\| h \| _{L^\infty (B_1)}$.
For almost all $x\in B_1$, we have
\begin{align*}
|h(x)| & \leq \| u\| _{L^\infty (B_1)} \int _{B_2^c} \frac{1}{|x-y|^{d+2s}} \intd y +\int _{B_2^c} \frac{|u(y)|}{|x-y|^{d+2s}} \intd y \\
& \leq c \left\| u \right\| _{L^\infty _{2s-\epsilon} (B_1)} + c \int _{B_2^c} \frac{|u(y)|}{(1+|y|^{2s-\epsilon }) |y|^{d+\epsilon }} \intd y \\
& \leq c \left\| u \right\| _{L^\infty _{2s-\epsilon} (B_1)} +  c \left\| u \right\| _{L^\infty_{2s-\epsilon} (B_2^c)} \int _{B_2^c} \frac{1}{ |y|^{d+\epsilon }} \intd y \\
& \leq c \left\| u \right\| _{L^\infty _{2s-\epsilon} (\Omega )} 
\end{align*}
for some constant $c=c(d,s,\epsilon )>0$.
Using interior regularity results for the fractional Laplacian (see, e.g., \cite[Theorem 2.4.3]{RosOton2024Book}), we conclude
\begin{align*}
\| u \| _{C^{2s}(\overline{ B_{1/2}})} & \leq c \left( \| u\| _{L^\infty _{2s-\epsilon} (\Omega )} + \| f +h\|_{L^\infty (B_1)} \right) \\
& \leq c \left( \| u\| _{L^\infty _{2s-\epsilon} (\Omega )} + \| f \|_{L^\infty (B_1)} \right) 
\end{align*}
for some $c=c(d,s,\epsilon )>0$.
\end{proof}
By scaling (see \autoref{Lem:ScalingOperator}), we get the following interior regularity results on balls $B_R$ with radius $R\leq 1$.
The assumption $R\leq 1$ is used to get $\| u_R\| _{L^\infty _{2s-\epsilon}} \leq \| u\| _{L^\infty _{2s-\epsilon}}$ for the rescaled function $u_R(x):=u(Rx)$.
\begin{Corollary}\label{Cor:ScaledDownInteriorRegularity}
Fix $s\in (1/2,1)$, $\sigma \in (0,1]$, $R\leq 1$ and some $\epsilon \in (0,2s)$.
Let $\Omega \subset \R^d$ be a $C^{1,1}$ domain with $ B_{R\left( 2+\frac{3}{\sigma}\right) } \subset \Omega$.
Assume that $u$ is a solution of
\begin{equation*}
\mathcal{L}_{\Omega , \sigma} u = f \qquad \text{in} \quad B_R
\end{equation*}
with $u\in L^\infty _{2s-\epsilon} (\Omega )$ and $f\in L^{\infty}(B_R)$.
Then $u\in C^{\gamma } (\overline{ B_{1/2}})$ for any $\gamma \in [0,2s]$ and
\begin{equation*}
[ u ] _{C^{\gamma }(\overline{ B_{R/2}})} \leq c R^{-\gamma } \left( \| u\| _{L^\infty _{2s-\epsilon} (\Omega )} + R^{2s} \| f\|_{L^\infty (B_R)} \right)
\end{equation*}
for some constant $c=c(s,d ,\epsilon ,\gamma )>0$.
\end{Corollary}

\subsection{Calculations in half-space}
We will finish this section with some explicit calculations in the half-space $\R^d_+$.
Throughout this paper, we will use the notation
\begin{equation}
x= (x' , x_d) \in \R^{d-1} \times \R_+ .
\end{equation}
Recall that the kernel $\mathcal{K}_{\R^d_+}(x,y) = \mathcal{B}(x,y) |x-y|^{-d-2s}$ is given by
\begin{equation}\label{eq:BKernelInHalfSpace}
\mathcal{B}(x,y) = \frac{\mathbbm{1} _{\{ y\in B_{\sigma x_d}(x) \} }  +\mathbbm{1} _{\{ x\in B_{\sigma y_d}(y) \} } }{2} .
\end{equation}
From the definition of $\mathcal{B}$, we directly get for all $x,y,z\in \R^d_+$
\begin{align}
\mathcal{B}(x,y) & = \mathcal{B}(y,x) , \label{eq:BSymmetry} \\
\mathcal{B}(x+(z',0),y + (z',0)) &= \mathcal{B}(x,y)  , \label{eq:ShiftingBInHalfSpace} \\
\mathcal{B}(ax,ay)& = \mathcal{B}(x,y) \qquad (a\in \R_+ )  \label{eq:ScalingBInHalfSpace} .
\end{align}
Now, fix some $x\in \R^d_+$ and let us explicitly determine the two domains, corresponding to the two indicator functions in \eqref{eq:BKernelInHalfSpace}, where $\mathcal{B}(x,\cdot )$ is nonvanishing.
These two domains are shown in \autoref{fig:SupportKernelSigma23} and \autoref{fig:SupportKernelSigma1} for the case $d=2$, $x=e_2=(0,1)$ with $\sigma =2/3$ and $\sigma =1$.

The first domain is just a ball around $x$ with radius $\sigma x_d$, which is shaded in blue in \autoref{fig:SupportKernelSigma23} and \autoref{fig:SupportKernelSigma1}.
The second domain is more interesting.
Notice, that $x\in B_{\sigma y_d}(y)$ if and only if $|x-y|^2=|x'-y'|^2 + (x_d-y_d)^2 < \sigma ^2 y_d^2$ .
This inequality describes different shapes in cases $\sigma \in (0,1)$ and $\sigma =1$.

For $\sigma \in (0,1)$, this inequality is equivalent to
\begin{equation} \label{eq:EllipsoidEquation}
\frac{\left( y_d -\frac{x_d}{1-\sigma ^2} \right) ^2}{ \frac{x_d^2}{1-\sigma ^2} \left( \frac{1}{1-\sigma ^2} -1 \right) } + \frac{|y' - x' | ^2}{x_d ^2 \left( \frac{1}{1-\sigma ^2} -1 \right) } < 1
\end{equation}
which describes an ellipsoid (compare with the red shaded domain in \autoref{fig:SupportKernelSigma23}). We will denote this ellipsoid by $E(x)$.

For $\sigma =1$, the above inequality is equivalent to
\begin{equation}\label{eq:ParaboloidEquation}
\frac{|y'-x'|^2}{2x_d} +\frac{x_d}{2} < y_d
\end{equation}
which describes the domain above a paraboloid (compare with the red shaded domain in \autoref{fig:SupportKernelSigma1}).
We will denote this domain by $P(x)$.

\begin{figure}
    \centering
    \begin{minipage}{0.48\textwidth}
        \centering
        
\begin{tikzpicture}
\begin{axis}[axis lines=middle,
			xmin=-3,xmax=3,
			ymin=0,ymax=5.5,
			grid=both,
			width=\textwidth,
            height=11/12*\textwidth,
            xlabel=$\R$,
            ylabel=$\R$,
			xtick={-2,-1,1,2},
            ytick={1,2,3,4,5}
            ]

\draw[pattern=north west lines, pattern color=blue] (0,1) circle[radius=2/3 ] ;
\draw[pattern=north east lines, pattern color=red] (0,9/5) ellipse (0.894 and 6/5);
\end{axis}
\end{tikzpicture}        
        
        \caption{Shows the support of $\mathcal{B}_{\R^2_+,\sigma }(e_2,\cdot )$ for $\sigma =2/3$.}
        \label{fig:SupportKernelSigma23}
    \end{minipage}\hfill
    \begin{minipage}{0.48\textwidth}
        \centering

\begin{tikzpicture}
\begin{axis}[axis lines=middle,
			xmin=-3,xmax=3,
			ymin=0,ymax=5.5,
			grid=major,
			width=\textwidth,
            height=11/12*\textwidth,
            xlabel=$\R$,
            ylabel=$\R$,
            xtick={-2,-1,1,2},
            ytick={1,2,3,4,5}
            ]

\draw[pattern=north west lines, pattern color=blue] (0,1) circle[radius=1 ] ;

\addplot [
	name path=P,
    domain=-3.2:3.2, 
    samples=100, 
    color=black,
]
{0.5+0.5* x^2};
 \path[name path=L] (axis cs:\pgfkeysvalueof{/pgfplots/xmin},5) -- (axis cs:\pgfkeysvalueof{/pgfplots/xmax},5);

    \addplot+[draw,pattern=north east lines,pattern color=red]
    fill between[
        of=P and L,
    ];
\end{axis}
\end{tikzpicture}

        \caption{Shows the support of $\mathcal{B}_{\R^2_+,\sigma }(e_2,\cdot )$ for $\sigma =1$.}
        \label{fig:SupportKernelSigma1}
    \end{minipage}
\end{figure}

Finally, we will compute the action of $\mathcal{L}_{\R^d_+}$ on monomials of the form $u(x)=x_d^p$.
We will follow the calculations from \cite[p. 120 f.]{Bogdan2003} to obtain the following result.
\begin{Proposition} \label{Prop:CalculationsInHalfspace} 
Let $s\in (1/2,1)$ and $\sigma \in (0,1]$.
Furthermore, let $u\colon \R^d_+ \to \R$ be defined by $u(x)=x_d^p =(d_{\R^d_+} (x))^p$ for some $p\in (-1,2s)$ when $\sigma =1$ and $p\in \R$ when $\sigma \in (0,1)$.
Then,
\begin{equation*}
\mathcal{L}_{\R^d_+} u (x)=  a(p,\sigma ) x_d^{p-2s}
\end{equation*}
for all $ x\in \R^d_+$ and some constant $a(p, \sigma)$ which satisfies 
\begin{equation} \label{eq:SignAPSigma}
\begin{aligned}
a(p,\sigma ) = 0 \qquad \text{for } &p=2s-1 ,\\
a(p,\sigma ) < 0 \qquad \text{for } &p>2s-1 \text{ or } p<0 , \\
a(p,\sigma ) > 0 \qquad \text{for } &0<p<2s-1 .
\end{aligned}
\end{equation}
\end{Proposition}
\begin{proof}
Let $x=(x',x_d)\in \R^{d-1} \times \R_+$.
By scaling and shifting (see Lemma \ref{Lem:ScalingOperator}), we have
\begin{equation*}
\mathcal{L}_{\R^d_+,\sigma } u (x) = x_d^{p-2s} \mathcal{L}_{\R^d_+,\sigma } u (e_d) =:a(p,\sigma ) x_d^{p-2s}
\end{equation*}
where $e_d=(0, \dots ,0 ,1)\in \R^d$. It remains to compute the sign of 
\begin{equation}\label{eq:APSigmaDecomposition}
\begin{aligned}
a(p,\sigma )& = \lim _{\epsilon \to 0} \int _{\R^d_+ \setminus B_\epsilon (e_d)} \frac{1-y_d^p}{\left((1-y_d)^2 + |y'|^2  \right) ^\frac{d+2s}{2}} \mathcal{B}(e_d,y) \intd y' \intd y_d \\
& = \lim _{\epsilon \to 0}\Bigg( \int _{\R^d_+ \setminus I_\epsilon } \frac{1-y_d^p}{\left((1-y_d)^2 + |y'|^2  \right) ^\frac{d+2s}{2}} \mathcal{B}(e_d,y) \intd y' \intd y_d \\
& \qquad  +  \int _{I_\epsilon \setminus B_\epsilon (e_d) } \frac{1-y_d^p}{\left((1-y_d)^2 + |y'|^2  \right) ^\frac{d+2s}{2}} \mathcal{B}(e_d,y) \intd y' \intd y_d \Bigg) \\
&=: \lim _{\epsilon \to 0} (J_{1,\epsilon} + J _{2,\epsilon} ).
\end{aligned}
\end{equation}
where $I_\epsilon = \R^{d-1}\times (1-\epsilon ,1+\epsilon)$.
By the change of variable $y'=|1-y_d| u$, we get
\begin{align*}
J_{1,\epsilon }&= \int _{\R^{d-1}} \frac{1}{\left( 1+ |u|^2  \right) ^\frac{d+2s}{2}} \int _{\R_+ \setminus (1-\epsilon ,1+ \epsilon )}  \frac{1-y_d^p}{|1-y_d|^{1+2s}}   \mathcal{B}(e_d,(|1-y_d| u,y_d))   \intd y_d \intd u \\
&=  \int _{\R^{d-1}} \frac{1}{\left( 1+ |u|^2  \right) ^\frac{d+2s}{2}} \tilde{J} _{1,\epsilon}  \intd u .
\end{align*}
We further decompose $\tilde{J} _{1,\epsilon}$ as
\begin{align*}
\tilde{J} _{1,\epsilon} & = \int _0^{1-\epsilon }  \frac{1-y_d^p}{|1-y_d|^{1+2s}}   \mathcal{B}(e_d,((1-y_d) u,y_d))   \intd y_d \\
& \qquad + \int _{1+\epsilon} ^{(1-\epsilon  ) ^{-1}}  \frac{1-y_d^p}{|1-y_d|^{1+2s}}   \mathcal{B}(e_d,((y_d-1) u,y_d))   \intd y_d \\
& \qquad + \int _{(1-\epsilon  ) ^{-1}} ^{ \infty }  \frac{1-y_d^p}{|1-y_d|^{1+2s}}   \mathcal{B}(e_d,((y_d-1) u,y_d))   \intd y_d \\
& = \tilde{J} _{1,1,\epsilon} +\tilde{J} _{1,2,\epsilon} +\tilde{J} _{1,3,\epsilon} 
\end{align*}
where we use $1+\epsilon <(1-\epsilon  ) ^{-1} $ for $\epsilon <1$.

We start with $\tilde{J} _{1,3,\epsilon}$. Using the transformation $t=1/y_d$, we obtain
\begin{align*}
\tilde{J} _{1,3,\epsilon} & = \int _0^{1-\epsilon} \frac{1-t^{-p}}{|1-1/t|^{1+2s}} t ^{-2} \mathcal{B}(e_d,((1/t-1)u,1/t))  \intd t \\
& = - \int _0^{1-\epsilon} \frac{1-t^p}{|1-t|^{1+2s}} t ^{2s-1-p} \mathcal{B}(e_d,((1/t-1)u,1/t))  \intd t \\
&= -  \int _0^{1-\epsilon} \frac{1-t^p}{|1-t|^{1+2s}} t ^{2s-1-p} \mathcal{B}(e_d,(-(1-t)u,t))  \intd t 
\end{align*}
where in the last step we have used \eqref{eq:ScalingBInHalfSpace}, \eqref{eq:ShiftingBInHalfSpace} and the symmetry of $\mathcal{B}$ to get
\begin{equation*}
\mathcal{B}(e_d,((1/t-1)u,1/t)) = \mathcal{B}(te_d,((1-t)u,1)) = \mathcal{B} ((-(1-t)u,t),e_d) =\mathcal{B} (e_d, (-(1-t)u,t) ).
\end{equation*}
For $\tilde{J} _{1,2,\epsilon}$, we use that $|1-y_d^p | \leq c \epsilon$ for all $y_d \in (1+\epsilon ,(1-\epsilon )^{-1})$ and the fact that $\mathcal{B}\leq 1$.
Furthermore, $(1-\epsilon )^{-1} -(1+\epsilon )\leq 2 \epsilon ^2 $ for $\epsilon <1/2$. Hence, we conclude
\begin{equation*}
|\tilde{J} _{1,2,\epsilon}| \leq 2 c \epsilon ^2 \epsilon \epsilon ^{-1-2s} \to 0 \qquad \text{as} \quad \epsilon \to 0.
\end{equation*}
Together, we obtain
\begin{align*}
J_{1,\epsilon } & =   \int _{\R^{d-1}} \frac{1}{\left( 1+ |u|^2  \right) ^\frac{d+2s}{2}} \left( \int _0^{1-\epsilon} \frac{(1-t^p)(1-t ^{2s-1-p})}{|1-t|^{1+2s}} \mathcal{B}(e_d,((1-t) u,t)) \intd t + \tilde{J} _{1,2,\epsilon} \right)  \intd u \\
& \to \int _0^1  \frac{(1-t^p)(1-t ^{2s-1-p})}{|1-t|^{1+2s}}   \left( \int _{\R^{d-1}} \frac{1}{\left( 1+ |u|^2  \right) ^\frac{d+2s}{2}} \mathcal{B}(e_d,((1-t) u,t)) \intd u \right) \intd t
\end{align*}
as $\epsilon \to 0$.
Note, that the last integral above is absolute convergent.

It remains to estimate $J _{2,\epsilon}$ from \eqref{eq:APSigmaDecomposition}.
Recall, that we can find $\eta = \eta (\sigma )>0$ and $\epsilon _0 =\epsilon _0(\sigma )>0$ such that $\mathcal{B}(e_d,y)=1$ for all $y\in \R^d_+$ satisfying $|y'|\leq \eta $ and $|1-y_d| \leq \epsilon$ for $\epsilon \leq \epsilon_0$.
Hence, for small $\epsilon$, we obtain 
\begin{equation}\label{eq:J2Epsilon}
\begin{aligned}
J _{2,\epsilon} & =  \int _{ (I_\epsilon \setminus B_\epsilon (e_d) ) \cap (B_\eta \times \R ) } \frac{1-y_d^p +p(y_d -1)}{\left((1-y_d)^2 + |y'|^2  \right) ^\frac{d+2s}{2}}  \intd y' \intd y_d  \\
& \qquad +  \int _{(I_\epsilon \setminus B_\epsilon (e_d) ) \cap (B_\eta ^c \times \R ) } \frac{1-y_d^p }{\left((1-y_d)^2 + |y'|^2  \right) ^\frac{d+2s}{2}} \mathcal{B}(e_d,y) \intd y' \intd y_d
\end{aligned}
\end{equation}
Note, that the absolute value of the first summand is bounded by
\begin{equation*}
 \int _{ I_\epsilon   \cap (B_\eta \times \R ) } \frac{| 1-y_d^p +p(y_d -1) |}{\left((1-y_d)^2 + |y'|^2  \right) ^\frac{d+2s}{2}}  \intd y' \intd y_d
\end{equation*}
which is finite and, hence, converges to $0$ as $\epsilon \to 0$.
It is also easy to see that the second summand in \eqref{eq:J2Epsilon} converges to $0$ since we are away from the singularity.
Hence, $J _{2,\epsilon} \to 0$ as $\epsilon \to 0$.
All together, we obtain
\begin{equation*}
a(p,\sigma )= \int _0^1  \frac{(1-t^p)(1-t ^{2s-1-p})}{|1-t|^{1+2s}}   \left( \int _{\R^{d-1}} \frac{1}{\left( 1+ |u|^2  \right) ^\frac{d+2s}{2}} \mathcal{B}(e_d,((1-t) u,t)) \intd u \right) \intd t .
\end{equation*}
Finally, \eqref{eq:SignAPSigma} follows, by evaluating the sign of $(1-t^p)(1-t ^{2s-1-p})$ for $t\in (0,1)$.
\end{proof}
\begin{Remark}\label{Rem:GeneralizeComputationsInHalfspace}
Note that in the proof of \autoref{Prop:CalculationsInHalfspace}, we have not used the explicit form of the kernel \smash{$\mathcal{B}$}.
We only used the properties \eqref{eq:BSymmetry}, \eqref{eq:ShiftingBInHalfSpace}, \eqref{eq:ScalingBInHalfSpace} (symmetry, translational invariance in the tangential direction $\R^{d-1}$ and scaling of $\mathcal{B}$) as well as nice behavior of the kernel close to the diagonal in order to evaluate the principle value integral in a pointwise sense.

We also want to mention that if we replace the scaling property $\mathcal{B}(ax,ay)=\mathcal{B}(x,y)$ from \eqref{eq:ScalingBInHalfSpace} by the assumption
\begin{equation*}
\mathcal{B}(ax,ay) = a^\gamma \mathcal{B}(x,y) ,
\end{equation*}
then one can still determine harmonic functions in the half-line, following the above calculations.
Note that the operator from \eqref{eq:OperatorScottDu} has such a scaling property.
Indeed, one can calculate that the function $\varphi (x)=x$ is a harmonic solution in the half-line for the operator $L$ from \eqref{eq:OperatorScottDu}.
\end{Remark}

\section{Barriers}\label{sec:Barriers}
For this section, we will fix a $C^{1,1}$ domain $\Omega \subset \R^d$ and prove pointwise estimates on $\mathcal{L}_{\Omega , \sigma} \varphi$ where $\varphi$ is a barrier function built form $d_\Omega$.
Note that the distance function $d_\Omega$ is not smooth on $\Omega$.
However, (see, e.g., \cite[Lem 14.16]{Gilbarg2001}) $d_\Omega$ is $C^{1,1}$ close to the boundary.
Hence, since all our estimates are close to the boundary, the pointwise evaluation of $\mathcal{L}_{\Omega , \sigma} \varphi$ will not cause problems.

The main result of this section is \autoref{Prop:Barrier}.
For the proof, we will follow the construction of \cite[Appendix B.2]{RosOton2024Book}.
However, herein, difficulties arise since the operator $ \mathcal{L}_{\Omega ,\sigma}$ is not translational invariant and depends on the domain $\Omega$.

\begin{Proposition}\label{Prop:Barrier}
Fix $s\in (1/2,1)$ and $\sigma \in (0,1]$.
Let $\Omega \subset \R^d$ be a $C^{1,1}$ domain, such that for every $z\in \partial \Omega \cap B_1$ we can find two balls of radius $r_0>0$ in $\Omega$ and $\Omega ^c$ which both touch $\partial \Omega$ in $z$.
Then for every $\epsilon>0$ there exists $\tilde{\varphi } \in H^s_{loc}(\R^d )$ and constants $C=C(r_0,d ,s, \sigma , \epsilon )>1$, $\eta =\eta (r_0,d  ,s, \sigma ,\epsilon )>0$ such that
\begin{alignat}{2}
\mathcal{L}_{\Omega , \sigma} \tilde{\varphi } &\geq d_\Omega ^{-1+\epsilon} \qquad & &\text{in} \quad \{ 0 < d_\Omega (x) < \eta \} \cap B_{4/5} \label{eq:BarrierLPhiBound} ,\\
d_\Omega ^{2s-1} \leq \tilde{\varphi } &\leq C d _\Omega^{2s-1}  & &\text{in} \quad B_1  \label{eq:BarrierUpperLowerBound} , \\
\tilde{\varphi } &= 0 & &\text{in} \quad \R^d\setminus B_2 .
\end{alignat} 
\end{Proposition}
\begin{proof}
W.l.o.g. assume that $\epsilon < \min \left\{ 2s-1,1/2,2-2s \right\}$ for the case $\sigma \in (0,1)$ and $\epsilon < \min \left\{ 2s-1 , \mu ,2-2s  \right\}$ with $\mu >0$ from \autoref{Lem:ConvergenceToParabola} when $\sigma =1$.
We set
\begin{equation*}
\varphi := C d _\Omega^{2s-1} - d _\Omega ^{2s-1+\epsilon' } 
\end{equation*}
for $\epsilon '= \epsilon /2$ and some $C>1$ so that \eqref{eq:BarrierUpperLowerBound} holds true.
Using \cite[Lemma B.2.5]{RosOton2024Book}, we can verify that $\varphi \in H^s_{loc}(\R^d )$.
We first consider the case $\sigma \in (0,1)$.
With \autoref{Lem:BoundLd2sMinus1} and \autoref{Lem:BoundLd2sMinus1PlusEpsilon}, we obtain
\begin{equation}\label{eq:BarrierBound1}
\mathcal{L}_\Omega  \varphi (x) \geq - c_1 d_\Omega ^{\min \left\{ 2s-2,-\frac{1}{2} \right\}}  (x) + c_2 d_\Omega ^{-1+\epsilon ' } (x) \geq \frac{c_2}{2} d_\Omega ^{-1+\epsilon ' } (x)
\end{equation}
for all $x \in \Omega \cap B_{4/5}$ with $d_\Omega (x) \leq \eta '$ for some constant $\eta '>0$.
For the last inequality in \eqref{eq:BarrierBound1}, we have used that $-1+\epsilon ' < \min \left\{ 2s-2,-1/2 \right\} $.

For $\sigma =1$, we also use \autoref{Lem:BoundLd2sMinus1} and \autoref{Lem:BoundLd2sMinus1PlusEpsilon} and obtain
\begin{equation} \label{eq:BarrierBound2}
\mathcal{L}_\Omega  \varphi (x) \geq - c_1 d_\Omega ^{\min \left\{ 2s-2, \mu -1 \right\}}  (x) + c_2 d_\Omega ^{-1+\epsilon ' } (x) \geq  \frac{c_2}{2} d_\Omega ^{-1+\epsilon ' } (x)
\end{equation}
for all $x \in \Omega \cap B_{4/5}$ close to the boundary $\partial \Omega$.

We will choose a cutoff function $\chi \in C^\infty _c(\R^d)$ with $\ind _{B_1} \leq \chi \leq \ind _{B_2}$ and define $\tilde{\varphi} := \chi \varphi$.
Then,
\begin{equation*}
\mathcal{L}_\Omega \tilde{\varphi} (x) = \mathcal{L}_\Omega \varphi (x) + \mathcal{L}_\Omega (\varphi (\chi -1))(x)  .
\end{equation*}
Furthermore, for all $x\in \Omega \cap B_{4/5}$, we have
\begin{align*}
|\mathcal{L}_\Omega (\varphi (\chi -1))(x) | & \leq \int _\Omega \frac{|\varphi (y) | (1- \chi (y))}{|x-y|^{d+2s}} \mathcal{B}(x,y) \intd y \leq c
\end{align*}
since we are away from the singularity.
Together with \eqref{eq:BarrierBound1} and \eqref{eq:BarrierBound2}, we conclude that \eqref{eq:BarrierLPhiBound} holds true close to the boundary $\partial \Omega$.
\end{proof}
\begin{Lemma} \label{Lem:BoundLd2sMinus1}
Fix $s\in (1/2,1)$ and $\sigma \in (0,1]$. Let $\Omega \subset \R^d$ be a $C^{1,1}$ domain satisfying the exterior and interior ball condition with radius $r_0$ from \autoref{Prop:Barrier}.
Then there exists some $0<\rho _0 = \rho _0 (r_0 ,s ,d, \sigma )$ such that for all $x_0\in \Omega \cap B_{4/5}$ with $d_\Omega (x_0) \leq \rho _0$ we have
\begin{equation}\label{eq:BoundLD2s-1}
\left| \mathcal{L}_{\Omega , \sigma} d _\Omega^{2s-1}  (x_0)\right| \leq c d_\Omega ^{ -\gamma } (x_0)
\end{equation}
for some constant $c=c(r_0 ,s,d, \sigma )>0$, where $\gamma = \min \left\{ 2-2s, \frac{1}{2} \right\}$ for $\sigma \in (0,1)$ and $\gamma = \min \left\{ 2-2s, 1-\mu  \right\}$ for $\sigma =1$ with $\mu = \mu (d,s)\in (0,1)$ from \autoref{Lem:ConvergenceToParabola}.
\end{Lemma}
\begin{proof}
Let $x_0 \in \Omega \cap B_{4/5}$ and set $\rho := d_\Omega (x_0)$.
We will prove \eqref{eq:BoundLD2s-1} assuming that $\rho \leq \rho _0$ for some fixed constant $\rho _0>0$, which we will specify later.
After rotation and shifting we may assume that $x_0=(0, \dots, 0,\rho) \in \R^d$ and $0\in \R^d$ is the nearest boundary point and furthermore $B_{r_0} (0, \dots ,0 , r_0) \subset \Omega$ and $B_{r_0} (0, \dots ,0 , - r_0) \subset \Omega ^c$.

We define $l\colon \R^d \to \R$ by $l(y)=(y_d)_+$ where we use the notation $y=(y',y_d)\in \R^{d-1} \times \R$.
Notice that $l(x_0) =d_\Omega (x_0)$ and $\nabla l (x_0)= \nabla d_\Omega (x_0)$ for $\rho_0$ small enough (choose $\rho _0$ small enough such that $d_\Omega$ is $C^{1,1}$ at $x_0$).
We will write
\begin{equation}\label{eq:LDistWithLinear}
\mathcal{L}_\Omega d_\Omega^{2s-1} (x_0) = \mathcal{L}_\Omega ( d_\Omega^{2s-1} -l^{2s-1}) (x_0) + \mathcal{L}_\Omega l^{2s-1} (x_0) 
\end{equation}
and start by estimating $|\mathcal{L}_\Omega ( d_\Omega ^{2s-1} -l^{2s-1}) (x_0)|$.

Since $|a^{2s-1}-b^{2s-1}|\leq |a-b| (a^{2s-2}+b^{2s-2})$, we conclude
\begin{equation}\label{eq:DMinusLPower2sMinus1Estimate}
|d_\Omega ^{2s-1} (y) -l^{2s-1} (y)|  \leq |d_\Omega (y) -l (y)| \left( d_\Omega ^{2s-2} (y) +l^{2s-2} (y) \right)
\end{equation}
for all $y\in \Omega$.
Using that $l(x_0)=d_\Omega (x_0)$, $\nabla l (x_0) = \nabla d_\Omega (x_0)$ and the fact that $d_\Omega$ is $C^{1,1}$ at $x_0$ we get
\begin{equation}\label{eq:dMinusLTaylor}
|d_\Omega (y) -l (y)|  \leq c |y-x_0|^2
\end{equation}
for all $y\in \Omega$.
Combining \eqref{eq:DMinusLPower2sMinus1Estimate} and \eqref{eq:dMinusLTaylor} yields
\begin{equation*}
|d_\Omega ^{2s-1} (y) -l^{2s-1} (y)|  \leq c |y-x_0|^2 \left( d_\Omega ^{2s-2} (y) +l^{2s-2} (y) \right) \leq c |y-x_0|^2 \rho ^{2s-2}
\end{equation*}
for all $y\in B_{\rho /2} (x_0)$.
Since $|a^{2s-1}-b^{2s-1}|\leq |a-b|^{2s-1}$, we also get
\begin{equation*}
|d_\Omega ^{2s-1} (y) -l^{2s-1} (y)| \leq  |d_\Omega  (y) -l (y)|  ^{2s-1} \leq c |y-x_0|^{2(2s-1)}
\end{equation*}
for all $y\in \Omega \setminus B_{\rho /2} (x_0)$.
All together, we obtain
\begin{align*}
|\mathcal{L}_\Omega ( d_\Omega ^{2s-1} -l^{2s-1}) (x_0)| & \leq \pv \int _\Omega \frac{|d_\Omega ^{2s-1} (y) -l^{2s-1} (y)|}{|y-x_0|^{d+2s}} \intd y \\
& \leq c \int _{B_{\rho /2} (x_0)}  \frac{|y-x_0|^2 \rho ^{2s-2}}{|y-x_0|^{d+2s}} \intd y \\
& \qquad + c \int _{\Omega \setminus B_{\rho /2} (x_0)} \frac{  |y-x_0|^{2(2s-1)}  }{|y-x_0|^{d+2s}} \intd y \\
& \leq c \rho ^{2s-2} \left( \frac{\rho}{2} \right) ^{2-2s} + c (1+ \rho ^{2s-2} ) \\
& \leq c \rho ^{2s-2} .
\end{align*}

Next, we investigate the second summand $|\mathcal{L}_\Omega l^{2s-1} (x_0)|$ of \eqref{eq:LDistWithLinear}.
For $\sigma \in (0,1)$, we use \autoref{Lem:ConvergenceToEllipse} (recall that $a(2s-1,\sigma )=0$) to get
\begin{equation*}
|\mathcal{L}_\Omega l^{2s-1} (x_0)| \leq c \rho ^{-\frac{1}{2} } .
\end{equation*}
For $\sigma =1$, we use \autoref{Lem:ConvergenceToParabola} and get
\begin{equation*}
|\mathcal{L}_\Omega l^{2s-1} (x_0)| \leq c \rho ^{-1+\mu } 
\end{equation*}
for some $\mu >0$.
This finishes the proof of the Lemma.
\end{proof}

\begin{Lemma}\label{Lem:ConvergenceToEllipse}
Fix $\sigma \in (0,1)$ and $p \in \R$.
Let $\Omega \subset \R^d$ be a $C^{1,1}$ domain and $x_0\in \Omega$ be a point such that $x_0=(0, \dots , 0, \rho )$ for some $\rho >0$ such that $0\in \partial \Omega$ is the nearest boundary point of $x_0$, i.e. $d_\Omega (x_0)=\rho$.
Furthermore, assume that $B_{r_0} (0, \dots ,0 , r_0) \subset \Omega$ and $B_{r_0} (0, \dots ,0 , - r_0) \subset \Omega ^c$ for some fixed $r_0>0$. 
Let $l\colon \R^d \to \R$ be given by $l(y)=y_d$.

Then, there exists some $0<\rho_0 =\rho _0 (d,s,r_0, \sigma )$ such that whenever $0< \rho \leq \rho_0$, we have
\begin{equation}
\mathcal{L}_\Omega l^p (x_0)= \rho ^{p -2s} \left( a(p ,\sigma) + R \right)
\end{equation}
where $a(p, \sigma )$ is the constant from \autoref{Prop:CalculationsInHalfspace} and
\begin{equation}
|R| \leq c \rho^{\frac{1}{2}} .
\end{equation}
for some constant $c=c(d,s,r_0, \sigma , p)>0$.
\end{Lemma}
\begin{proof}
By a change of variables, we have
\begin{align*}
\mathcal{L}_\Omega l^{p} (x_0) & = \pv \int _\Omega (l^{p}(x_0)-l^{p}(y)) \mathcal{K}_\Omega (x,y) \intd y \\
& = \rho ^{p-2s} \left( \frac{1}{2}\pv \int _{B_{\sigma } (e_d )} \frac{1-y_d^{p}}{|e_d-y|^{d+2s}} \intd y  + 
 \frac{1}{2} \pv \int _{E_{\rho ^{-1}\Omega}}  \frac{1-y_d^{p}}{|e_d-y|^{d+2s}} \intd y  \right)
\end{align*}
where $e_d=(0, \dots ,0, 1) \in \R^d$ and
\begin{equation*}
E_{\rho ^{-1}\Omega} = \left\{ y\in \R^d \mid |e_d -y | < \sigma d_{\rho ^{-1} \Omega} (y) \right\} .
\end{equation*}
using the notation $\rho ^{-1} \Omega = \{ \rho^{-1} y \colon y \in \Omega \}$.
Recall from \autoref{Prop:CalculationsInHalfspace}, that \smash{$\mathcal{L}_{\R^d_+  } l^{p} (e_d) =a(p,\sigma )$}.
We have
\begin{align*}
 \frac{1}{2}\pv & \int _{B_{\sigma } (e_d )} \frac{1-y_d^{p}}{|e_d-y|^{d+2s}} \intd y  +  
 \frac{1}{2} \pv \int _{E_{\rho ^{-1}\Omega}}  \frac{1-y_d^{p}}{|e_d-y|^{d+2s}} \intd y \\
 & = \mathcal{L}_{\R^d_+} l^{p} (e_d) 
  +\frac{1}{2} \pv \int _{E_{\rho ^{-1}\Omega} \setminus E_{\R^d_+}}  \frac{1-y_d^{p}}{|e_d-y|^{d+2s}} \intd y
   - \frac{1}{2} \pv \int _{  E_{\R^d_+} \setminus E_{\rho ^{-1}\Omega}}  \frac{1-y_d^{p}}{|e_d-y|^{d+2s}} \intd y \\
   & = a(p,\sigma ) +R
\end{align*}
where
\begin{equation*}
E_{\R^d _+} = \left\{ y\in \R^d \mid |y-e_d| < \sigma y_d  \right\} 
\end{equation*}
is the Ellipsoid from \eqref{eq:EllipsoidEquation} with $x=e_d$ (compare with \autoref{fig:SupportKernelSigma23}).
It remains to estimate the term $R$.

\textbf{Claim:} There exists some $c_1=c_1(d,s,\sigma , r_0)>0$, such that if we set $f(\rho ):=c_1\rho^{\frac{1}{2}}$, the symmetric difference $E_{\R^d_+} \Delta E_{\rho ^{-1}\Omega}$ will be contained in the domain we obtain by thickening up the boundary of the Ellipsoid \smash{$\partial E_{\R^d _+}$} in the following sense
\begin{equation}\label{eq:InlcusionSymmetricDiffernce}
 E_{\R^d_+} \Delta E_{\rho ^{-1}\Omega}  \subset \left\{ y+ z  \mid y\in \partial E_{\R^d _+} \quad \text{and} \quad z=(z',0) \text{ with } z'\in \R^{d-1} , |z'| \leq f( \rho) \right\} .
\end{equation}

Assume we have already shown the claim \eqref{eq:InlcusionSymmetricDiffernce}.
Then \smash{$\left| E_{\R^d_+} \Delta E_{\rho ^{-1}\Omega} \right|\leq c \rho ^{\frac{1}{2}}$}. Hence, we can estimate $|R|$ by 
\begin{align*}
|R|     \leq \left| E_{\R^d_+} \Delta E_{\rho ^{-1}\Omega} \right| \left\|  \frac{1-y_d^{p}}{|e_d-y|^{d+2s}} \right\| _{L^\infty ( E_{\R^d_+} \Delta E_{\rho ^{-1}\Omega} )}  \leq c \rho ^{\frac{1}{2}}
\end{align*}
for some constant $c=c(d,s,\sigma ,r_0, p)>0$.

It remains to prove the claim \eqref{eq:InlcusionSymmetricDiffernce}.
Note that we have 
\begin{equation*}
y=(y',y_d) \in E_{\R^d_+} \quad \Leftrightarrow \quad  y_d \in ((1+\sigma) ^{-1},(1-\sigma )^{-1} ) \quad \text{and}\quad |y'| < \sqrt{\sigma ^2 y_d^2 - (y_d -1)^2}   .
\end{equation*}
Using this characterization of $E_{\R^d_+}$, we will prove the claim \eqref{eq:InlcusionSymmetricDiffernce} by showing the two implications \eqref{eq:OneInclusionStripAroundEllipse} and \eqref{eq:SecondInclusionStripAroundEllipse}. 
First, we will show that
\begin{equation}\label{eq:OneInclusionStripAroundEllipse}
y_d \in ((1+\sigma) ^{-1},(1-\sigma )^{-1} )  \text{ and } |y'| < \sqrt{\sigma ^2 y_d^2 - (y_d -1)^2 } -f(\rho )  \quad \Rightarrow \quad (y',y_d) \in E_{\rho ^{-1}\Omega } .
\end{equation}
Recall, that $B_{r_0} (0, \dots ,0 , r_0) \subset \Omega$, i.e. $B_{r_0\rho ^{-1}} (0, \dots , 0, r_0 \rho ^{-1}) \subset \rho ^{-1} \Omega$.
Hence, we can estimate the distance to boundary by $d_\Omega (y',y_d) \geq r_0\rho^{-1} - \sqrt{(y_d-r_0\rho^{-1} )^2 + |y'|^2} $ for points $(y',y_d)\in B_{r_0\rho ^{-1}} (0, \dots , 0, r_0 \rho ^{-1})$.
By choosing $\rho _0$ small, we see that elements $(y',y_d)$ satisfying \eqref{eq:AssumptionForFirstInclusionEllipse} have the property $(y',y_d)\in B_{r_0\rho ^{-1}} (0, \dots , 0, r_0 \rho ^{-1})$.
So to get \eqref{eq:OneInclusionStripAroundEllipse}, it is enough to prove, that if 
\begin{equation}\label{eq:AssumptionForFirstInclusionEllipse}
y_d \in ((1+\sigma) ^{-1},(1-\sigma )^{-1} )  \quad \text{and}\quad  |y'| < \sqrt{\sigma ^2 y_d^2 - (y_d -1)^2 } -f(\rho ),
\end{equation}
then we have
\begin{equation}\label{eq:FirstInclusionEllipseWithDistanceToUpperBall}
 \sqrt{(y_d-1)^2+|y'|^2} <\sigma \left( r_0\rho^{-1} - \sqrt{(y_d-r_0\rho^{-1} )^2 + |y'|^2} \right) .
\end{equation}
By squaring \eqref{eq:FirstInclusionEllipseWithDistanceToUpperBall}, we see that \eqref{eq:FirstInclusionEllipseWithDistanceToUpperBall} is equivalent to
\begin{gather*}
\begin{aligned}
(y_d-1)^2+|y'|^2 & \leq \sigma ^2 \left( \frac{r_0^2}{\rho^2}+ \left( y_d-\frac{r_0}{\rho} \right) ^2 +|y'|^2 -2\frac{r_0}{\rho} \sqrt{\left( y_d- \frac{r_0}{\rho} \right)^2 +|y'|^2}  \right) \\
& =  \sigma ^2 y_d ^2 + \sigma ^2 \rho \left(  \frac{2 r_0^2}{\rho ^3} - \frac{2y_d r_0}{\rho ^2} + \frac{|y'|^2}{\rho} -\frac{2r_0}{\rho ^2} \sqrt{\left( y_d-\frac{r_0}{\rho} \right) ^2 + |y'|^2 }\right) \\
& = \sigma ^2 y_d ^2 + \sigma ^2 \rho \frac{\left( \frac{2 r_0^2}{\rho ^3} - \frac{2y_d r_0}{\rho ^2} + \frac{|y'|^2}{\rho} \right) ^2 -\frac{4r_0^2}{\rho ^4} \left( \left( y_d-\frac{r_0}{\rho} \right) ^2 + |y'|^2 \right)   }{\frac{2 r_0^2}{\rho ^3} - \frac{2y_d r_0}{\rho ^2} + \frac{|y'|^2}{\rho} +\frac{2r_0}{\rho ^2} \sqrt{\left( y_d-\frac{r_0}{\rho} \right) ^2 + |y'|^2 }} \\
& = \sigma ^2 y_d ^2 +\sigma ^2 \rho \left( \frac{-4y_dr_0|y'|^2 + \rho |y'|^4 }{2r_0^2-2\rho y_d r_0 + \rho ^2|y'|^2 + 2r_0\sqrt{r_0^2-2\rho r_0y_d + \rho ^2 y_d^2+ \rho |y'|^2}} \right) \\
& = :\sigma ^2 y_d ^2 +\sigma ^2 \rho A .
\end{aligned}
\end{gather*}
Notice, that if $(y,y_d)$ satisfies \eqref{eq:AssumptionForFirstInclusionEllipse}, both $|y'|$ and $|y_d|$ are bounded and hence, we can choose $c_1$ large enough and $\rho _0$ small enough such that $\sigma ^2 \rho |A| \leq c_1^2 \rho =(f(\rho ))^2$.
So let us assume that $(y,y_d)$ satisfies \eqref{eq:AssumptionForFirstInclusionEllipse}. Then, by using this assumption in the first two steps, we obtain
\begin{align*}
|y'|^2+(y_d-1)^2 & < \sigma ^2y_d^2 - 2 f(\rho )\sqrt{\sigma ^2y_d^2 - (y_d-1)^2} +(f(\rho ))^2 \\
& \leq \sigma ^2y_d^2 - 2 f(\rho )\left( |y'| + f(\rho ) \right) +(f(\rho ))^2 \\
& = \sigma ^2y_d^2 - 2 f(\rho ) |y'|  -(f(\rho ))^2 \\
& \leq  \sigma ^2y_d^2   -(f(\rho ))^2  \\
& \leq \sigma ^2 y_d ^2 +\sigma ^2 \rho A 
\end{align*}
which proves \eqref{eq:FirstInclusionEllipseWithDistanceToUpperBall} and therefore also the implication \eqref{eq:OneInclusionStripAroundEllipse}

In the second step, we will show 
\begin{equation}\label{eq:SecondInclusionStripAroundEllipse}
 (y',y_d) \in E_{\rho ^{-1}\Omega } \quad \Rightarrow \quad   y_d \in ((1+\sigma) ^{-1},(1-\sigma )^{-1} )  \text{ and } |y'| < \sqrt{\sigma ^2 y_d^2 - (y_d -1)^2 } + f(\rho )  .
\end{equation}
First notice that if $y_d \notin ((1+\sigma) ^{-1},(1-\sigma )^{-1} )$ then, we have $(y',y_d) \notin E_{\rho ^{-1}\Omega }$.
Indeed, notice, that $y_d \notin ((1+\sigma) ^{-1},(1-\sigma )^{-1} )$ implies $\sigma ^2 y_d^2 - (y_d -1)^2 \leq 0$ which gives us
\begin{equation*}
 \left( \sigma d_{\rho ^{-1}\Omega}(y) \right) ^2 \leq \sigma^2 \left( |y'|^2 +y_d^2 \right) \leq \sigma^2|y'|^2 + (y_d -1)^2 \leq  |e_d -y| ^2  ,
\end{equation*}
i.e. $(y',y_d) \notin E_{\rho ^{-1}\Omega }$.

Furthermore, if $|y|>(1-\sigma )^{-1}$, then we have $y\notin E_{\rho ^{-1}\Omega }$ since
\begin{equation*}
\sigma d_{\rho ^{-1}\Omega}(y)\leq \sigma |y| < |y| -1  \leq |e_d-y| .
\end{equation*}
So to show \eqref{eq:SecondInclusionStripAroundEllipse}, it is enough to prove, that if
\begin{equation}\label{eq:ConditionSecondInclusion}
|y|\leq\frac{1}{1-\sigma}, \text{ and }y_d \in ((1+\sigma) ^{-1},(1-\sigma )^{-1} ), \text{ and }|y'| \geq \sqrt{\sigma ^2 y_d^2 - (y_d -1)^2 } + f(\rho )
\end{equation}
then $y\notin E_{\rho ^{-1}\Omega }$.

Recall that $B_{r_0} (0, \dots ,0 , -r_0) \subset \Omega ^c$, i.e. $B_{r_0\rho ^{-1}} (0, \dots , 0, -r_0 \rho ^{-1}) \subset  (\rho ^{-1} \Omega )^c$.
Hence, we can estimate the distance to boundary by $d_\Omega (y',y_d) \leq  \sqrt{(y_d+r_0\rho^{-1} )^2 + |y'|^2} - r_0\rho^{-1} $.
So it is enough to prove that if \eqref{eq:ConditionSecondInclusion} is true, then we have
\begin{equation}\label{eq:SecondInclusionEllipseWithDistanceToLowerBall}
 \sqrt{(y_d-1)^2+|y'|^2} > \sigma \left( \sqrt{(y_d+r_0\rho^{-1} )^2 + |y'|^2} -r_0\rho^{-1}  \right) .
\end{equation}
Notice that, as above, this is equivalent to
\begin{align*}
(y_d-1)^2+|y'|^2 & >  \sigma ^2 y_d ^2 +\sigma ^2 \rho \left( \frac{4y_dr_0|y'|^2 + \rho |y'|^4 }{2r_0^2+2\rho y_d r_0 + \rho ^2|y'|^2 + 2r_0\sqrt{r_0^2+2\rho r_0y_d + \rho ^2 y_d^2+ \rho |y'|^2}} \right) \\
& = \sigma ^2 y_d ^2 +\sigma ^2 \rho B
\end{align*}
where again since $|y|\leq\frac{1}{1-\sigma}$ we can choose $c_1$ large enough and $\rho _0$ small enough so that $\sigma ^2 \rho |B| \leq c_1^2\rho =(f(\rho ))^2$.
So assuming \eqref{eq:ConditionSecondInclusion}, we get
\begin{align*}
|y'|^2+(y_d-1)^2 & > \sigma ^2y_d^2 + 2 f(\rho )\sqrt{\sigma ^2y_d^2 - (y_d-1)^2} +(f(\rho ))^2 \\
& \geq \sigma ^2y_d^2 + (f(\rho ))^2 \\
& \geq \sigma ^2 y_d ^2 +\sigma ^2 \rho B .
\end{align*}
This finishes the proof of \eqref{eq:SecondInclusionStripAroundEllipse}.
Note that \eqref{eq:OneInclusionStripAroundEllipse} and \eqref{eq:SecondInclusionStripAroundEllipse} together prove \eqref{eq:InlcusionSymmetricDiffernce} which finishes the proof of this lemma.
\end{proof}

\begin{Lemma}\label{Lem:ConvergenceToParabola}
Fix $\sigma =1$ and $p \in (0,2s )$.
Let $\Omega \subset \R^d$ be a $C^{1,1}$ domain and $x_0\in \Omega$ be a point such that $x_0=(0, \dots , 0, \rho )$ for some small $\rho >0$ and $0\in \partial \Omega$ as the nearest boundary point of $x_0$, i.e. $d_\Omega (x_0)=\rho$.
Furthermore, assume that $B_{r_0} (0, \dots ,0 , r_0) \subset \Omega$ and $B_{r_0} (0, \dots ,0 , - r_0) \subset \Omega ^c$ for some fixed $r_0>0$. 
Let $l\colon \R^d \to \R$ be given by $l(y)=y_d$.

Then there exists some $0<\rho_0 =\rho _0 (d,s,r_0 )$ such that whenever $0< \rho \leq \rho_0$, we have
\begin{equation}
\mathcal{L}_\Omega l^p (x_0)= \rho ^{p -2s} \left( a(p ,1) + R \right)
\end{equation}
where $a(p, \sigma )$ is the constant from Proposition \ref{Prop:CalculationsInHalfspace} and
\begin{equation}
|R| \leq c \rho^{\mu } .
\end{equation}
for some constant $c=c(d,s ,r_0 ,p)>0$ and $\mu = \mu (d,s,p)\in (0,1)$.
\end{Lemma}
\begin{proof}
As in the proof of \autoref{Lem:ConvergenceToEllipse}, we have
\begin{align*}
\mathcal{L}_\Omega l^{p} (x_0) & = \rho ^{p-2s} \Bigg( a(p,1) +  \frac{1}{2} \pv \int _{P_{\rho ^{-1}\Omega} \setminus P_{\R^d_+}}  \frac{1-y_d^{p}}{|e_d-y|^{d+2s}} \intd y \\
& \qquad \qquad \qquad \qquad  - \frac{1}{2} \pv \int _{  P_{\R^d_+} \setminus P_{\rho ^{-1}\Omega}}  \frac{1-y_d^{p}}{|e_d-y|^{d+2s}} \intd y \Bigg) \\
   & = \rho ^{p-2s} \left( a(p,1) + R \right)
\end{align*}
where
\begin{equation*}
P_{\rho ^{-1}\Omega} = \left\{ y\in \R^d \mid |e_d -y | <  d_{\rho ^{-1} \Omega} (y) \right\} .
\end{equation*}
and
\begin{equation*}
P_{\R^d _+} = \left\{ y\in \R^d \mid |y-e_d| <  y_d  \right\} 
\end{equation*}
is the domain above the paraboloid from \eqref{eq:ParaboloidEquation} with $x=e_d$ (compare with \autoref{fig:SupportKernelSigma1}).
It remains to estimate the error term $R$.
Note that we have 
\begin{equation*}
y=(y',y_d) \in P_{\R^d_+} \quad \Leftrightarrow \quad y_d\geq \frac{1}{2} \text{ and } |y'| <  \sqrt{2y_d-1 } .
\end{equation*}
We want to continue as in \autoref{Lem:ConvergenceToEllipse} and show that the symmetric difference $P_{\R^d_+} \Delta P_{\rho ^{-1}\Omega}$ is contained in some strip around the paraboloid $\partial P_{\R^d_+}$ which is given by $|y'|= \sqrt{2y_d-1}$.
However, notice, that in the proof of \autoref{Lem:ConvergenceToEllipse} for the estimates $\sigma ^2\rho |A| \leq c_1 \rho$ and $\sigma ^2\rho |B| \leq c_1 \rho$, we have used that $|y|$ is bounded.
We can replace this condition by the assumption that $\rho |y|^4$ converges to $0$ uniformly in $y$, which holds if we assume $y\in B_{\rho ^{-1 / \kappa}}$ for some $\kappa >4$.

Now, fix some $\kappa > \max \{ d, 4 \} $.
Following the same arguments as in the proof of \autoref{Lem:ConvergenceToEllipse}, there exists some $c_1>0$, so that for $f(\rho):=c_1\rho ^{1/2}$, we have
\begin{equation}\label{eq:ParabolaStrip}
\begin{aligned}
B& _{\rho ^{-\frac{1}{\kappa}}} (e_d)  \cap  \left( P_{\R^d_+} \Delta P_{\rho ^{-1}\Omega } \right) \\
& \subset B_{\rho ^{-\frac{1}{\kappa}}} (e_d)  \cap \left \{ y+z \mid y\in \partial P_{\R^d _+} \quad \text{and} \quad z=(z',0) \text{ with } z'\in \R^{d-1} , |z'| \leq f( \rho)  \right\}
\end{aligned}
\end{equation}
for all $\rho \leq \rho _0$ for some small $\rho_0>0$.

Let us denote the set on the righthand side of \eqref{eq:ParabolaStrip} by $S$.
Then, we can estimate the measure of $S$ by
\begin{align*}
|S| & \leq c \int  _\frac{1}{2} ^{\rho ^{-\frac{1}{\kappa}}}   \left( \left( \sqrt{2z-1}+c_1 \sqrt{\rho}  \right) ^{d-1} -\left( \sqrt{2z-1}-c_1 \sqrt{\rho} \right) ^{d-1} \right) \intd z \\
& \leq c \rho ^{-\frac{1}{\kappa}} \left\| \left( \sqrt{2z-1}+c_1 \sqrt{\rho} \right) ^{d-1} -\left( \sqrt{2z-1}-c_1 \sqrt{\rho} \right) ^{d-1} \right\| _{L^\infty (1/2, \rho ^{-\frac{1}{\kappa}} )} \\
& = c \rho ^{-\frac{1}{\kappa}}  \left( \left( \sqrt{2\rho ^{-\frac{1}{\kappa}}-1}+c_1 \sqrt{\rho} \right) ^{d-1} -\left( \sqrt{2\rho ^{-\frac{1}{\kappa}}-1}-c_1 \sqrt{\rho} \right) ^{d-1} \right) \\
& \leq c \rho ^{-\frac{1}{\kappa}}  \left( \sqrt{\rho} \left(  \sqrt{\rho ^{-\frac{1}{\kappa}}} \right) ^{d-2} + \mathcal{O} \left( \rho^\frac{3}{2} \left(  \sqrt{\rho ^{-\frac{1}{\kappa}}} \right) ^{d-4} \right) \right) \\
& \leq c \rho ^{-\frac{1}{\kappa} +\frac{1}{2}-\frac{1}{2\kappa}(d-2)} \\
& = c \rho ^{\frac{1}{2}\left( 1- \frac{d}{\kappa } \right)}.
\end{align*}
Hence, we see that $|S|$ converges to $0$ since $\kappa > d$.
All together, we can estimate $|R|$ by
\begin{align*}
|R| & \leq \int _{P_{\R^d_+} \Delta P_{\rho ^{-1}\Omega }}   \frac{|1-y_d^{p}|}{|e_d-y|^{d+2s}} \intd y \\
& \leq \int _S  \frac{|1-y_d^{p}|}{|e_d-y|^{d+2s}} \intd y + \int _{\R ^d \setminus B_{\rho ^{-\frac{1}{\kappa}}} (e_d) }    \frac{|1-y_d^{p}|}{|e_d-y|^{d+2s}} \intd y \\
& \leq c |S| + c  \int _{\R ^d \setminus B_{\rho ^{-\frac{1}{\kappa}}}  }    \frac{|y|^{p}}{|y|^{d+2s}} \intd y \\
& \leq c \rho ^{\frac{1}{2}\left( 1- \frac{d}{\kappa } \right)} +c \rho^{\frac{1}{\kappa}(2s-p)} .
\end{align*}
Hence, we finish the proof by choosing $\mu = \min \left\{ \frac{1}{2}\left( 1- \frac{d}{\kappa } \right) ,\frac{1}{\kappa}(2s-p) \right\}  \in (0,1)$.
\end{proof}

\begin{Lemma}\label{Lem:BoundLd2sMinus1PlusEpsilon}
Fix $s\in (1/2,1)$ and $\sigma \in (0,1]$ and $\epsilon \in (0, \min \{ 2-2s , 2s-1 \} )$. Let $\Omega \subset \R^d$ be a $C^{1,1}$ domain satisfying the exterior and interior ball condition with radius $r_0$ form \autoref{Prop:Barrier}.
Then there exists some $0<\rho _0 = \rho _0 (r_0,d ,s , \sigma ,\epsilon )$ such that for all $x_0\in \Omega \cap B_{4/5}$ with $d_\Omega (x_0) \leq \rho _0$, we have
\begin{equation*}
\mathcal{L}_{\Omega , \sigma} d _\Omega^{2s-1+\epsilon }  (x_0) \leq  - c d_\Omega ^{-1+\epsilon } (x_0)
\end{equation*}
for some constant $c=c(r_0,d ,s, \sigma , \epsilon )>0$.
\end{Lemma}
\begin{proof}
We will use the same geometric setting as in the proof of \autoref{Lem:BoundLd2sMinus1}.
Let us consider
\begin{equation}
\mathcal{L}_\Omega d_\Omega^{2s-1+\epsilon } (x_0) = \mathcal{L}_\Omega ( d_\Omega^{2s-1 + \epsilon} -l^{2s-1 + \epsilon}) (x_0) + \mathcal{L}_\Omega l^{2s-1 +\epsilon} (x_0) 
\end{equation}
where $l(y) =(y_d)_+$.
As in the proof of \autoref{Lem:BoundLd2sMinus1}, we get
\begin{equation*}
|d_\Omega ^{2s-1+\epsilon} (y) -l^{2s-1+\epsilon} (y)|  \leq c |y-x_0|^2 \left( d_\Omega ^{2s-2+\epsilon} (y) +l^{2s-2+\epsilon} (y) \right) \leq c |y-x_0|^2 \rho ^{2s-2+\epsilon}
\end{equation*}
for all $y\in B_{\rho /2} (x_0)$.
Since $|a^{\gamma}-b^{\gamma}|\leq |a-b|^{\gamma}$ for $\gamma \in (0, 1]$, we also get
\begin{equation*}
|d_\Omega ^{2s-1+\epsilon} (y) -l^{2s-1+\epsilon} (y)| \leq \left( |d_\Omega  (y) -l (y)| \right) ^{2s-1 + \epsilon } \leq c |y-x_0|^{2(2s-1+\epsilon )}
\end{equation*}
for all $y\in \Omega \setminus B_{\rho /2} (x_0)$.
All together, we obtain
\begin{align*}
|\mathcal{L}_\Omega ( d_\Omega ^{2s-1+\epsilon } -l^{2s-1+\epsilon }) (x_0)| & \leq \pv \int _\Omega \frac{|d_\Omega ^{2s-1 +\epsilon } (y) -l^{2s-1+\epsilon } (y)|}{|y-x_0|^{d+2s}} \intd y \\
& \leq c \int _{B_{\rho /2} (x_0)}  \frac{|y-x_0|^2 \rho ^{2s-2+\epsilon}}{|y-x_0|^{d+2s}} \intd y \\
& \qquad + c \int _{\Omega \setminus B_{\rho /2} (x_0)} \frac{  |y-x_0|^{2(2s-1+\epsilon)}  }{|y-x_0|^{d+2s}} \intd y \\
& \leq c \rho ^{2s-2+\epsilon } \left( \frac{\rho}{2} \right) ^{2-2s} + c (1+ \rho ^{2s-2+2\epsilon} ) \\
& \leq c (1+\rho ^{2s-2 +2 \epsilon} )  .
\end{align*}
For $\sigma \in (0,1)$ we can apply \autoref{Lem:ConvergenceToEllipse} and use the fact that $a(2s-1+\epsilon , \sigma )<0$ to get
\begin{equation*}
\mathcal{L}_\Omega l^{2s-1 +\epsilon} (x_0) \leq \rho ^{-1+\epsilon } \left( a(2s-1+\epsilon ,\sigma) + R \right)
\end{equation*}
with $|R| \leq c \rho ^{\frac{1}{2}}$.
All together, for $\sigma \in (0,1)$, we obtain
\begin{equation*}
\mathcal{L}_{\Omega , \sigma} d _\Omega^{2s-1+\epsilon }  (x_0) \leq a(2s-1+\epsilon ,\sigma) \rho^{-1+\epsilon } + c \rho ^{-\frac{1}{2}+\epsilon} + c (1+\rho ^{2s-2 +2 \epsilon} ) \leq -c \rho ^{-1+\epsilon }
\end{equation*}
for $\rho_0 >0$ small enough.
For $\sigma =1$ we can apply \autoref{Lem:ConvergenceToParabola} and use the fact that $a(2s-1+\epsilon , \sigma )<0$ to get
\begin{equation*}
\mathcal{L}_\Omega l^{2s-1 +\epsilon} (x_0) \leq \rho ^{-1+\epsilon } \left( a(2s-1+\epsilon ,\sigma) + R \right)
\end{equation*}
with $|R| \leq c \rho ^{\mu}$.
We obtain
\begin{equation*}
\mathcal{L}_{\Omega , \sigma} d _\Omega^{2s-1+\epsilon }  (x_0) \leq a(2s-1+\epsilon ,\sigma) \rho^{-1+\epsilon } + c \rho ^{-1+\mu +\epsilon} + c (1+\rho ^{2s-2 +2 \epsilon} ) \leq -c \rho ^{-1+\epsilon }
\end{equation*}
for $\rho_0>0$ small enough.
\end{proof}

\section{Hölder regularity up to the boundary} \label{sec:HoelderRegUpToBoundary}
Using the fact that $\mathcal{E}_\Omega  (v,v )$ is comparable to $[v] _{H^s(\Omega )}$ (see \autoref{Cor:EnergiesComparable}), we can apply \cite[Thm. 1]{Fukushima1977} to get an $L^\infty$ bound of $u$.
For completion, we present the proof from \cite{Fukushima1977} here.
\begin{Proposition}\label{Prop:LinftyBound}
Fix $s\in (1/2,1)$ and $\sigma \in (0,1]$. Let $\Omega \subset \R^d$ be a bounded $C^{1,1}$ domain.
Assume that $u\in H^s_0(\Omega )$ is a weak solution to
\begin{gather*}
\begin{aligned}
\mathcal{L}_{\Omega , \sigma} u &= f \quad \text{ in }\Omega \\
u& =0 \quad \text{ on } \partial \Omega 
\end{aligned}
\end{gather*}
where $f\in L^\infty (\Omega )$.
Then $u$ is bounded with
\begin{equation*}
\| u\| _{L^\infty (\Omega )} \leq c \| f\| _{L^\infty (\Omega )} .
\end{equation*}
for some $c=c(d,s,\sigma ,\Omega )>0$.
\end{Proposition}
\begin{proof}
Set $v=(u-k)_+$ for some $k\geq 0$. Then $v\in H^s_0(\Omega )$.
Testing the equation with $\varphi =v$ and using that $\mathcal{E}_\Omega  (v,v ) \leq\mathcal{E}_\Omega  (u,v ) $, we obtain
\begin{equation*}
\mathcal{E}_\Omega  (v,v ) \leq \int _{\Omega } f v . 
\end{equation*}
Since $\mathcal{E}_\Omega (v,v )$ is comparable to $[v] _{H^s(\Omega )}$ (see Corollary \ref{Cor:EnergiesComparable}), we can use Sobolev's embedding theorem to get
\begin{equation*}
 \| v \| ^2_{L^p (\Omega )} \leq c \mathcal{E}_\Omega  (v,v ) \leq c \int _{\Omega } f v 
\end{equation*}
where $p=\frac{2d}{d-2s}>2$. Applying Hölder's inequality, we obtain
\begin{equation*}
\| v \| ^2 _{L^p (\Omega )} \leq c \int _{\Omega } f v \leq c \| f\| _{L^\infty (\Omega )} \| v \|  _{L^p (\Omega )} A(k) ^{1-\frac{1}{p}}
\end{equation*}
where $A(k)= | \left\{ x\in \Omega \colon u(x)> k \right\} |$.
Hence, for every $h>k\geq 0$, we get
\begin{equation*}
A(h) \leq \frac{1}{(h-k)^p} \int _\Omega |v|^p \leq \frac{1}{(h-k)^p} c^p \| f\| _{L^\infty (\Omega )} ^p A(k) ^{p-1} .
\end{equation*}
By Stampacchia's Lemma (see, e.g., \cite[Lemma 4.1]{Stampacchia1965}), we get that $u$ is bounded from above 
\begin{equation*}
u \leq c \|f\| _{L^\infty (\Omega )} . 
\end{equation*}
Repeating the same proof for $-u$ and $-f$ completes the proof.
\end{proof}
\begin{Proposition}\label{Prop:HoelderRegulUptoBoundaryLocal}
Fix $s\in (1/2,1)$, $\sigma \in (0,1]$, and $\epsilon >0$.
Let $\Omega \subset \R^d$ be a $C^{1,1}$ domain, such that for every $z\in \partial \Omega \cap B_1$ we can find two balls of radius $r_0>0$ in $\Omega$ and $\Omega ^c$ which both touch $\partial \Omega$ in $z$.
Assume that $u$ is a weak solution of
\begin{equation}\label{eq:EquationForLocalizedHoelderReg}
\begin{aligned}
\mathcal{L}_{\Omega ,\sigma} u & =f \qquad \text{in} \quad   B_1 \cap \Omega , \\
u&= 0 \qquad \text{on} \quad \partial  B_1\cap \Omega  .
\end{aligned}
\end{equation}
with $u\in L^\infty _{2s-\epsilon} (\Omega ) $ and $d_\Omega ^{1-\epsilon '} f \in L^\infty (B_1\cap \Omega)$ for some $\epsilon '>0$.
We impose the boundary condition $u=0$ on $B_1\cap \Omega$ by assuming that $u\chi \in H^s(\R^d )$ for every $\chi \in C^\infty _c (B_1)$ where we extend $u$ by $u=0$ on $\Omega ^c$.

Then $u \in C^{2s-1}(B_{1/2} \cap \overline{\Omega })$ with
\begin{equation} \label{eq:LocalHoelderBoundUoToBoundary}
[u] _{C^{2s-1}(B_{1/2} \cap \overline{\Omega })} \leq c \left( \| u \| _{L^\infty _{2s-\epsilon} (\Omega )} + \| d_\Omega ^{1-\epsilon '} f\| _{L^\infty (B_1\cap \Omega)} \right)
\end{equation}
for some constant $c=c(r_0,s,\sigma ,\epsilon , \epsilon ')>0$.
\end{Proposition}
\begin{proof}
W.l.o.g., we can assume that $|d_\Omega (x) | \leq c$ for all $x\in B_1$, for some constant $c=c(\sigma )>0$, since otherwise we can use \autoref{Lem:InteriorRegularity} to get \eqref{eq:LocalHoelderBoundUoToBoundary}.
Furthermore, by dividing the equation by some constant, we can assume that $\| u \| _{L^\infty _{2s-\epsilon} (\Omega )} + \| d_\Omega ^{1-\epsilon '} f\| _{L^\infty (B_1\cap \Omega)} \leq 1$.
We split the proof into three steps.

\emph{Step 1:}
Let $\chi  \in C^\infty _c(\R^d )$ be a smooth cutoff function with $\ind _{B_{3/4}} \leq \chi  \leq \ind _{B_{4/5}}$.
Define $\tilde{u}:=u\chi $. We want to bound $\mathcal{L}_\Omega \tilde{u}$ on $B_{4/5} \cap \Omega$.
For $x\in B_{4/5} \cap \Omega$, we have
\begin{gather}\label{eq:ProductEvaluationTildeU}
\begin{aligned}
\mathcal{L}_\Omega \tilde{u} (x) &= u(x) \mathcal{L}_\Omega \chi  (x) + \chi  (x) \mathcal{L}_\Omega u(x) \\
& \qquad \quad - \int _\Omega \frac{(u(x)-u(y))(\chi  (x) -\chi  (y) ) }{|x-y|^{d+2s}} \mathcal{B}_\Omega (x,y) \intd y  \\
& = u(x) \mathcal{L}_\Omega \chi  (x) + \chi  (x) \mathcal{L}_\Omega u(x) - 2 \Gamma (u, \chi )
\end{aligned}
\end{gather}
where $\Gamma (u, \chi )$ is the carré du champ operator.
Notice, that for $x\in B_{4/5}\cap \Omega$ we have $|u(x) \mathcal{L}_\Omega \chi  (x)| \leq c_1$ since $\| u \| _{L^\infty _{2s-\epsilon} (\R^d)} \leq 1$ and $\chi $ is smooth.
Furthermore,
\begin{equation*}
| \chi  (x) \mathcal{L}_\Omega u(x) | \leq | \mathcal{L}_\Omega u(x) | \leq d_\Omega ^{-1+\epsilon '} (x)
\end{equation*} 
for $x\in B_{4/5} \cap \Omega $.
It remains to estimate $\Gamma (u, \chi )$ from \eqref{eq:ProductEvaluationTildeU}.
Notice that $| \chi  (x) -\chi  (y) | \leq c |x-y|$ for $c= \| \chi  \| _{C^1}$.
Hence,
\begin{align*}
2 |\Gamma (u, \chi )| & \leq   c\int _{ \Omega \cap B_{\kappa d_\Omega (x)}(x)} \frac{|u(x)-u(y)| }{|x-y|^{d+2s-1}} \mathcal{B}_\Omega (x,y) \intd y  \\
& \quad + c \int _{ \Omega \cap B_{\kappa d_\Omega (x)}^c (x)} \frac{|u(x)-u(y)| }{|x-y|^{d+2s-1}} \mathcal{B}_\Omega (x,y) \intd y \\
 & = I_1 +I_2 
\end{align*}
where $\kappa =\left( 2(2+3/\sigma ) \right) ^{-1}$.
We have
\begin{align*}
|I_2|  & \leq c \int _{B_{\kappa d_\Omega (x)}^c (x) \cap \Omega} \frac{|u(x)| + |u(y)| }{|x-y|^{d+2s-1}} \intd y \\
 & \leq c \| u \| _{L^{\infty}_{2s-\epsilon}(\R^d )} \int _{B_{\kappa d_\Omega (x)}^c (x) } \frac{1}{|x-y|^{d+2s-1}} \intd y + c\int _{B_{\kappa d_\Omega (x)}^c  (0)} \frac{|u(x+y)|}{|y|^{d+2s-1}} \intd y  \\
 & \leq c  (\kappa d_\Omega (x)) ^{1-2s} +  c \| u \| _{L^{\infty}_{2s-\epsilon} (\R^d) } \left( \int _{B_{\kappa d_\Omega (x)}^c \cap B_2} \frac{1}{|y|^{d+2s-1}} \intd y +\int _{B_{\kappa d_\Omega (x)}^c \cap B_2^c} \frac{|y|^{2s-\epsilon}}{|y|^{d+2s-1}}\intd y \right) \\
 & \leq c_2  \left( d_\Omega (x)\right) ^{1-2s} 
\end{align*}
for $x\in B_{4/5}\cap \Omega$.
For $I_1$, we will use the interior regularity result from \autoref{Cor:ScaledDownInteriorRegularity}.
Notice, that by \autoref{Cor:ScaledDownInteriorRegularity}, we have
\begin{align*}
[u] _{C^1(B_{\kappa d_\Omega (x)} (x))} & \leq c  (\kappa d_\Omega (x) )^{-1} \left( 1+ (\kappa d_\Omega (x))^{2s} \| f \| _{L^\infty (B_{2\kappa d_\Omega (x)} (x))} \right) \\
& \leq c ( d_\Omega (x) )^{-1}
\end{align*}
where in the last step we have used that $|f(y)|\leq d_\Omega ^{-1+\epsilon '} (y)\leq c d_\Omega ^{-1+\epsilon '} (x) $ for all $y\in B_{2\kappa d_\Omega (x)} (x)$.
Hence, for all $y\in B_{\kappa d_\Omega (x)} (x)$
\begin{equation*}
|u(x)-u(y)| \leq c |x-y| ( d_\Omega (x) )^{-1} .
\end{equation*}
Therefore, we get
\begin{align*}
|I_1|& \leq c( d_\Omega (x) )^{-1} \int _{ \Omega \cap B_{\kappa d_\Omega (x)}(x)} \frac{1 }{|x-y|^{d+2s-2}} \mathcal{B}_\Omega (x,y) \intd y \\
& \leq c( d_\Omega (x) )^{-1} \left(  \kappa d_\Omega (x) \right) ^{2-2s} \leq c_3 \left( d_\Omega (x)\right) ^{1-2s} .
\end{align*}
All together, we obtain for $x\in B_{4/5}\cap \Omega$
\begin{equation*}
| \mathcal{L}_\Omega \tilde{u} (x) | \leq c_1+ \left( d_\Omega (x) \right) ^{-1+\epsilon '} + (c_2+c_3) \left( d_\Omega (x) \right) ^{1-2s} \leq c_4 \left( d_\Omega (x) \right) ^{-1+\epsilon '}
\end{equation*}
where in the last step, w.l.o.g., we have assumed that $\epsilon ' <2-2s$.

\emph{Step 2:} By \autoref{Prop:Barrier}, there exists a barrier $\tilde{\varphi } \in H^s_{loc} (\R^d )$ satisfying
\begin{align*}
\mathcal{L}_\Omega \tilde{\varphi } & \geq d_\Omega ^{-1+\epsilon '} \qquad \text{in} \quad \{ 0 < d_\Omega (x) <d \} \cap B_{4/5} \\
 d_\Omega ^{2s-1} \leq \tilde{\varphi } & \leq c_5 d_\Omega ^{2s-1}  \qquad \text{in} \quad \Omega \cap B_2 \\
 \tilde{\varphi } & = 0  \qquad \text{in} \quad \R^d\setminus B_3
\end{align*}
for some constants $c_5>1$ and $d >0$.
Then,
\begin{equation*}
c_6 \mathcal{L}_\Omega \tilde{\varphi } \geq \mathcal{L}_\Omega \tilde{u} \quad \text{in } \{ 0 < d_\Omega (x) <d \} \cap B_{4/5} \quad \text{and} \quad  c_6 \tilde{\varphi } \geq \tilde{u} \quad \text{in } \R^d \setminus \left( \{ 0 < d_\Omega (x) <d \} \cap B_{4/5} \right)
\end{equation*}
for some large $c_6>0$ depending on $c_4$ and $d$.
Hence, by the maximum principle (see \autoref{Lem:WeakMaximumPrinciple}), this implies $\tilde{u}=u \leq c_6 \tilde{\varphi } \leq c_5c_6 d_\Omega ^{2s-1}$ in $\Omega \cap B_{3/4}$.
Repeating the same argument for $-u$ yields
\begin{equation}\label{eq:uLeqDist}
|u(x)| \leq c d_\Omega (x)^{2s-1} \qquad \text{for all} \quad x \in B_{3/4}\cap \Omega .
\end{equation}

\emph{Step 3:}
In the third step, we combine $|u| \lesssim d _\Omega ^{2s-1}$ with interior regularity results to conclude $[u] _{C^{2s-1}(\overline{ \Omega } \cap B_{1/2} )} \leq c$.

Let $x_1 ,x_2 \in \Omega \cap B_{1/2} $ and assume that $r:=d_\Omega (x_1) \leq d_\Omega (x_2)$ and set $\rho := |x_1-x_2|$.
We will distinguish between two cases.

\textbf{Case 1:}
If $\rho \geq \frac{r}{2}\left( 2+\frac{3}{\sigma} \right) ^{-1} $, then
\begin{equation*}
|u(x_1) - u(x_2)| \leq  |u(x_1) | + | u(x_2)| \leq c r^{2s-1}+ c ( \rho +r )^{2s-1} \leq c \rho ^{2s-1} 
\end{equation*}
using \eqref{eq:uLeqDist}.

\textbf{Case 2:}
Now assume $\rho <\frac{r}{2}\left( 2+\frac{3}{\sigma} \right) ^{-1} $. Set $R:= r\left( 2+\frac{3}{\sigma} \right) ^{-1}$, then $B_{R\left( 2+\frac{3}{\sigma} \right)}(x_1)\subset \Omega$ and $x_2\in B_\frac{R}{2}(x_1)$.
We define the rescaled function $\overline{u}\colon R^{-1} (\Omega - x_1) \to \R$, $\overline{u}(x)= \tilde{u}(x_1+Rx)$.
By scaling of the operator $\mathcal{L}$ (see Lemma \ref{Lem:ScalingOperator}), we get
\begin{equation*}
\mathcal{L}_{R^{-1} (\Omega - x_1) } \overline{u} (x) = R^{2s} \mathcal{L}_\Omega \tilde{u} (x_1+Rx) =:\overline{f}(x)
\end{equation*}
for $x\in B_1$ in the weak sense.
By interior regularity (see Lemma \ref{Lem:InteriorRegularity}), we have
\begin{equation*}
[\overline{u}] _{C^{2s-1} (\overline{B_{1/2}})} \leq c \left( \| \overline{u} \| _{L^\infty _{2s-1} ( R^{-1} (\Omega - x_1) )} +\| \overline{f} \| _{L^\infty (B_1) } \right)  .
\end{equation*}
Let $x_0 \in \partial \Omega$ such that $|x_0-x_1| = d_\Omega (x_1)=r$.
Using \eqref{eq:uLeqDist}, we get for all $x\in B_{1/(4R)}$
\begin{align*}
|\overline{u} (x)| &= |\tilde{u}(x_1+Rx)| \leq c d_\Omega ^{2s-1} (x_1+Rx) \leq c |x_1+Rx - x_0| ^{2s-1} \leq c \left( |x_1-x_0| + |Rx| \right) ^{2s-1} \\
& \leq c \left( |x_1-x_0|^{2s-1} + |Rx|^{2s-1} \right) = c \left( r^{2s-1} +R^{2s-1} |x|^{2s-1} \right) \leq c R^{2s-1} \left( 1+ |x|^{2s-1} \right) .
\end{align*}
Furthermore, using $\|\overline{u}\| _{L^\infty (\R^d)} = \| \tilde{u} \| _{L^\infty (\R^d )} \leq 2$, we also obtain $|\overline{u} (x)| \leq c R^{2s-1} \left( 1+ |x|^{2s-1} \right)$ for all $x\in B_{1/(4R)} ^c$, i.e. $\| \overline{u} \| _{L^\infty _{2s-1} ( R^{-1} (\Omega - x_1) )} \leq c R^{2s-1}$.
Since also $\| \overline{f} \| _{L^\infty (B_1) } \leq c R^{2s-1}$, we obtain
\begin{equation*}
 R^{2s-1} [ u ]_{C^{2s-1}(B_{R/2}(x_1))}= [\overline{u}] _{C^{2s-1} (\overline{B_{1/2}})} \leq c R^{2s-1} ,
\end{equation*}
i.e. $[ u ]_{C^{2s-1}(B_{R/2}(x_1))} \leq c$. Since $x_1,x_2 \in B_{R/2}(x_1))$, we get
\begin{equation*}
|u(x_1) - u(x_2)| \leq c |x_1 -x_2|^{2s-1} = c \rho ^{2s-1} .
\end{equation*}

Combining case 1 and  case 2 yields $[u] _{C^{2s-1}(\overline{ \Omega } \cap B_{1/2})} \leq c$ which finishes the proof.
\end{proof}
\begin{Remark}
\autoref{Prop:HoelderRegulUptoBoundaryLocal} remains true for strong solutions $u \in C^{2s+\epsilon }_{loc} (\Omega \cap B_1) \cap C(\overline{\Omega \cap B_1}) \cap L^\infty _{2s-\epsilon } (\Omega )$ to \eqref{eq:EquationForLocalizedHoelderReg}.
\end{Remark}
\begin{Remark}\label{Rem:HopfTypeLemma}
Following the proof of \cite[Prop. 2.6.6]{RosOton2024Book}, one can also prove a Hopf-type Lemma of the following form.
If $u\in H^s_0(\Omega )$ is a weak solution to \eqref{eq:DirichletProblem} with $0 \leq f\in L^\infty (\Omega )$, then either $u\equiv 0$ in $\Omega$ or
\begin{equation}\label{eq:HopfTypeEstimate}
|u(x)| \geq c \left( d_\Omega (x) \right)^{2s-1} \qquad \text{for all} \quad x\in \Omega ,
\end{equation} 
for some $c>0$.

The main step to prove \eqref{eq:HopfTypeEstimate} consists of changing the barrier in \autoref{Prop:Barrier} from $cd_\Omega^{2s-1} - d_\Omega ^{2s-1+\epsilon }$ to a barrier of the form $cd_\Omega^{2s-1} + d_\Omega ^{2s-1+\epsilon }$.
\end{Remark}

\begin{proof}[Proof of \autoref{Thm:HoelderContinuityWeakSol}]
First we use \autoref{Prop:LinftyBound} to see that $u$ is bounded.
Then, using \autoref{Prop:HoelderRegulUptoBoundaryLocal} on finitely many balls which cover $\Omega$, we finish the proof.
\end{proof}

\section{One dimensional boundary Harnack} \label{sec:BoundaryHarnack}
We start by introducing some notation.
For $R>0$, we define the following subset of $\R^d_+$
\begin{equation*}
I_R:= \R^{d-1} \times (0,R).
\end{equation*}
The goal of this section is to prove \autoref{Prop:BoundaryHarnack}.
For the proof, we mostly follow \cite[Sec. 5]{Audrito2023}.
However, difficulties arise since the operator degenerates as we approach the boundary.

In \autoref{Sec:BarriersForSigma1}, we construct specific one-dimensional barriers needed for the case $\sigma =1$.
However, they are not needed when $\sigma \in (0,1)$ due to the effect described in the last bullet point of \autoref{Rem:CommentsOnMainResults}.
In \autoref{Sec:OscillationDecay}, we combine a Hopf type result (\autoref{Lem:HopfTypeResult}) with an interior Harnack inequality (\autoref{Lem:InteriorHarnack}) to get an oscillation decay at the boundary (\autoref{Lem:OscDecayBoundaryHaranck}).
Finally, in \autoref{sec:OneDBoundaryHarnack}, we prove \autoref{Prop:BoundaryHarnack}.

\begin{Proposition}\label{Prop:BoundaryHarnack}
Fix $\sigma \in (0,1]$, $s\in (1/2,1)$, and $\epsilon >0$.
Assume that $u\in C_{loc}^{2s+\epsilon}(\R^d_+) \cap C(\overline{\R^d_+ })$ is a one dimensional function, i.e. $u(x)=\tilde{u}(x_d)$, and $u$ satisfies  
\begin{align*}
\mathcal{L}_{\R^d_+,\sigma } u & = 0 \qquad \text{in} \quad  I_{2R} , \\
u&=0 \qquad \text{on} \quad \partial \R^d_+
\end{align*}
for some $R\geq 1$.
Furthermore, assume that $u$ satisfies the growth condition
\begin{equation}\label{eq:GrowthCondtionBoundaryHarnack}
|u(y)| \leq c_0 \left( 1+ |y|^{2s-\epsilon _0} \right) \qquad (y\in \R^d_+ )
\end{equation}
for some $c_0 \geq 0$ and $\epsilon _0 >0$.

Then there exists some $\alpha = \alpha (d,s,\sigma ) >0$ and $c=c(d,s,\sigma ,c_0 ,\epsilon _0)>0$ such that
\begin{equation*}
\left[ \frac{u(x)}{x_d^{2s-1}}  \right] _{C^\alpha (I_R)} \leq c R^{1-2s-\alpha }  \| u \| _{L^\infty (I_{2R})}  .
\end{equation*}
\end{Proposition}

\subsection{Barriers for $\sigma =1$}\label{Sec:BarriersForSigma1}
\begin{Lemma}\label{Lem:UpperBarrierForBoundaryHarnack}
Fix $\sigma = 1$ and $s\in (1/2,1)$.
Let $\eta \in C^\infty _c ([0,2))$ with $\ind _{[0,1]} \leq \eta \leq \ind _{[0,2]}$ and define $\varphi \colon \R^d_+ \to \R$ by $\varphi (x) = \eta (x_d) x_d^{2s-1}$. Then 
\begin{equation}\label{eq:BoundSupperSol1d}
\mathcal{L}_{\R^d_+,\sigma } \varphi (x) \geq c  \left( x_d\right) ^\frac{d-1}{2} \qquad \text{for} \quad x \in I_1 ,
\end{equation}
for some $c=c(s,d)>0$.
\end{Lemma}
\begin{proof}
Let $x\in I_1$. By shifting (see \autoref{Lem:ScalingOperator}), we can assume that $x=(0, \dots , 0 , x_d)$.
Using that $\mathcal{L}_{\R^d_+ } x_d^{2s-1} =0$ on $\R^d_+$, we get
\begin{align*}
\mathcal{L}_{\R^d_+} \varphi (x) & = \mathcal{L}_{\R^d_+} x_d^{2s-1} + \mathcal{L}_{\R^d_+} (x_d^{2s-1} (\eta -1))  \\
& = \int _{\R^d_+}(1-\eta (y_d) ) y_d^{2s-1} \mathcal{K}(x,y) \intd y  \\
& \geq  \int _{\R^d_+ \cap \{ y\colon y_d\in (2,3) \}} (1-\eta (y_d) ) y_d^{2s-1} \mathcal{K}(x,y) \intd y.
\end{align*}
Recall that
\begin{equation}\label{eq:KernelKInHalfspace}
\mathcal{K}(x,y)= |x-y|^{-d-2s} \frac{1}{2} \left( \ind _{B_{x_d}(x)} (y) + \ind _{P(x)} (y)  \right)
\end{equation}
where
\begin{equation}\label{eq:ParabolaPartKernel}
P(x)= \left\{ y\in \R^d_+ \colon |y-x|< y_d \right\}
= \left\{ y\in \R^d_+ \colon y_d>\frac{x_d}{2} \text{ and }|y'| <   \sqrt{x_d} \sqrt{2y_d-x_d} \right\} .
\end{equation}
Here we will only use the $P(x)$ part of the kernel.
Hence,
\begin{align*}
\mathcal{L}_{\R^d_+ } \varphi (x) & \geq \frac{1}{2} \int _{P(x)\cap \{ y_d \in (2,3) \}} (1-\eta (y_d) ) y_d^{2s-1} |x-y|^{-d-2s} \intd y  \\
& \geq c \int _{P(x)\cap \{ y_d \in (2,3) \}}  |x-y|^{-d-2s} \intd y   .
\end{align*}
Notice, that $ \left\{ y\colon y_d \in (2,3) , |y'| \leq \sqrt{x_d} \right\} \subset P(x)\cap \{ y \colon y_d \in (2,3) \} $.
Therefore,
\begin{align*}
\mathcal{L}_{\R^d_+ } \varphi (x) \geq c \int _{\{ y \colon y_d \in (2,3), |y'| \leq \sqrt{x_d} \}} |x-y|^{-d-2s}  \intd y \geq   c \left( \sqrt{x_d} \right)^{d-1} .
\end{align*}
\end{proof}

\begin{Lemma}\label{Lem:LowerBarrierForBoundaryHarnack}
Fix $\sigma =1$ and $s\in (1/2,1)$.
Let $\eta \in C^\infty _c ([0,2))$ with $\ind _{[0,5/4 ]} \leq \eta \leq \ind _{[0,2]}$.
Furthermore let $\zeta \in C^\infty _c ((1,2))$ with $\ind _{[5/4,7/4]} \leq \zeta \leq \ind _{[1,2]}$.
Then, for every $N>0$ there exists $M=M(s,N ,d )>1$ such that $\varphi \colon \R^d_+ \to \R$, $\varphi (x) = \eta (x_d) x_d^{2s-1} + M \zeta (x_d)$ satisfies
\begin{equation} \label{eq:LowerBarrierOneDimensionalBH}
\mathcal{L}_{\R^d_+,\sigma } \varphi (x) \leq - N  \left( x_d \right) ^\frac{d-1}{2} \qquad \text{for} \quad x \in I_1 .
\end{equation}
\end{Lemma}
\begin{proof}
Let $x\in I_1$. By Shifting (see \autoref{Lem:ScalingOperator}), we can assume that $x=(0, \dots 0,x_d)$.
Since $\mathcal{L}_{R^d_+}x_d^{2s-1}=0$ on $\R^d_+$ (see \autoref{Prop:CalculationsInHalfspace}), we have
\begin{align*}
\mathcal{L}_{\R^d_+ } \varphi (x) & = \int _{\R^d_+} y_d^{2s-1} (1 - \eta (y_d)) \mathcal{K}(x,y) \intd y - M \int _{\R^d_+} \zeta (y_d) \mathcal{K}(x,y) \intd y \\
& = \int _{\R^d_+ \cap \left\{y \colon 1\leq y_d\leq 2  \right\}} y_d^{2s-1} (1 - \eta (y_d)) \mathcal{K}(x,y) \intd y + \int _{\R^d_+ \cap \left\{y \colon 2\leq y_d< \infty \right\}} y_d^{2s-1}  \mathcal{K}(x,y) \intd y \\
& \qquad - M \int _{\R^d_+ \cap \left\{y \colon 1\leq y_d\leq 2  \right\}} \zeta (y_d) \mathcal{K}(x,y) \intd y \\
& = I_1(x) + I_2(x) -MI_3(x) .
\end{align*}
First, we want to show that $I_1(x) \leq c_1 \sqrt{x_d}^{d-1}$ for some constant $c_1=c_1(s,d)>0$.
Recall that
\begin{equation*}
\mathcal{K}(x,y)= |x-y|^{-d-2s} \frac{1}{2} \left( \ind _{B_{x_d}(x)} (y) + \ind _{P(x)} (y)  \right)
\end{equation*}
where
\begin{equation*}
P(x)= \left\{ y\in \R^d_+ \colon |y-x|< y_d \right\}
= \left\{ y\in \R^d_+ \colon y_d>\frac{x_d}{2} \text{ and }|y'| <   \sqrt{x_d} \sqrt{2y_d-x_d} \right\}  .
\end{equation*}
Note that $\supp \mathcal{K}(x, \cdot ) \cap \left\{y \colon y_d\in [1,2]  \right\} \subset \left\{ y\colon y_d \in [1,2] , |y'| \leq 2 \right\}$.
Furthermore, $1-\eta (y_d) \leq c |1-y_d|^2$ for all $y_d \in (1,2)$. Hence,
\begin{align*}
I_1(x) \leq \int _{\{ y \colon y_d \in [1,2], |y'| \leq 2 \}} c y_d^{2s-1}  \frac{|1-y_d|^2}{ |x-y|^{d+2s}} \intd y 
\leq \int _{\{ y \colon y_d \in [1,2], |y'| \leq 2 \}} c y_d^{2s-1}   |x-y|^{2-d-2s} \intd y 
\leq c 
\end{align*}
for some constant $c$ independent of $x$.
If $x_d< 1/2$ then $\supp \mathcal{K}(x, \cdot ) \cap \left\{y \colon y_d\in [1,2]  \right\} \subset \left\{ y\colon y_d \in [1,2] , |y'| \leq 2\sqrt{x_d} \right\}$, so we get
\begin{align*}
I_1(x) 
\leq \int _{\{ y \colon y_d \in [1,2], |y'| \leq 2\sqrt{x_d} \}} y_d^{2s-1} c  |x-y|^{2-d-2s} \intd y 
\leq c \sqrt{x_d}^{d-1} .
\end{align*}
Hence, we have shown that $I_1(x) \leq c_1 \sqrt{x_d}^{d-1}$.

Next, we will estimate $I_2$.
Since $B_{x_d}(x)\cap \{ y\colon y_d\geq 2\}=\emptyset $, we only have to deal with $P(x)$ as a domain for integration.
By transforming in spherical coordinates, we get
\begin{align*}
I_2(x)&= \frac{1}{2}\int_{P(x) \cap \{ y\colon y_d\geq 2\}} y_d^{2s-1} |x-y|^{-d-2s} \intd y \leq c \int_{P(x) \cap \{ y\colon y_d\geq 2\}} |y|^{-d-1} \intd y \\
&\leq c \int_2^\infty  \left| \{ y\in \R^d_+ \colon |y|=r\} \cap P(x) \right| r^{-d-1} \intd r .
\end{align*}
We get
\begin{align*}
\{ y\in \R^d_+ \colon |y|=r\} \cap P(x) & = \left\{ y\in \R^d_+ \colon |y|=r \text{ and } |y'| <\sqrt{2x_dy_d -x_d^2}  \right\} \\
&\subset \left\{ y\in \R^d_+ \colon |y|=r \text{ and } |y'| <\sqrt{2x_dy_d }  \right\} =:A_r.
\end{align*}
By scaling the hypersurface $A_r$ onto the unit sphere, we obtain
\begin{align*}
|A_r| & = r^{d-1} \left| \left\{ y\in \R^d_+ \colon |y|=1 \text{ and } |y'| <\sqrt{\frac{2y_dx_d}{r}}  \right\} \right| \\
& \leq r^{d-1} \left| \left\{ y\in \R^d_+ \colon |y|=1 \text{ and } |y'| <\sqrt{2y_dx_d}  \right\} \right| =: r^{d-1} | \tilde{A_r} |
\end{align*} 
where in the last step we have used that $r\geq 2$.
We will estimate $|\tilde{A_r}|$ by using $n$ dimensional spherical coordinates.
The important coordinate is the polar angle $\theta =\arcsin (|y'|)$.
By a direct calculation, for $\tilde{A_r}$ we integrate over those $\theta$ such that $\sin \theta \in [0,h]$ where
\begin{align*}
h =\sqrt{x_d \left( \sqrt{x_d^2+4}-x_d \right)} .
\end{align*}
By taylor expanding for small $x_d$, we obtain
\begin{align*}
|\tilde{A_r}| = c \int _0^{\arcsin (h)} \sin ^{d-2} (\theta ) \intd \theta \leq c \sqrt{x_d}^{d-1}.
\end{align*}
Hence, we get
\begin{align*}
I_2(x) \leq c \int_2^\infty  \sqrt{x_d}^{d-1} r^{-2} \intd r \leq c_2 \sqrt{x_d}^{d-1}
\end{align*}
for some positive constant $c_2=c_2(d,s)$.

For $I_3$, notice, that $ \left\{ y\colon y_d \in [1,2] , |y'| \leq \sqrt{x_d} \right\} \subset \supp \mathcal{K}(x, \cdot ) \cap \left\{y \colon y_d \in [1,2]  \right\} $.
Therefore
\begin{align*}
I_3(x)\geq c \int _{\{ y \colon y_d \in [1,2], |y'| \leq \sqrt{x_d} \}} \zeta (y_d)  \intd y \geq   c \sqrt{x_d}^{d-1} \int_1^2 \zeta (y_d)  \intd y_d \geq c_3 \sqrt{x_d}^{d-1}
\end{align*}
for some positive constant $c_3=c_3(d,s)$. 
All together, we obtain
\begin{align*}
\mathcal{L}_{\R^d_+ } \varphi (x) & = I_1(x) + I_2(x) -MI_3(x) \leq \left( c_1+c_2 - M c_3 \right) \sqrt{x_d}^{d-1} .
\end{align*}
By choosing $M>0$ large enough, we get \eqref{eq:LowerBarrierOneDimensionalBH}.
\end{proof}

\subsection{Oscillation decay}\label{Sec:OscillationDecay}
For $R>0$ define the following subsets of $\R^d_+$
\begin{equation*}
 I_R^+=I_{R,1}^+:= \R^{d-1} \times (R/4,R/2)
\end{equation*}
and for $\sigma \in (0,1)$
\begin{equation*}
I_{R,\sigma } ^+ := \R^{d-1} \times \left( R(1-\sigma )^2 /2 , R (1-\sigma ) /2 \right) .
\end{equation*}
\begin{Lemma}\label{Lem:HopfTypeResult}
Fix $\sigma =1 $, $s\in (1/2,1)$, $\epsilon >0$  and let $R>0$.
Assume that $u\in C(\overline{\R^d_+}) \cap C^{2s+\epsilon}_{loc}(\R^d_+) \cap L^\infty_{2s-\epsilon }(\R^d_+)$ is a one dimensional function, i.e. $u(x)=\tilde{u}(x_d)$, and $u$ satisfies 
\begin{alignat*}{2}
\mathcal{L}_{\R^d_+,\sigma } u(x) & \geq -K_0 \left( \frac{x_d}{R}\right) ^\frac{d-1}{2} \qquad & &\text{for} \quad x\in I_{R/4},  \\
u&\geq 0  \qquad & &\text{on }\quad \R^{d}_+ .
\end{alignat*}
for some constant $K_0\geq 0$.
Then
\begin{equation}\label{eq:HopfTypeResultIn1d}
\inf _{x\in I_R^+} \frac{u(x)}{x_d^{2s-1}} \leq c \left( \inf _{x\in I_{R/4}} \frac{u(x)}{x_d^{2s-1}} + K_0R \right)
\end{equation}
for some constant $c=c(s,d )>0$.

For $\sigma \in (0,1)$, assume that $u$ satisfies
\begin{align*}
\mathcal{L}_{\R^d_+,\sigma } u(x) & \geq 0  \qquad \text{for} \quad x\in I_{R (1-\sigma)^2 /2} ,\\
u&\geq 0  \qquad \text{on }\quad I_{R (1-\sigma) /2} .
\end{align*}
Then 
\begin{equation}\label{eq:HopfTypeResultIn1dForSigmaLess1}
\inf _{x\in I_{R,\sigma }^+} \frac{u(x)}{x_d^{2s-1}} \leq   \inf _{x\in I_{R (1-\sigma)^2 /2}} \frac{u(x)}{x_d^{2s-1}} .
\end{equation}
\end{Lemma}
\begin{proof}
We start with the proof of \eqref{eq:HopfTypeResultIn1d} for $\sigma =1$.

Let us first consider the case $K_0=0$.
By scaling, it is enough to prove \eqref{eq:HopfTypeResultIn1d} for $R=4$.
Define
\begin{equation*}
m:=\inf _{x\in I_4^+} \frac{u(x)}{x_d^{2s-1}} .
\end{equation*}
Then $u(x)\geq m x_d^{2s-1} \geq m$ for $x_d\in [1,2]$.
Let $\varphi$ be the barrier from \autoref{Lem:LowerBarrierForBoundaryHarnack} (with $N=0$ in \autoref{Lem:LowerBarrierForBoundaryHarnack}), i.e. $\varphi (x) = \eta (x_d) x_d^{2s-1} + M \zeta (x_d)$ for some $M>1$ and smooth cutoff functions $\eta$ and $\zeta$ with $\ind _{[0,5/4 ]} \leq \eta \leq \ind _{[0,2]}$ and $\ind _{[5/4,7/4]} \leq \zeta \leq \ind _{[1,2]}$, such that $\mathcal{L}_{\R^d_+} \varphi \leq 0$ on $I_1$.
Then obviously $m(M+2)^{-1}\varphi \leq u$ on $[2,\infty )$ since $\varphi =0$ on $[2,\infty )$ and 
\begin{equation*}
m(M+2)^{-1}\varphi (x) = m(M+2)^{-1} \left(\eta (x_d) x_d^{2s -1} + M \zeta (x_d) \right) \leq m \leq u(x)
\end{equation*}
for $x_d\in [1,2]$ as well.
Furthermore $m(M+2)^{-1}\varphi (0)=0 \leq u(0)$ and $\mathcal{L}_{\R^d_+} m(M+2)^{-1}\varphi \leq 0 \leq \mathcal{L}_{\R^d_+} u$ on $I_1$.
Hence, by the maximum principle (\autoref{Lem:MaximumPrincipleStrongSol} is still applicable even though $I_1$ is not a bounded domain, since $u$ and $\varphi$ are one-dimensional functions), we deduce
\begin{equation*}
m(M+2)^{-1}\varphi (x)   \leq u(x)
\end{equation*}
for $x\in \R^{d-1}\times (0,1)$, which proves \eqref{eq:HopfTypeResultIn1d} for $K_0=0$.

Now, let us consider the case $K_0>0$.
Recall that $\sigma =1$.
Again, by scaling, it is enough to proof \eqref{eq:HopfTypeResultIn1d} for $R=4$.
Let $\varphi$ be the barrier from \autoref{Lem:UpperBarrierForBoundaryHarnack}, i.e. $\varphi (x) = \eta (x_d)x_d^{2s-1}$ for some smooth cutoff function $\eta$ with $\ind_{[0,1]} \leq \eta \leq \ind_{[0,2]}$ and $\mathcal{L}_{\R^d_+} \varphi \geq c_1 \left( \sqrt{x_d}\right) ^{d-1}$ on $I_1$ for some constant $c_1>0$.
Define
\begin{equation*}
v(x):=\frac{K_0}{c_1} \varphi (x) + u(x) .
\end{equation*}
Then $\mathcal{L}_{\R^d_+}v \geq 0$ on $I_1$ and $v\geq 0$ in $\R^d_+$.
Hence, we can apply \eqref{eq:HopfTypeResultIn1d} to $v$, by the first part of the proof and get
\begin{equation*}
\inf _{x\in I_4^+} \frac{u(x)}{x_d^{2s-1}} \leq \inf _{x\in I_4^+} \frac{v(x)}{x_d^{2s-1}} \leq c \inf _{x\in I_1} \frac{v(x)}{x_d^{2s-1}}  \leq c \left( \inf _{x\in I_1} \frac{u(x)}{x_d^{2s-1}} + K_0 \right) .
\end{equation*}
This proves \eqref{eq:HopfTypeResultIn1d} also for $K_0>0$.

Finally we will prove \eqref{eq:HopfTypeResultIn1dForSigmaLess1} for the case $\sigma \in (0,1)$.
Again, by scaling, we can assume that $R=1$.
Set $m=\inf _{I_{1,\sigma} ^+} \frac{u(x)}{x_d ^{2s-1}} \geq 0$ and $\varphi (x) :=m x_d^{2s-1}$.
Let $\kappa =(1-\sigma )/2$.
By construction, we have
\begin{equation*}
\varphi (x) \leq u(x) \qquad \text{for all} \quad x\in I_{1,\sigma} ^+ =  \R^{d-1} \times \left( (1-\sigma ) \kappa, \kappa \right) .
\end{equation*}
By \autoref{Prop:CalculationsInHalfspace}, we also have $\mathcal{L}_{\R^d_+} \varphi = 0 \leq \mathcal{L}_{\R^d_+}  u$ on $I_{ (1-\sigma) \kappa}$.
Notice, that $\supp \mathcal{K} (x, \cdot ) \subset I_{ \kappa}$ for all $x\in I_{ (1-\sigma)\kappa }$.
Hence, we can apply the maximum principle to conclude
\begin{equation*}
m x_d^{2s-1} \leq u(x)
\end{equation*}
for all $x\in I_{ (1-\sigma) \kappa }$ which proves \eqref{eq:HopfTypeResultIn1dForSigmaLess1}.
\end{proof}

\begin{Lemma}\label{Lem:InteriorHarnack}
Fix $\sigma  =1$, $s\in (1/2,1)$, $\epsilon >0$.
Assume that $u\in C(\overline{\R^d_+}) \cap C^{2s+\epsilon}(\R^d_+) \cap L^\infty_{2s-\epsilon }(\R^d_+)$ is a one dimensional function, i.e. $u(x)=\tilde{u}(x_d)$, and $u$ satisfies  
\begin{alignat*}{2}
| \mathcal{L}_{\R^d_+,\sigma } u(x) | & \leq K_0 \qquad  & &\text{for} \quad x\in I_{3R/4} \\
u&\geq 0  \qquad & & \text{on }\quad \R^{d}_+ 
\end{alignat*}
for some $K_0\geq 0$ and $R>0$.
Then
\begin{equation}\label{eq:InteriorHarnackWithx2s-1}
\sup _{x\in I_{R}^+} \frac{u(x)}{x_d^{2s-1}} \leq c \left( \inf _{x\in I^+_{R }} \frac{u(x)}{x_d^{2s-1}} + K_0R  \right)
\end{equation}
for some constant $c=c(d,s )>0$.

For $\sigma \in (0,1)$, assume that $u$ satisfies
\begin{align*}
 \mathcal{L}_{\R^d_+,\sigma } u(x)  & =0 \qquad \text{for} \quad x\in I_{R(1-\sigma )} \\
u&\geq 0  \qquad \text{on}\quad I_R .
\end{align*}
Then
\begin{equation*}
\sup _{x\in I_{R,\sigma }^+} \frac{u(x)}{x_d^{2s-1}} \leq c  \inf _{x\in I^+_{R,\sigma }} \frac{u(x)}{x_d^{2s-1}} 
\end{equation*}
for some constant $c=c(d,s,\sigma )>0$.
\end{Lemma}
\begin{proof}
We start with the case $\sigma =1$.
By scaling, it is enough to prove the lemma for $R=1$.
The interior Harnack inequality (see, e.g., \cite{Cozzi2017} and adapt \cite[Lem. 4.2]{DiCastro2014} with the assumptions \cite[(2.2), (2.3)]{DiCastro2014} to obtain \cite[(6.47)]{Cozzi2017}) gives us
\begin{equation*}
\sup _{x\in I_{1,\sigma }^+} u(x) \leq c \left(  \inf _{x\in I_{1,\sigma }^+} u(x) +K_0 \right)   .
\end{equation*}
Now \eqref{eq:InteriorHarnackWithx2s-1} follows easily since $1/(x_d)^{2s-1}$ is bounded from above and below for $x\in I_{1,\sigma }^+$.

For $\sigma \in (0,1)$ we proceed similarly.
Note, that for $\sigma \in (0,1)$, the assumption $u\geq 0$ on $I_{R}$ is enough, since $\supp \mathcal{K}(x,\cdot ) \subset I_R$ for all $x\in I_{R(1-\sigma )}$. 
\end{proof}

\begin{Lemma}\label{Lem:OscDecayBoundaryHaranck}
Fix $\sigma \in (0,1]$, $s\in (1/2,1)$, $\epsilon >0$.
Assume that $u\in C^{2s+\epsilon}_{loc}(\R^d_+) \cap C(\overline{\R^d_+ }) \cap L^\infty_{2s-\epsilon}(\R^d_+)$ is a one dimensional function, i.e. $u(x)=\tilde{u}(x_d)$, and $u$ satisfies  
\begin{align*}
\mathcal{L}_{\R^d_+,\sigma } u(x) & = 0 \qquad \text{for} \quad x\in I_2 , \\
u(0)&=0 .
\end{align*}
Furthermore, assume that $u$ satisfies the growth condition \eqref{eq:GrowthCondtionBoundaryHarnack} for some $c_0 \geq 0$ and $\epsilon _0 >0$.
Set $\theta =4$ for $ \sigma =1$ and $\theta = 2(1-\sigma )^{-2}$ for $\sigma \in (0,1)$. 

Then there exists a non-decreasing sequence $(m_n)$, a non-increasing sequence $(M_n)$ and constants $\alpha = \alpha ( s,d, \sigma ) \in (0,1)$ and $K=K(s,d, \sigma ,c_0 ,\epsilon _0)>0$ such that
\begin{equation} \label{eq:BoundsMnmn}
m_n \leq \frac{u(x)}{x_d^{2s-1}} \leq M_n \qquad \text{for all} \quad x \in I_{\theta ^{-n}}
\end{equation}
and
\begin{equation}\label{eq:OscDecayMnmn}
 M_n - m_n = K \theta^{-n \alpha} 
\end{equation}
for all $n\in \N_0$.
\end{Lemma}
\begin{proof}
We will first show how to prove \eqref{eq:BoundsMnmn} and \eqref{eq:OscDecayMnmn} for $\sigma =1$.

\emph{Step 1:}
We start with the case $n=0$.
Using step 2 from the proof of \autoref{Prop:HoelderRegulUptoBoundaryLocal} (see \eqref{eq:uLeqDist}), we get $|u| \leq c_1 x_d^{2s-1}$ on $I_1$ for some constant $c_1=c_1(d,s,c_0,\epsilon _0)>0$.
By setting $m_0 := \inf _{x\in I_1} \frac{u(x)}{x_d^{2s-1}}$ and $M_0 := K+m_0$ for a constant $K\geq 2c_1$ which we will choose later, \eqref{eq:BoundsMnmn} and \eqref{eq:OscDecayMnmn} follow for $n=0$.

\emph{Step 2:}
Assume that \eqref{eq:BoundsMnmn} and \eqref{eq:OscDecayMnmn} holds for all $n\leq k$ for some $K\geq 2 c_1$ and $\alpha \in (0,1)$ which we will choose later. We will now construct $m_{k+1}$ and $M_{k+1}$.
Set
\begin{equation*}
v(x) := u(x) -m_k x_d^{2s-1} .
\end{equation*}
We decompose $v$ into its positive and negative part $v=v_+ - v_-$.
Notice that $v=v_+$ on $I_{4^{-k}}$ by \eqref{eq:BoundsMnmn}.

Let $y \in I_1 \setminus I_{4^{-k}}$ and choose the integer $j\leq k-1$ such that $4^{-j-1}< y_d \leq 4^{-j}$.
Then, 
\begin{gather}\label{eq:LowerBoundOnV}
\begin{aligned}
v(y) & \geq (m_j -m _k)y_d^{2s-1} \geq (m_j-M_j+M_k-m_k) y_d^{2s-1}  \\ 
& = K (-4 ^{-j\alpha} + 4^{-k \alpha }) y_d^{2s-1}  = -K 4^{-k (2s-1+\alpha)} \left( \frac{y_d}{4^{-k}} \right) ^{2s-1} \left( \frac{4^{-j\alpha}}{4^{-k \alpha}}-1 \right) \\
& \geq -K4^{-k(2s-1+\alpha )} \left( \frac{y_d}{4^{-k}} \right) ^{2s-1} \left( \left( \frac{4y_d}{4^{-k}} \right) ^\alpha -1 \right) .
\end{aligned}
\end{gather}
Now let $x \in  I_{\frac{3}{4} 4^{-k} }$. We want to bound $|\mathcal{L}_{\R^d_+} v_+(x) |$, so w.l.o.g. we can assume that $x=(0, \dots ,0, x_d)$.
Then
\begin{equation}\label{eq:J1J2PartLVMinus}
\begin{aligned}
0\leq -\mathcal{L}_{\R^d_+} v_-(x)  & = \frac{1}{2} \int _{P(x)} \frac{v_-(y)}{|x-y|^{d+2s}} \intd y  +
\frac{1}{2} \int _{B_{x_d}(x)} \frac{v_-(y)}{|x-y|^{d+2s}} \intd y  \\
& = J_1 +J_2 .
\end{aligned}
\end{equation}
We will start with $J_1$.
Note, that
\begin{align*}
J_1 & = \frac{1}{2}\int_{4 ^{-k}}^\infty  \int_{\left\{ |y'|\leq \sqrt{2x_dt-x_d^2} \right\}}  \,  \frac{v_-(t)}{\left( (t-x_d)^2+|y'|^2  \right) ^{\frac{d+2s}{2}}}  \intd y' \intd t\\
& = \frac{1}{2}\int_{4 ^{-k}}^\infty  \frac{v_-(t)}{|t-x_d|^{1+2s}}  \int_{\left\{ |u| \leq \frac{\sqrt{2x_dt-x_d^2}}{t-x_d} \right\}}  \frac{1}{\left( 1+|u| ^2 \right) ^\frac{d+2s}{2}}  \intd u \intd t
\end{align*}
where we have used that $v_-=0$ on $I_{4^{-k}}$, as well as the explicit formula for the kernel $\mathcal{K}$ (see \eqref{eq:KernelKInHalfspace} and \eqref{eq:ParabolaPartKernel}) and the transformation $y'= (t-x_d)u$.
To bound the integral in the $u$ variable, we will use that
\begin{equation*}
\frac{\sqrt{2x_dt-x_d^2}}{t-x_d} = \sqrt{x_d} \frac{\sqrt{2t-x_d}}{t-x} \leq \sqrt{x_d} \frac{\sqrt{2t-x_d}}{t/4} = \sqrt{x_d}\frac{4}{\sqrt{t}} \sqrt{2-\frac{x_d}{t}} \leq 8 \sqrt{\frac{x_d}{4^{-k}}}
\end{equation*}
for all $t\in ( 4^{-k} , \infty )$ and $x \in  I_{\frac{3}{4} 4^{-k} }$ .
Hence, we obtain
\begin{equation*}
0\leq J_1 \leq c\left(\frac{x_d}{4 ^{-k}}\right) ^\frac{d-1}{2}  \int_{4 ^{-k}}^\infty  \frac{v_-(t)}{|t-x_d|^{1+2s}}\intd t  .
\end{equation*}
Using \eqref{eq:GrowthCondtionBoundaryHarnack} and \eqref{eq:LowerBoundOnV}, we get
\begin{align*}
0\leq J_1 &\leq  c\left(\frac{x_d}{4 ^{-k}}\right) ^\frac{d-1}{2} \int_{4 ^{-k}}^1 K4^{-k(2s-1+\alpha )} \left( \frac{t}{4^{-k}} \right) ^{2s-1} \frac{  \left( \frac{4t}{4^{-k}} \right) ^\alpha -1 }{|t-x_d|^{1+2s}}\intd t \\
 & \qquad + c  \left(\frac{x_d}{4 ^{-k}}\right) ^\frac{d-1}{2} \int _1^\infty c_0 \frac{1+ t ^{2s- \epsilon _0}}{|t-x_d|^{1+2s}}\intd t \\
 & = c \left(\frac{x_d}{4 ^{-k}}\right) ^\frac{d-1}{2} (J_{2,1} + J_{2,2}) 
\end{align*}
for $x \in  I_{\frac{3}{4} 4^{-k} }$.
By a direct computation, we see that $J_{2,2} \leq c_2$ for some constant $c_2=c_2(d,s,c_0,\epsilon _0)>0$.
For $J_{2,1} $, we obtain for $x \in  I_{\frac{3}{4} 4^{-k} }$
\begin{align*}
J_{2,1}  &= K4^{-k(2s-1+\alpha )} \int_{4 ^{-k} - x_d}^{1-x_d}  \left( \frac{t+x_d}{4^{-k}} \right) ^{2s-1} \frac{  \left( \frac{4(t+x_d)}{4^{-k}} \right) ^\alpha -1 }{t^{1+2s}}\intd t \\
& \leq K4^{-k(2s-1+\alpha )} \int_{4^{-k}/4}^{1}  \left( \frac{t+x_d}{4^{-k}} \right) ^{2s-1} \frac{  \left( \frac{4(t+x_d)}{4^{-k}} \right) ^\alpha -1 }{t^{1+2s}}\intd t \\
& \leq K4^{-k(2s-1+\alpha )} \int_{4^{-k}/4}^{1}  \left( \frac{4t}{4^{-k}} \right) ^{2s-1} \frac{  \left( \frac{16t}{4^{-k}} \right) ^\alpha -1 }{t^{1+2s}}\intd t \\
& = K4^{-k(\alpha -1 )} \int_{1/4}^{4^k}  \left( 4t \right) ^{2s-1} \frac{  \left( 16t  \right) ^\alpha -1 }{t^{1+2s}}\intd t \\
& \leq K4^{-k(\alpha -1 )} f( \alpha )
\end{align*}
where
\begin{equation*}
f(\alpha ) := \int_{1/4}^{\infty}  \left( 4t \right) ^{2s-1} \frac{  \left( 16t  \right) ^\alpha -1 }{t^{1+2s}}\intd t .
\end{equation*}
Notice, that by the dominated convergence theorem, $f(\alpha ) \to 0$ as $\alpha \to 0$.

Combining the estimates on $J_{2,1}$ and $J_{2,2}$, we obtain
\begin{equation}\label{eq:J1Bound}
J_1 \leq c ( K f(\alpha ) +c_2 ) 4^{-k(\alpha -1 )} \left(\frac{x_d}{4 ^{-k}}\right) ^\frac{d-1}{2} .
\end{equation}
Note, that by similar arguments, one can also proof the estimate \eqref{eq:J1Bound} for $J_2$ from \eqref{eq:J1J2PartLVMinus}.
Hence, together with $\mathcal{L}_{\R^d_+} v (x) =0$, we get
\begin{equation*}
| \mathcal{L}_{\R^d_+} v_+ (x) | \leq  | \mathcal{L}_{\R^d_+} v (x) | +| \mathcal{L}_{\R^d_+} v_- (x) | \leq c ( K f(\alpha ) +c_2 ) 4^{-k(\alpha -1 )} \left(\frac{x_d}{4 ^{-k}}\right) ^\frac{d-1}{2}
\end{equation*}
for all $x \in  I_{\frac{3}{4} 4^{-k} }$.

Now, we apply \autoref{Lem:HopfTypeResult} and \autoref{Lem:InteriorHarnack} for $v_+$ with $R= 4^{-k}$ to obtain
\begin{gather}\label{eq:SupInfBoundV}
\begin{aligned}
\sup _{x\in I^+_{4^{-k}}} \left( \frac{u(x)}{x_d ^{2s-1}} - m_k \right) & \leq c  \left[ \inf _{x\in I^+_{4^{-k}}} \left( \frac{u(x)}{x_d ^{2s-1}} - m_k \right) + ( K f(\alpha ) +c_2 ) 4^{-k\alpha  }\right] \\
& \leq c   \left[ \inf _{x\in I_{4^{-k}/4}} \left( \frac{u(x)}{x_d ^{2s-1}} - m_k \right) + ( K f(\alpha ) +c_2 ) 4^{-k\alpha  }\right] 
\end{aligned}
\end{gather}
Repeating the same arguments as above for $w(x)= M_k x_d^{2s-1} -u(x)$ yields
\begin{equation} \label{eq:SupInfBoundW}
\sup _{x\in I^+_{4^{-k}}} \left( M_k - \frac{u(x)}{x_d ^{2s-1}}  \right) \leq c \left[ \inf _{x\in I_{4^{-k}/4}} \left( M_k - \frac{u(x)}{x_d ^{2s-1}}  \right) + ( K f(\alpha ) +c_2 ) 4^{-k\alpha  }\right] .
\end{equation}
Summing up \eqref{eq:SupInfBoundV} and \eqref{eq:SupInfBoundW} gives us
\begin{align*}
M_k - m_k & \leq c \left[ \inf _{x\in I_{4^{-k}/4}} \left( \frac{u(x)}{x_d ^{2s-1}} - m_k \right) +\inf _{x\in I_{4^{-k}/4}} \left( M_k - \frac{u(x)}{x_d ^{2s-1}}  \right) + ( K f(\alpha ) +c_2 ) 4^{-k\alpha  } \right] \\
& = c_3 \left[ M_{k} -m_k +\inf _{x\in I_{4^{-k}/4}}  \frac{u(x)}{x_d ^{2s-1}} -\sup _{x\in I_{4^{-k}/4}}  \frac{u(x)}{x_d ^{2s-1}}  +( K f(\alpha ) +c_2 ) 4^{-k\alpha  }\right] 
\end{align*}
for some constant $c_3=c_3(d,s )>1$.
Hence,
\begin{align*}
\sup _{x\in I_{4^{-k}/4}}  \frac{u(x)}{x_d ^{2s-1}} - \inf _{x\in I_{4^{-k}/4}}  \frac{u(x)}{x_d ^{2s-1}} & \leq \frac{c_3-1}{c_3} (M_k -m_k) +  ( K f(\alpha ) +c_2 ) 4^{-k\alpha  } \\
& = \left( \frac{c_3-1}{c_3} + f(\alpha ) +\frac{c_2}{K} \right) K 4^{-k \alpha }.
\end{align*}
By choosing $K = \max \{ 3c_2c_3 ,2c_1 \}$ and $\alpha \in (0,1)$ small enough such that
\begin{equation*}
f(\alpha ) \leq \frac{1}{3 c_3} \quad \text{and} \quad 1-\frac{1}{3c_3} \leq 4^{-\alpha} ,
\end{equation*}
we get
\begin{equation*}
\frac{c_3-1}{c_3} + f(\alpha ) +\frac{c_2}{K} \leq \frac{c_3-1}{c_3} +  \frac{1}{3 c_3} + \frac{c_2}{3c_2c_3} =  1-\frac{1}{3c_3} \leq 4^{-\alpha} .
\end{equation*}
By setting $m_{k+1}:= \inf _{x\in I_{4^{-k}/4}}  \frac{u(x)}{x_d ^{2s-1}}$ and $M_{k+1}:= m_{k+1} + K 4^{-\alpha (k+1)}$, both \eqref{eq:BoundsMnmn} and \eqref{eq:OscDecayMnmn} follow for $n=k+1$ in the case $\sigma =1$.

Let us now consider the case $\sigma \in (0,1)$. 
We will just comment on the differences in the proof.
The case $n=0$ follows from \eqref{eq:uLeqDist} as above.
In step 2, we can leave out the estimate on $| \mathcal{L}_{\R^d_+} v_+ (x)|$ since by the choice of $\theta$ we can directly apply \autoref{Lem:HopfTypeResult} and \autoref{Lem:InteriorHarnack} to $v(x)=u(x)-m_kx_d^{2s-1}$ and $R=\theta ^{-k}$ to obtain
\begin{equation*}
\sup _{x\in I^+_{\theta^{-k},\sigma }} \left( \frac{u(x)}{x_d ^{2s-1}} - m_k \right)  \leq c   \inf _{x\in I^+_{\theta^{-k},\sigma }} \left( \frac{u(x)}{x_d ^{2s-1}} - m_k \right) 
 \leq c   \inf _{x\in I_{\theta^{-(k+1)}}} \left( \frac{u(x)}{x_d ^{2s-1}} - m_k \right) .
\end{equation*}
Hence, as above, we get
\begin{equation*}
\sup _{x\in I_{\theta^{-(k+1)}}}  \frac{u(x)}{x_d ^{2s-1}} - \inf _{x\in I_{\theta^{-(k+1)}}}  \frac{u(x)}{x_d ^{2s-1}} \leq  \frac{c_3-1}{c_3} K \theta ^{-k \alpha } .
\end{equation*}
By choosing $\alpha \in (0,1)$ such that $\frac{c_3-1}{c_3} \leq \theta ^{-\alpha}$, we finish the proof.
\end{proof}

\subsection{Proof of 1-d boundary Harnack} \label{sec:OneDBoundaryHarnack}

\begin{proof}[Proof of \autoref{Prop:BoundaryHarnack}]
By scaling, it is enough to prove the result for $R=1$.
Set $d\colon \R_+\to \R$, $d(x):=x$ and $v\colon \R_+ \to\R$, $v:= \tilde{u}  /d^{2s-1}$.
By dividing $u$ by some constant, we can assume that $\| u \|_{L^\infty (I_2)} \leq 1$, so it remains to show that $[v]_{C^\alpha (0,1)} \leq c $ for some constant $c=c(d,s,\sigma , c_0 ,\epsilon _0 )$.

By \autoref{Lem:OscDecayBoundaryHaranck}, we have
\begin{equation} \label{eq:OscDecayFromOscDecayLemmaUsedInProofHoelderReg}
\sup _{x\in (0, \theta ^{-n})} v(x)  - \inf _{x\in (0, \theta ^{-n})} v(x)  \leq K \theta ^{-n\gamma } \quad \text{for all} \quad n\in \N
\end{equation}
for some constants $\theta =\theta (\sigma )>1$, $K=K(s,d,\sigma ,c_0 ,\epsilon _0 )>0$ and $\gamma = \gamma (s,d,\sigma , c_0 , \epsilon _0) \in (0,1)$.
Hence, from \eqref{eq:OscDecayFromOscDecayLemmaUsedInProofHoelderReg}, we know that $v$ can be continuously extended to $x=0$.

Set $\kappa = \frac{1}{2(2+3/\sigma )}$.
Then, by \autoref{Cor:ScaledDownInteriorRegularity}, we have for all $y\in (0,1)$
\begin{equation}\label{eq:InteriorEstimateUInHoelderProof}
[\tilde{u} ] _{C^{2s-1} (y-\kappa y, y+\kappa y)}  \leq c y^{-2s+1}
\end{equation}
for some constant $c =c(d,s,\sigma ,c_0 ,\epsilon _0)$.
Furthermore, by scaling, we obtain
\begin{equation}\label{eq:InteriorEstimateDInHoelderProof}
[d^{-2s+1} ]_{C^{2s-1} (y(1-\kappa ), y(1+\kappa ))} = y^{-4s+2} [d^{-2s+1} ]_{C^{2s-1} (1-\kappa , 1+\kappa )} \leq c y^{-4s+2}
\end{equation}
for some constant $c=c(s,\sigma )>0$.
Combining \eqref{eq:InteriorEstimateUInHoelderProof} and \eqref{eq:InteriorEstimateDInHoelderProof}, together with $\| \tilde{u} \| _{L^\infty (0,2)}\leq 1$ and $\| d^{-2s+1} \| _{L^\infty ((y(1-\kappa ), y(1+\kappa )))} \leq cy^{-2s+1}$ we get
\begin{gather}\label{eq:InteriorEstimatesVInHoelderProof}
\begin{aligned}
[v]_{C^{2s-1} (y(1-\kappa ), y(1+\kappa ))} & \leq \| d^{-2s+1} \| _{L^\infty ((y(1-\kappa ), y(1+\kappa )))} [\tilde{u} ] _{C^{2s-1} (y-\kappa y, y+\kappa y)}  \\
& \qquad + \| \tilde{u} \| _{L^\infty (0,2)} [d^{-2s+1} ]_{C^{2s-1} (y(1-\kappa ), y(1+\kappa ))} \\
& \leq c y^{2-4s}
\end{aligned}
\end{gather}
for all $y\in (0,1)$ for some constant $c=c(d,s,\sigma ,\epsilon _0 ,c_0 )>0$.

Now, fix some $p>2$.
Let $x,y\in (0,1)$ with $x\leq y$ and $\rho =|y-x|$.

\emph{Case 1:} Assume that $\rho \geq \kappa y^p$.
We will choose $n\in \N$ such that $\theta ^{-(n+1)} < y \leq \theta ^{-n}$ with $\theta >1$ from \autoref{Lem:OscDecayBoundaryHaranck}.
Then, by using \eqref{eq:OscDecayFromOscDecayLemmaUsedInProofHoelderReg}, we obtain
\begin{equation*}
|v(x) -v(y)| \leq |v(x) -v(0)| + |v(y) -v(0)| \leq 2 K \theta ^{-n\gamma } \leq  c \rho ^{\gamma /p } 
\end{equation*}
for some constant $c=c(d,s,\sigma ,\epsilon _0 ,c_0 )>0$.

\emph{Case 2:} If $\rho < \kappa y^p$, then $\rho < \kappa y$ since $p>1$, which implies $x\in (y(1-\kappa),y(1+\kappa ))$.
Using \eqref{eq:InteriorEstimatesVInHoelderProof}, we obtain
\begin{equation*}
|v(x) -v(y)| \leq c \rho ^{2s-1} y^{2-4s} \leq c \rho ^{2s-1} \rho ^{\frac{2-4s}{p}} = c \rho ^{(2s-1) \left( 1-\frac{2}{p} \right)}
\end{equation*}
for some constant $c=c(d,s,\sigma ,\epsilon _0 ,c_0 )>0$

We finish the proof by choosing $\alpha = \min \{ \gamma /p, (2s-1)(1-2/p) \} >0$.
\end{proof}

\section{Liouville type result in half-space}\label{sec:LiouvilleTypeResult}
The goal of this section is to prove the following Liouville type result in the half-space using the boundary Harnack result from \autoref{Prop:BoundaryHarnack}.
\begin{Proposition}\label{Prop:LiouvilleInHalfspace}
Fix $s\in (1/2,1)$, $\sigma \in (0,1]$, $\epsilon >0$.
Assume that $u \in C(\overline{\R^d_+}) \cap C^{2s+\epsilon}_{loc} (\R^d_+)$ satisfies
\begin{gather}\label{eq:EquationLiouvilleThm}
\begin{aligned}
\mathcal{L}_{\R^d_+ ,\sigma } u &= 0 \qquad \text{in} \quad \R^d_+ ,\\
u&=0 \qquad \text{on} \quad \partial \R ^d_+ .
\end{aligned}
\end{gather}
Then there exists some $\alpha _*=\alpha _* (d,s,\sigma )>0$ such that, if $u$ satisfies the growth bound
\begin{equation}\label{eq:GrowthAssumptionLiouvilleThm}
|u(x)| \leq c (1+ |x| ^{2s-1+\alpha _* } ) \qquad (x\in \R^d_+)
\end{equation}
for some $c>0$, then $u(x)=A x_d^{2s-1}$ for some $A \in \R$.
\end{Proposition}
\begin{proof}[Proof of \autoref{Prop:LiouvilleInHalfspace}]
Fix some $\alpha _* \in (0, \alpha )$ where $\alpha =\alpha (d,s,\sigma ) \in (0,1)$ is taken from \autoref{Prop:BoundaryHarnack}.
By \autoref{Lem:UOneDFunction}, we know that $u$ is a one dimensional function, i.e. $u(x)=\tilde{u}(x_d)$.
Then, \autoref{Prop:BoundaryHarnack} gives us
\begin{align*}
\left[ \frac{\tilde{u}(x)}{x^{2s-1}} \right] _{ C^\alpha (0,R)} & \leq c R^{1-2s-\alpha} \| \tilde{u} \| _{L^\infty (0,2R) } \\
& \leq c  R^{1-2s-\alpha} (1+(2R)^{2s-1+\alpha _\ast}) \\
& \leq c R^{\alpha_\ast -\alpha} \to 0 \qquad \text{as} \quad R\to \infty ,
\end{align*}
which implies \smash{$\left[ \frac{\tilde{u}(x)}{x^{2s-1}} \right]  _{ C^\alpha (\R _+)} =0$}.
Hence, $\tilde{u}(x)/x_d^{2s-1}$ is constant, i.e. $u(x)=A x_d^{2s-1}$ for some $A \in \R$.
\end{proof}
We finish the proof of \autoref{Prop:LiouvilleInHalfspace} by showing that the solution $u(x)$ only depends on $x_d$.
Essentially, the proof only uses the translational invariance of \smash{$\mathcal{L}_{\R^d_+}$} in the tangential direction \smash{$\R^{d-1}$} as well as the Hölder regularity bounds up to the boundary from \autoref{Prop:HoelderRegulUptoBoundaryLocal}.
\begin{Lemma}\label{Lem:UOneDFunction}
Let $u\in C(\overline{\R^d_+}) \cap C^{2s+\epsilon}_{loc} (\R^d_+)$ solve \eqref{eq:EquationLiouvilleThm} from \autoref{Prop:LiouvilleInHalfspace} and assume that \eqref{eq:GrowthAssumptionLiouvilleThm} holds true for some $\alpha _* \in (0,\alpha )$ where $\alpha =\alpha (d,s,\sigma ) \in (0,1)$ is taken from \autoref{Prop:BoundaryHarnack}.
Then $u$ is a one dimensional function, i.e. $u(x)=\tilde{u}(x_d)$ for some $\tilde{u}\colon \R_+\to \R$.
\end{Lemma}
\begin{proof}
Set $\beta =2s-1 \in (0,1)$ and $\gamma =2s-1+\alpha _* \in (0,2s) $.
We start with the following claim.

\textbf{Claim:} If $v$ satisfies 
\begin{align*}
\mathcal{L}_{\R^d_+} v &= 0 \quad \text{in} \quad\R^d_+ , \\
v& = 0 \quad \text{on}  \quad \partial \R^d_+
\end{align*}
with $\| v \| _{L^\infty (B_R)} \leq R^\mu $ for all $R\geq R_0$ for some $\mu \in (0,2s)$ and $R_0\geq 1$.
Then 
\begin{equation*}
[ v ] _{C^\beta (B_R)} \leq c_1 R^{\mu - \beta}  \quad \text{for all}\quad R\geq 1
\end{equation*}
for some constant $c_1 =c_1(R_0, d,s, \sigma , \mu ) >0$.

The claim can be proven by rescaling \autoref{Prop:HoelderRegulUptoBoundaryLocal}.
Indeed, for $R\geq 1$ let $v_R(x):=v(Rx)$. Then,
\begin{equation*}
|v_R(x)| =|v (Rx)| \leq R_0^\mu + R^\mu |x|^\mu \leq R_0^\mu  R^\mu(1+|x|^\mu ) \quad (x\in \R^d_+),
\end{equation*}
i.e. $\| v_R\| _{ L^\infty _\mu (\R^d )} \leq R_0 ^\mu R^\mu$.
Furthermore, $\mathcal{L}_{\R^d_+} v_R =0$ on $\R^d_+$. Hence, by \autoref{Prop:HoelderRegulUptoBoundaryLocal}
\begin{equation*}
[ v ] _{C^{2s-1} (B_{R/2})} = R^{-\beta} [v_R] _{C^{2s-1} (B_{1/2})} \leq c R^{-\beta } \| v_R\| _{ L^\infty _\mu (\R^d )} \leq c R_0 ^\mu R^{\mu -\beta }
\end{equation*}
for some constant $c=c(d,s, \sigma )$. This proves the claim.

Now let us prove \autoref{Lem:UOneDFunction}.
By dividing $u$ by some constant, we can assume that $\| u\| _{L^\infty (B_R)} \leq R^{\gamma}$ for all $R\geq 1$.
Hence, by the claim we know that $[u] _{C^{2s-1} (B_R)} \leq c_1 R^ {\gamma - \beta }$.

Fix some $h\in \R^{d-1} \times \{ 0\}$. We define
\begin{equation*}
v_1^h (x):= \frac{u(x+h)-u(x)}{c_1 |h|^{2s-1}} .
\end{equation*}
Then, $\mathcal{L}_{\R^d_+} v_1^h =0$ on $\R^d_+$ and
\begin{equation*}
\|  v_1^h \| _{L^\infty (B_R)} \leq \frac{1}{c_1} [ u] _{ C^{2s-1} (B_{2R})} \leq R^{\gamma -\beta } 
\end{equation*} 
for all $R\geq \max \{ 1 ,|h| \}$.

Now, we repeat this argument inductively. For $k\geq 2$ define
\begin{equation*}
v_k^h(x) = \frac{v_{k-1}^h (x+h) - v_{k-1}^h (x)}{c_1 |h| ^{d+2s} } .
\end{equation*}
Then $\mathcal{L} _{\R^d_+} v_{k}^h =0$ on $\R^d_+$, and by the above claim, $\| v_k \| _{L^\infty (B_R)} \leq R^{\gamma -k \beta }$ for large $R$.
Let $\tilde{k}$ be the smallest positive integer such that $\gamma - \tilde{k} \beta <0$. By sending $R\to \infty$, we obtain $v _{\tilde{k}}^h =0$.
Since this is true for all $h\in \R^{d-1} \times \{ 0\}$, we know that $u$ is polynomial of degree at most $\tilde{k}-1$ in the first $d-1$ variables.
Using \eqref{eq:GrowthAssumptionLiouvilleThm} and the fact that $2s-1+\alpha _* <2$, we conclude that $u$ is just a polynomial of degree at most $1$ in the first $d-1$ variables.
Hence,
\begin{equation*}
u(x) = \sum _{i=1}^{d-1} u_i(x_d) x_i + \tilde{u} (x_d)
\end{equation*}
for some functions $u_i \colon \R_+ \to \R$ and $\tilde{u} \colon \R_+ \to \R$.
Due to the symmetry of the integration domain in the operator $\mathcal{L}_{\R^d_+}$ in the $\R^{d-1}\times \{ 0\}$ direction, we have
\begin{equation*}
0 = \mathcal{L}_{\R^d_+} u = \sum _{i=1}^{n-1} x_i\mathcal{L}_{\R^d_+} u_i + \mathcal{L}_{\R^d_+} \tilde{ u} \qquad \text{in} \quad \R^d_+ . 
\end{equation*}
This implies \smash{$\mathcal{L}_{\R^d_+} u_i  = 0$} in $\R^d_+$ for all $i\in \{ 1, \dots , d-1\}$.
Using the above prove of \autoref{Prop:LiouvilleInHalfspace}, we conclude that $u_i(x_d) =c_i x_d^{2s-1}$ for some constant $c_i\in \R$.
However, if $c_i\neq 0$, this would contradict \eqref{eq:GrowthAssumptionLiouvilleThm}.
This implies $u(x)=\tilde{u}(x_d)$.
\end{proof}
\section{Higher boundary regularity}\label{sec:HigherBoundaryReg}
Let $\Omega \subset \R^d$ be a bounded $C^{1,1}$ domain.
Recall that $d_\Omega$ is  $C^{1,1}$ close to the boundary $\partial \Omega$.
For convenience, we will replace $d_\Omega$ in the interior with a $C^{1,1}$ function.
To be more precise, we fix $d_0>0$ and a function $\delta _\Omega \in C^{1,1} (\overline{\Omega} )$ such that
\begin{equation*}
\delta _\Omega = d_\Omega \quad \text{on} \quad \{ 0\leq d_\Omega (x) \leq d_0 \} \quad  \text{and} \quad \frac{1}{2} d_\Omega \leq \delta _\Omega \leq  d_\Omega \quad \text{on} \quad \Omega .
\end{equation*}
In this section, we will prove \autoref{Thm:HigherOrderBoundaryReg} using a standard blow-up argument together with the Liouville type result in the half-space.
We mostly follow the proof of \cite[Sec. 2.7]{RosOton2024Book}.
\begin{Lemma}\label{Lem:ExpansionHigherBoundaryReg}
Fix $\sigma \in (0,1]$ and $s\in (1/2,1)$.
Let $\Omega \subset \R^d$ be a bounded $C^{1,1}$ domain with $0\in\partial \Omega $.
Assume that $u\in H^s_0(\Omega )$ is a weak solution to
\begin{gather*}
\begin{aligned}
\mathcal{L}_{\Omega , \sigma} u &= f \quad \quad \text{in} \quad \Omega , \\
u& =0 \quad \quad \text{on} \quad  \partial \Omega 
\end{aligned}
\end{gather*}
where $f\in L^\infty (\Omega )$.
Then there exists $q_0\in \R$ and some $\alpha = \alpha (d,s,\sigma )>0$ such that
\begin{equation}\label{eq:ExpansionAroundPoint}
\| u -q_0\delta _\Omega ^{2s-1} \| _{L^\infty (\Omega \cap B_r )} \leq c \left( \| u \| _{L^\infty (\Omega )} + \| f \| _{L^\infty (\Omega )} \right) r^{2s-1+\alpha }
\end{equation}
for all $r\in (0,1]$ for some constant $c=c(\Omega , d, s, \sigma )>0$.
\end{Lemma}
\begin{proof}
We will choose $\alpha =\min \{ \alpha _* , 1- \gamma \}$ with $\alpha_*$ from \autoref{Prop:LiouvilleInHalfspace} and $\gamma$ from \autoref{Lem:BoundLd2sMinus1}.
After dividing the equation by some constant, we can assume that
\begin{equation*}
\| u \| _{L^\infty (\Omega )} + \| f \| _{L^\infty (\Omega )} \leq 1 .
\end{equation*}
\emph{Step 1:}  In the first step, we will show that we can replace $q_0$ in \eqref{eq:ExpansionAroundPoint} by $q_r$ depending on $r\in (0,1)$.
To be more precise, we claim that if
\begin{equation}\label{eq:ExpansionWithQr}
\forall r\in (0,1), \quad \exists q_r \in \R   \colon \quad \| u -q_r\delta _\Omega ^{2s-1} \| _{L^\infty (\Omega \cap B_r )} \leq c_1 r^{2s-1+\alpha }
\end{equation}
holds true, then $\| u -q_0\delta _\Omega ^{2s-1} \| _{L^\infty (\Omega \cap B_r )} \leq c c_1 r^{2s-1+\alpha }$ for some $q_0\in \R$ and $c=c(\Omega ,d, s)>0$ independent of $r$ and $|q_0| \leq c(c_1+1)$.

To prove this claim, assume that \eqref{eq:ExpansionWithQr} holds.
Using that $r^{2s-1}\leq c \| \delta _\Omega ^{2s-1} \| _{L^\infty (B_r)}$ for some constant $c>0$ depending on $\Omega$, we obtain for all $r\in (0, 1/2]$ and $\beta \in [1,2]$
\begin{gather}\label{eq:qrIsCauchyAtZero}
\begin{aligned}
|q_{\beta r} - q_r| & \leq c r^{-2s+1} \| q_{\beta r} \delta _\Omega ^{2s-1} - q_r \delta _\Omega ^{2s-1} \| _{L^\infty (B_r)} \\
& \leq c r^{-2s+1} \| u - q_r \delta _\Omega ^{2s-1} \| _{L^\infty (B_r)} + c r^{-2s+1} \| u - q_{\beta r} \delta _\Omega ^{2s-1} \| _{L^\infty (B_r)} \\
& \leq c c_1 r^{\alpha } .
\end{aligned}
\end{gather}
Furthermore, using that $\| u\| _{L^\infty (\Omega )} \leq 1$, we have
\begin{gather}\label{eq:qrBoundedForRBigger}
\begin{aligned}
|q_r| & \leq c r^{-2s+1} \| q_r  \delta _\Omega ^{2s-1} \| _{L^\infty (\Omega \cap B_r )} \\
& \leq c r^{-2s+1}\|u- q_r  \delta _\Omega ^{2s-1} \| _{L^\infty (\Omega \cap B_r )} +c  r^{-2s+1}\| u\| _{L^\infty (\Omega \cap B_1)} \\
& \leq c(c_1+1)
\end{aligned}
\end{gather}
for all $r\in (1/2,1)$.
Note that \eqref{eq:qrIsCauchyAtZero} and \eqref{eq:qrBoundedForRBigger} imply that the limit
\begin{equation*}
q_0:= \lim _{r \to 0 } q_r
\end{equation*}
exists.
Summing up the following geometric sequence, we obtain
\begin{equation*}
|q_0 -q_r| \leq \sum _{j=0}^\infty | q_{2^{-j}r} - q_{2^{-j-1}r} | \leq c c_1 r^{\alpha} \sum _{j=0}^\infty 2^{-j\alpha} \leq c c_1 r^\alpha
\end{equation*}
for all $r\in (0,1)$, which in particular implies $|q_0| \leq c (c_1+1)$.
Finally, for all $r\in (0,1)$ we get
\begin{align*}
\| u -q_0\delta _\Omega ^{2s-1} \| _{L^\infty (\Omega \cap B_r )} & \leq \| u -q_r\delta _\Omega ^{2s-1} \| _{L^\infty (\Omega \cap B_r )} +|q_r -q_0| \| \delta _\Omega ^{2s-1} \| _{L^\infty (B_r)}  \leq c c_1 r^{2s-1+\alpha } ,
\end{align*}
which proves the claim.

\emph{Step 2:} By step 1, it is enough to prove \eqref{eq:ExpansionWithQr}. We will argue by contradiction.
Assume that there exists sequences of functions $(u_m) \in H^s_0(\Omega )$ and $(f_m) \in L^\infty (\Omega )$ such that $\| u_m \| _{L^\infty (\Omega )} + \| f_m\| _{L^\infty (\Omega )} \leq 1$ with 
\begin{align*}
\mathcal{L}_\Omega u_m & =f_m \quad \text{in}  \quad\Omega , \\
u_m&=0 \quad \text{on} \quad\partial \Omega
\end{align*}
and
\begin{equation*}
 \sup _{r\in (0,1)} r^{-2s+1-\alpha } \| u_m -q_{m,r}\delta _\Omega ^{2s-1} \| _{L^\infty (\Omega \cap B_r )} \geq m
\end{equation*}
for any $q_{m,r}\in \R$.
We will extend $u_m$ and $f_m$ to functions on $\R^d$ by setting $u_m=f_m=0$ on $\Omega ^c$.
Furthermore, we will choose
\begin{equation}\label{eq:Choiceqmr}
q_{m,r} = \frac{\int _{B_r} u_m \delta _\Omega ^{2s-1}}{\int _{B_r} \delta _\Omega ^{2(2s-1)}} .
\end{equation}
Set 
\begin{equation*}
Q_{m,r}:= \| u_m-q_{m,r} \delta _\Omega ^{2s-1} \| _{L^\infty (B_r)} \quad \text{and} \quad \theta (\rho ) := \sup _{m\in \N} \sup _{\rho \leq r < 1} r^{-2s+1-\alpha} Q_{m,r} .
\end{equation*}
Then $\theta (\rho )$ is non increasing, finite for any $\rho >0$, and satisfies $\lim _{\rho \to 0}\theta (\rho ) = \infty$.
Hence, there exists sequences $(m_k)$ and $(r_k)$ with $m_k \to \infty$ and $r_k\to 0$ such that $r_k^{-2s+1-\alpha }Q _{m_k,r_k} \geq  \frac{1}{2}\theta (r_k)$, in particular
\begin{equation}\label{eq:rk-2s+1-aQExplodes}
r_k^{-2s+1-\alpha }Q _{m_k,r_k} \to \infty \qquad \text{as} \quad k \to \infty .
\end{equation}

\emph{Step 3:} Now, we define
\begin{equation*}
v_k(x):=\frac{u_{m_k} (r_kx)-q_{m_k,r_k} \delta _\Omega ^{2s-1} (r_kx)}{Q_{m_k,r_k}} .
\end{equation*}
By the definition of $Q_{m,r}$, we immediately get
\begin{equation} \label{eq:vkLinftyis1}
\| v_k \| _{L^\infty (B_1)} =1.
\end{equation}
Furthermore, we claim that
\begin{equation}\label{eq:vkGrowthCondition}
\| v_k\| _{L^\infty (B_R)} \leq c R^{2s-1+\alpha } \qquad \text{for} \quad 1\leq R \leq \frac{1}{2r_k} .
\end{equation}
Indeed, as in \eqref{eq:qrIsCauchyAtZero}, we have for any $\beta \in [1,2]$
\begin{equation*}
|q_{m,\beta r} - q_{m,r}| \leq c r^{-2s+1} \left( Q_{m,r} + Q_{m,\beta r} \right) \leq c r^\alpha \theta (r) .
\end{equation*}
Hence, if $1\leq R\leq 1/(2r)$ then there exists $\beta \in [1,2)$ and $N\in \N$ such that $R=\beta 2^N$ and we get
\begin{align*}
|q_{m,Rr}-q_{m,r} | & \leq c \sum _{j=0}^{N-1} (2^jr)^\alpha \theta (2^j r) +|q_{m,\beta 2^N r} - q_{m,2^Nr} | \\
& \leq c  r^\alpha 2^{N \alpha } \theta (r) + c r^\alpha 2^{N\alpha } \theta (2^N r) \\
& \leq c r^\alpha R^\alpha \theta (r) .
\end{align*}
Finally, we obtain
\begin{align*}
\| v_k\| _{L^\infty (B_R)} & = Q_{m_k,r_k}^{-1} \| u_{m_k} - q_{m_k,r_k} \delta _\Omega ^{2s-1} \| _{L^\infty (B_{Rr_k})} \\
& \leq Q_{m_k,r_k}^{-1} \left( Q_{m_k,Rr_k} + | q_{m_k,Rr_k} - q_{m_k,r_k} | \| \delta _\Omega ^{2s-1} \| _{L^\infty (B_{Rr_k})} \right) \\
& \leq Q_{m_k,r_k}^{-1} \theta (Rr_k) R^{2s-1+\alpha} r_k^{2s-1+\alpha } + c Q_{m_k,r_k}^{-1} R^\alpha r_k^\alpha \theta (r_k) (Rr_k)^{2s-1} \\
& \leq R^{2s-1+\alpha } \theta (Rr_k) 2 \left( \theta (r_k) \right)^{-1} + c \theta (r_k) Q_{m_k,r_k}^{-1} r_k^{2s-1+\alpha } R^{2s-1+\alpha } \\
& \leq c R^{2s-1+\alpha}
\end{align*}
which proves \eqref{eq:vkGrowthCondition}.

Using $|u_{m_k}| \lesssim \delta _\Omega ^{2s-1}$ (see \eqref{eq:uLeqDist}) and \eqref{eq:Choiceqmr}, we get that $|q_{m_k,r_k}| \leq c$ for some constant $c$ independent of $k$.
Hence, with $\| u_{m_k} \| _{L^\infty (\R^d )} \leq 1$, we get
\begin{equation*}
| v_k(x) |  \leq c Q_{m_k,r_k}^{-1} r_k^{2s-1+\alpha } |x|^{2s-1+\alpha }
\end{equation*}
for all $x\in \R^d\setminus B_{1/(2r_k)}$.
Hence, together with \eqref{eq:rk-2s+1-aQExplodes} and \eqref{eq:vkGrowthCondition}, we conclude
\begin{equation}\label{eq:LInftyTailBoundOnVk}
\| v_k \| _{L^\infty _{2s-1+\alpha } (\R^d)} \leq c .
\end{equation}
for some constant $c>0$ independent of $k$.

Let $\Omega _k :=\frac{1}{r_k} \Omega$.
Then, by definition of $q_{m,r }$, we have
\begin{equation}\label{eq:IntegralVkd21-1}
\int _{B_1} v_k(x) \delta ^{2s-1}_{\Omega _k}(x) \intd x = Q_{m_k,r_k}^{-1} r_k^{-2s+1} \int _{B_1}\left[ u_{m_k}(r_k x)\delta_\Omega^{2s-1}(r_kx)-q_{m_k ,r_k} \delta _\Omega ^{2(2s-1)}(r_kx) \right] \intd x      =0.
\end{equation}

\emph{Step 4:} 
The function $v_k$ satisfies
\begin{gather}\label{eq:EquationVk}
\begin{aligned}
\mathcal{L}_{\Omega _k }v_k &=\tilde{f}_k \quad \text{in } \Omega _k , \\
v_k&=0 \quad \text{on } \partial \Omega _k
\end{aligned}
\end{gather}
where
\begin{equation*}
\tilde{f}_k (x) = Q_{m_k,r_k}^{-1} r_k ^{2s} \left[ f_{m_k}(r_kx) -q_{m_k,r_k} \left( \mathcal{L}_\Omega \delta _\Omega ^{2s-1} \right) (r_kx)  \right] .
\end{equation*}
Recall from \autoref{Lem:BoundLd2sMinus1} that $| \mathcal{L}_\Omega \delta _\Omega ^{2s-1} | \lesssim d_\Omega ^{-\gamma }$ for some $\gamma \in (0,1)$.
Together with the fact that $d_{\Omega}(r_k x) =r_k d_{\Omega _k} (x)$, as well as $\| f_{m_k} \|_{L^\infty (\Omega )} \leq 1$ and $|q_{m_k,r_k}|\leq c$ we obtain
\begin{gather}\label{eq:LinftyBoundOnfk}
\begin{aligned}
\| \tilde{f}_k d_{\Omega _k}^\gamma \| _{L^\infty (\Omega _k)} & \leq Q_{m_k,r_k}^{-1} r_k^{2s}  \| d_{\Omega_k}^\gamma \| _{L^\infty (\Omega _k)} +  Q_{m_k,r_k}^{-1} r_k^{2s} |q_{m_k,r_k}| r_k^{-\gamma}\|d_{\Omega _k}^{-\gamma } d_{\Omega _k}^\gamma\| _{L^\infty (\Omega _k)}  \\
&\leq c Q_{m_k,r_k}^{-1} r_k^{2s-\gamma} \to 0 \quad (k\to \infty )
\end{aligned}
\end{gather}
where in the last step, we have used that $\alpha < 1- \gamma$ and $Q_{m_k,r_k}^{-1} r_k^{2s-1+\alpha} \to 0$ (see \eqref{eq:rk-2s+1-aQExplodes}).

\emph{Step 5:}
Using \autoref{Prop:HoelderRegulUptoBoundaryLocal} together with \eqref{eq:LInftyTailBoundOnVk} and \eqref{eq:LinftyBoundOnfk}, we obtain
\begin{equation}
\| v_k\| _{C^{2s-1}(B_R)} \leq c
\end{equation}
where $c>0$ is a constant depending on $R$ but not on $k$.

Hence, by compactness, up to a subsequence $(v_k)$ converges locally uniformly (actually in $C^\nu_{loc}$ for any $\nu < 2s-1$) in $\R^d$ to some $v_\infty$.
By taking limits in \eqref{eq:vkLinftyis1}, \eqref{eq:vkGrowthCondition} and \eqref{eq:IntegralVkd21-1}, $v_\infty$ satisfies
\begin{equation*}
 \| v_\infty \|_{L^\infty (B_1)} =1 , \quad  v_\infty \in L^\infty _{2s-1+\alpha } (\R^d ) \cap C(\R^d ) , \quad \int _{B_1} v_\infty(x) x_d^{2s-1} \intd x = 0
\end{equation*}
and $v_\infty = 0$ on $\R^d \setminus \R^d_+$.
Furthermore, all the $v_k$ satisfy \eqref{eq:EquationVk} in the weak sense. Therefore, they also satisfy the equation in the distributional sense, i.e.
\begin{equation}\label{eq:DistributionalEquationVj}
\int _{\Omega _k} v_k \mathcal{L}_{\Omega _k}\varphi =  \int _{\Omega _k} \tilde{f}_k \varphi , \quad \text{for all } \varphi \in C^\infty _c (\Omega _j) .
\end{equation}
Now, we claim that $\mathcal{L}_{\R^d_+} v_\infty =0$ on $\R^d_+$ in the distributional sense, i.e.
\begin{equation}\label{eq:DistributionalEquationVInfty}
\int _{\R^d_+} v_\infty \mathcal{L}_{\R^d_+}\varphi = 0 , \quad \text{for all } \varphi \in C^\infty _c (\R^d_+) .
\end{equation}
To prove this, we want to take limits in \eqref{eq:DistributionalEquationVj}.
Fix $\varphi \in C^\infty _c (\R^d_+)$.
Note, that $\varphi \in C^\infty _c (\Omega _k)$ for large $k$ since $\Omega _k \to \R^d_+$.
Since $\supp \varphi$ is away from $\partial \R^d_+$, $\tilde f _k$ converges to $0$ uniformly on $\supp \varphi$. Hence, by \eqref{eq:LinftyBoundOnfk}, we have
\begin{equation*}
\int _{\Omega _k} \varphi \tilde{f}_k \to 0 \qquad \text{as } k \to \infty .
\end{equation*}
On the other hand, it is easy to see that $v_k(x) \mathcal{L}_{\Omega _j} \varphi (x) \to v_\infty (x) \mathcal{L}_{\R^d_+} \varphi (x)$ pointwise for every $x \in \R^d_+$.
Since $|v_k(x)|\leq c(1+|x| ^{2s-1+\alpha } )$ for all $x\in \R^d$ and
\begin{equation*}
|\mathcal{L}_{\Omega _k} \varphi (x) | \leq c (1+|x| ) ^{-d-2s} \qquad (x\in \R^d)
\end{equation*}
for some constant $c$ depending on $\varphi$, we can apply dominated convergence theorem to conclude
\begin{equation*}
\int _{\Omega _j} v_k \mathcal{L}_{\Omega _k}\varphi \to \int _{\R^d_+} v_\infty \mathcal{L}_{\R^d_+}\varphi
\end{equation*}
as $k\to \infty$. This proves \eqref{eq:DistributionalEquationVInfty}.

By \autoref{Lem:InteriorRegularityOnHalfspace}, $v_\infty$ is actually a strong solution $\mathcal{L}_{\R^d_+} v_\infty =0$ on $\R^d_+$.
Hence, we can apply \autoref{Prop:LiouvilleInHalfspace} to conclude that $v_\infty (x) =Ax_d^{2s-1}$ for some constant $A\in\R$.
Note that $\int _{B_1} v_\infty(x) x_d^{2s-1} \intd x = 0$ implies $A=0$. But this contradicts $\| v_\infty \|_{L^\infty (B_1) } =1$.
\end{proof}
\begin{Lemma} \label{Lem:InteriorRegularityOnHalfspace}
Fix $s\in (1/2,1)$, $\sigma \in (0,1]$, $\epsilon >0$, and set $\kappa =(2+3/\sigma )^{-1}$.
Let $u\in C(\overline{ \R^d_+}) \cap L^\infty _{2s-\epsilon} (\R^d)$ be a distributional solution to $\mathcal{L}_{\R^d_+,\sigma } u = f$ in $B_{\kappa }(e_d)$ , i.e.
\begin{equation*}
 \int _{\R^d_+} u(x) \mathcal{L}_{\R^d_+,\sigma} \varphi (x)  \intd x =  \int _{\R^d_+} f(x) \varphi (x) \intd x \qquad \text{for all} \quad \varphi \in C^\infty _c (B_{\kappa }(e_d)) ,
\end{equation*}
for some $f \in C^\alpha (B_{\kappa }(e_d))$ where $\alpha \in (0,1)$ with $2s+\alpha \notin \N$.

Then $u \in C^{2s+\alpha} (B_{\kappa /4}(e_d))$ with
\begin{equation*}
\| u \| _{C^{2s+\alpha}(B_{\kappa /4})(e_d)} \leq c \left( \| u\| _{L^\infty _{2s-\epsilon} (\R^d)} + \| f\| _{C^\alpha (B_{\kappa }(e_d))} \right) 
\end{equation*}
for some $c=c(d,s,\sigma , \epsilon ,\alpha )>0$.
\end{Lemma}
\begin{proof}
After dividing by some constant, we can assume that $ \| u\| _{L^\infty _{2s-\epsilon} (\R^d)} + \| f\| _{C^\alpha (B_{\kappa }(e_d))} \leq 1$.
By \autoref{Lem:InteriorRegularity}, we know that $\| u \| _{C^{\alpha} (B_{\kappa /2})} \leq c$ for some constant $c=c(d,s,\sigma ,\epsilon )>0$.
As in \autoref{Lem:InteriorRegularity}, we rewrite our operator in terms of the fractional Laplacian $(-\Delta )^s$ and then use interior regularity results for $(-\Delta )^s$.
Notice that for all $x\in B_{\kappa /2}(e_d)$, we have 
\begin{gather}\label{eq:RewriteOperatorAsFractLap}
\begin{aligned}
(-\Delta )^s u(x) &= f(x) + u(x) \int _{\R^d} \frac{1-\mathcal{B}(x,y)}{|x-y|^{d+2s}} \intd y + \int _{\R^d} \frac{u(y)(1-\mathcal{B}(x,y))}{|x-y|^{d+2s}} \intd y \\
& = f(x) +u(x) h_1(x) + h_2(x)
\end{aligned}
\end{gather}
in the distributional sense.
We want to show that the right-hand side of \eqref{eq:RewriteOperatorAsFractLap} is $C^\alpha$.
By assumption we have $\| f \|_{ C^\alpha (B_{\kappa /2}(e_d))} \leq 1$.
It was shown in \autoref{Lem:InteriorRegularity} that $\| h_1 \| _{L^\infty (B_{\kappa /2}(e_d))} \leq c$ and $\| h_2 \| _{L^\infty (B_{\kappa /2}(e_d))} \leq c$ for some constant $c=c(d,s, \sigma , \epsilon )>0$.
Assume that we have already shown that $[h_1 ] _{C^\alpha (B_{\kappa /2}(e_d))} \leq c$ and $[h_2 ] _{C^\alpha (B_{\kappa /2}(e_d))} \leq c $ for some constant $c=c(d,s, \sigma , \epsilon ,\alpha )>0$. Then (using that the product of two bounded Hölder continuous functions is Hölder continuous), we get that
\begin{equation*}
 \| f +u h_1 + h_2 \|  _{C^\alpha (B_{\kappa /2}(e_d))} \leq c .
\end{equation*}
By interior regularity results for the fractional Laplacian (see, e.g., \cite[Thm. 2.4.1.]{RosOton2024Book}), we conclude that $\| u \| _{C^{2s+\alpha}(B_{\kappa /4})(e_d)} \leq c$ for some constant $c=c(d,s,\sigma ,\epsilon ,\alpha )>0$, which proves \autoref{Lem:InteriorRegularityOnHalfspace}.

It remains to bound $[h_2 ] _{C^\alpha (B_{\kappa /2}(e_d))}$ ($[h_1 ] _{C^\alpha (B_{\kappa /2}(e_d))}$ can be bounded similarly by setting $u=1$).
Notice that
\begin{equation*}
h_2(x) = \frac{1}{2} \int _{(B_{\sigma x_d}   (x))^c } \frac{u(y)}{|x-y|^{d+2s}} \intd y + \frac{1}{2} \int _{R(x) ^c} \frac{u(y)}{|x-y|^{d+2s}} \intd y
\end{equation*}
where $R(x)$ is the ellipsoid $E(x)$ (from \eqref{eq:EllipsoidEquation}) in the case $\sigma \in (0,1)$ or the paraboloid $P(x)$ (from \eqref{eq:ParaboloidEquation}) in the case $\sigma =1$.
Notice that by rescaling the claim from the proof of \autoref{Lem:InteriorRegularity}, we see that
\begin{equation*}
\mathcal{B}(x,y) =1 \quad \text{for all} \quad x\in B_\kappa (e_d) , y \in B_{2\kappa} (e_d) .
\end{equation*}
As a consequence for $x,z\in B_{\kappa} (e_d)$, we have $|x-y| \geq \kappa$ for all $y \in (B_{\sigma z_d}   (z))^c \cup R(z) ^c $.
Hence, in the arguments below, we are always away from the singularity.

Let us fix $x, z\in B_{\kappa /2}(e_d)$.
Then,
\begin{align*}
2 |h_2 (x) -h_2(z)| & \leq  \int _{(B_{\sigma x_d}   (x))^c } |u(y)| \left| \frac{1}{|x-y|^{d+2s}} -\frac{1}{|z-y|^{d+2s}} \right|  \intd y  \\
 & \qquad  + \left| \int _{(B_{\sigma x_d}   (x))^c } \frac{u(y)}{|z-y|^{d+2s}} \intd y - \int _{(B_{\sigma z_d}   (z))^c} \frac{u(y)}{|z-y|^{d+2s}} \intd y \right| \\
 &\qquad + \int _{R(x)^c } |u(y)| \left| \frac{1}{|x-y|^{d+2s}} -\frac{1}{|z-y|^{d+2s}} \right|  \intd y  \\
 & \qquad  + \left| \int _{R(x)^c } \frac{u(y)}{|z-y|^{d+2s}} \intd y - \int _{R(z)^c} \frac{u(y)}{|z-y|^{d+2s}} \intd y \right| \\
 & = I_1 +I_2+I_3 +I_4 .
\end{align*}
We start with $I_1$.
Notice that $|x-y|$ and $|z-y|$ in the integral of $I_1$ are both bounded from below. Hence, we can use the $C^1$ regularity from $| \cdot | ^{-d-2s}$.
To be more precise, using the mean value formula, for every $y \in (B _{\sigma x_d}(x))^c$ there exists $w_y\in B_{\kappa /2}(e_d)$ such that
\begin{align*}
I_1 & \leq c |x-z| \int _{(B_{\sigma x_d}   (x))^c } |u(y)|  |w_y-y|^{-d-2s-1}\intd y \\
& \leq c |x-z| \| u\| _{L^\infty _{2s-\epsilon } (\R^d)} \int _{(B_{\sigma x_d}} (1+|y| ) ^{-d-1-\epsilon }\intd y \\
& \leq c |x-z| \| u\| _{L^\infty _{2s-\epsilon } (\R^d)}
\end{align*}
for some constant $c=c(d,s,\sigma ,\epsilon )$.
Similarly, we can bound $I_3 \leq c |x-z|$ for some constant $c=c(d,s,\sigma ,\epsilon )$.

Notice that for $I_2$, only those $y$ in the integral remain, which are contained in the symmetric difference $(B_{\sigma x_d}   (x))^c \Delta (B_{\sigma z_d}   (z))^c$.
By a direct computation, we see that
\begin{equation*}
(B_{\sigma x_d}   (x))^c \Delta (B_{\sigma z_d}   (z))^c \subset B_{\sigma x_d +2 |x-z|} (x) \setminus B_{\sigma x_d -2 |x-z|} (x) .
\end{equation*}
Hence, we can estimate $I_2$ by
\begin{equation*}
I_2 \leq \int _{B_{\sigma x_d +2 |x-z|} (x) \setminus B_{\sigma x_d -2 |x-z|} (x)}  \frac{|u(y)|}{|z-y|^{d+2s}} \intd y
\leq c \left| B_{\sigma x_d +2 |x-z|} (x) \setminus B_{\sigma x_d -2 |x-z|} (x)  \right|
\leq c |x-z|
\end{equation*}
for some $c=c(d,s,\sigma ,\epsilon )>0$, where we have used $\| u\| _{L^\infty _{2s-\epsilon} (\R^d)} \leq 1$ and the fact that $|z-y|$ in the integral is bounded by positive constants from above and below.

For $I_4$, also, only those $y$ in the symmetric difference $R(x)^c \Delta R(z)^c$ remain.
However, the geometric setting becomes much more complicated compared to $I_2$.

We will view the symmetric difference $D:=R(x)^c \Delta R(z)^c$ as the union
\begin{equation*}
D \subset \bigcup _{t\in [0,1]} \partial R((1-t)x+tz)
\end{equation*}
of the topological boundaries of the ellipsoids or paraboloids $R((1-t)x+tz)$ of the points connecting $x$ and $z$ in a straight line.
Then, using the coarea formula, we will write the integral over $D$ as an iterated integral over $t\in [0,1]$ and over the hypersurfaces $\partial R((1-t)x+tz)$.

We start with the case $\sigma \in (0,1)$.
Using the explicit inequality from \eqref{eq:EllipsoidEquation}, which describes the ellipsoid, we define a bijective function
\begin{align*}
g_t \colon \partial E(e_d) & \to \partial E((1-t)x+tz) \\
\begin{pmatrix}
y' \\
y_d
\end{pmatrix} & \mapsto \begin{pmatrix}
((1-t)x_d+tz_d) y' + (1-t)x' +t z' \\
((1-t)x_d+tz_d) \left( y_d -\frac{1}{1-\sigma ^2}  \right) + \frac{((1-t)x_d+tz_d)}{1-\sigma ^2}
\end{pmatrix}
\end{align*}
for every $t\in [0,1]$, which maps $\partial E(e_d)$ onto $\partial E((1-t)x+tz)$ by scaling and shifting the ellipsoid. 

By a direct computation, we can bound the derivative
\begin{equation}\label{eq:BoundDerivativeGtCaseSigmaLess1}
\left| \partial_t g_t(y)  \right| \leq c |x-z| (1+|y|)  \quad \text{for all}\quad y\in \partial E(e_d)
\end{equation}
for some constant $c=c(d, \sigma )$.
Furthermore, notice that $1+ |y|$ is bounded for all $y\in \partial E(e_d)$ by some constant depending only on $\sigma$.
Hence, $\left| \partial_t g_t(y)  \right| \leq c |x-z|$ for all $y\in \partial E(e_d)$.

Using the coarea formula, we obtain
\begin{align*}
I_4 & \leq \int _D \frac{|u(y)|}{|z-y|^{d+2s}} \intd y \\
& \leq \int _0^1 \int _{g_t(\partial E(e_d))}  \frac{|u(y)|}{|z-y|^{d+2s}} \left| \partial_t g_t(y) \right| \intd H_{d-1}(y)   \intd t \\
& \leq c |x-z| \int _0^1 \int _{ \partial E((1-t)x+tz) }  \frac{|u(y)|}{|z-y|^{d+2s}} \intd H_{d-1}(y)  \intd t 
\end{align*}
where $H_{d-1}$ is the $d-1$ dimensional Hausdorff measure.
Notice that in the last integral, both the ellipsoids $H_{d-1} (\partial E((1-t)x+tz))$ and the integrand are bounded.
This implies $I_4 \leq c |x-z|$ for some $c=c(d,s,\sigma )>0$ in the case $\sigma \in (0,1)$.

Now we consider the case $\sigma =1$.
We will proceed similarly as above.
The key differences are that both $\partial P((1-t)x+tz)$ as well as the right-hand side of \eqref{eq:BoundDerivativeGtCaseSigmaLess1} are not bounded anymore.
Using the explicit description of $P(x)$ from \eqref{eq:ParaboloidEquation}, we define a bijective map
\begin{align*}
g_t \colon \partial P(e_d) & \to \partial P((1-t)x+tz) \\
\begin{pmatrix}
y' \\
y_d
\end{pmatrix} & \mapsto \begin{pmatrix}
((1-t)x_d+tz_d) y' + (1-t)x' +t z' \\
 y_d -\frac{1}{2}   + \frac{((1-t)x_d+tz_d)}{2} 
\end{pmatrix} .
\end{align*}
A direct computation yields
\begin{equation*}
\left| \partial_t g_t(y)  \right| \leq c |x-z| (1+|y'|)  \quad \text{for all}\quad y\in \partial P(e_d)
\end{equation*}
for some constant $c=c(d)>0$.
Using the coarea formula, we get
\begin{equation}\label{eq:BoundI4CaseSigma1}
I_4  \leq c |x-z| \int _0^1 \int _{ \partial P((1-t)x+tz) }  \frac{|u(y)| (1+|y'| )}{|z-y|^{d+2s}} \intd H_{d-1}(y)  \intd t .
\end{equation}
We claim that
\begin{equation}\label{eq:BoundIntegralOverParaboloid}
\int _{ \partial P(w) }  \frac{|u(y)| (1+|y'| )}{|z-y|^{d+2s}} \intd H_{d-1}(y) \leq c
\end{equation}
for any $w\in B_{\kappa /2}(e_d)$, for some constant $c=c(d,s )>0$.

Notice, that $\partial P(w)$ is the graph of the function $h\colon \R^{d-1} \to \R$
\begin{equation*}
h(y'):= \frac{|y'-w|^2}{2w_d} +\frac{w_d}{2} .
\end{equation*}
Since $h(y') \asymp |y'|^2 +1$, we have $|(y',h(y'))| \asymp |y'|^2 +1$
Using that $\left| \nabla h (y') \right| \asymp |y'| +1$, we obtain
\begin{align*}
\int _{ \partial P(w) }  \frac{|u(y)| (1+|y'| )}{|z-y|^{d+2s}} \intd H_{d-1}(y) & \leq c \| u\| _{L^\infty_{2s-\epsilon} (\R^d)} \int _{ \partial P(w) }  \frac{ 1+|y'| }{(1+|y|)^{d+\epsilon}} \intd H_{d-1}(y) \\
& \leq c \int_{\R^{d-1}} \left| \nabla h (y') \right| \frac{1+|y'|}{(1+ |(y',h(y')|)^{d+\epsilon}} \intd y' \\
& \leq c \int_{\R^{d-1}} \frac{1}{(1+ |y'|)^{2d+2\epsilon -2}} \intd y' .
\end{align*}
Note that the above integral is finite.
This proves \eqref{eq:BoundIntegralOverParaboloid}. Combining \eqref{eq:BoundI4CaseSigma1} and \eqref{eq:BoundIntegralOverParaboloid}, we conclude that  $I_4 \leq c |x-z|$ for some $c=c(d,s, \sigma )>0$ in the case $\sigma =1$ as well.
Hence, we have shown that $|h_2(x)-h_2(z)| \leq I_1 +I_2+I_3+ I_4 \leq c |x-z|$, which proves that $h_2 \in C^\alpha (B_{\kappa /2}(e_d))$.
\end{proof}

\begin{proof}[Proof of \autoref{Thm:HigherOrderBoundaryReg}]
Fix $\alpha \in (0,1)$ from \autoref{Lem:ExpansionHigherBoundaryReg}.
By dividing the equation by some constant, we can assume that $\| f\| _{L^\infty (\Omega )} \leq 1$. Furthermore, by \autoref{Prop:LinftyBound}, we can also assume that $\| u \| _{L^\infty (\Omega )} \leq 1$.
We will now show that $ \| u / \delta _\Omega ^{2s-1} \| _{C^\alpha (\Omega )} \leq c$ for some constant $c=c(\Omega , d ,s,\sigma )>0$. Note, that this implies \eqref{eq:HigherBoundaryRegularity}, since by interior regularity results $u/d_\Omega ^{2s-1}\in C^\alpha _{loc}(\Omega )$ and $d_\Omega$ coincides with $\delta _\Omega$ close to the boundary.

Let $x_0 \in \Omega$ be a point with $\rho := d_\Omega (x_0)$, and choose $z_0 \in \partial \Omega$ such that $|x_0 -z_0| = \rho$.
Set $\kappa =\frac{1}{2(2+3/\sigma )}$ and define
\begin{equation*}
v(x):=u(x_0 + 2\kappa \rho x) - q(z_0) \delta _\Omega^{2s-1} (x_0 + 2 \kappa \rho x)
\end{equation*}
where $q(z_0)$ is the constant $q_0$ from \autoref{Lem:ExpansionHigherBoundaryReg} when the point $z_0$ is shifted to $0\in \R^d$.
Note, that $\| v \| _{L^\infty (\R^d )} \leq  c$ for some constant $c=c( \Omega ,d,s, \sigma )>0$.
Furthermore, using \autoref{Lem:ExpansionHigherBoundaryReg}, we obtain the weighted $L^\infty$-bound
\begin{equation*}
\| v \| _{L^\infty _{2s-1+\alpha } (\R^d )} \leq c \rho ^{2s-1+\alpha }
\end{equation*}
for some constant $c=c(\Omega ,d,s, \sigma )>0$.

Moreover, for $x\in B_1$, we have in the weak sense
\begin{align*}
| \mathcal{L}_{(2\kappa \rho )^{-1}(\Omega - x_0)} v(x) | & \leq c\rho ^{2s} | f(x_0 + 2\kappa \rho x) | + c |q(z_0)| \rho ^{2s} d_\Omega ^{-\gamma} (x_0 + 2\kappa \rho x)  \\
& \leq c \rho ^{2s} + c \rho ^{2s - \gamma} \leq c \rho ^{2s-1+\alpha }
\end{align*}
for some constant $c=c(\Omega ,d,s, \sigma )>0$, where we have used \autoref{Lem:BoundLd2sMinus1} and the fact that $\alpha < 1-\gamma$.

By \autoref{Lem:InteriorRegularity}, we get
\begin{align*}
[u-q(z_0) \delta _\Omega^{2s-1} ]_{C^\alpha (B_{\kappa \rho} (x_0))} & = c\rho ^{-\alpha } [v]_{C^\alpha (B_{1/2})} \\
&\leq c \rho ^{- \alpha } \left( \| v \| _{L^\infty _{2s-1+\alpha } (\R^d )} + \|  \mathcal{L}_{(2\kappa \rho )^{-1}(\Omega - x_0)} v(x)  \| _{L^\infty (B_1)} \right) \\
& \leq c \rho ^{2s-1 }
\end{align*}
for some constant $c=c(\Omega ,d,s, \sigma )>0$.

Since 
\begin{equation*}
\| \nabla \delta _\Omega ^{-2s+1} \| _{L^\infty (B _{\kappa \rho } (x_0))} \leq c \| \delta _\Omega ^{-2s}\| _{L^\infty (B _{\kappa \rho } (x_0))} \| \nabla \delta _\Omega \| _{L^\infty (B _{\kappa \rho } (x_0))} \leq c \rho ^{-2s} 
\end{equation*}
we obtain by interpolation
\begin{equation*}
[ \delta _\Omega^{-2s+1} ] _{C^\alpha (B _{\kappa \rho } (x_0))}  \leq \left[ \delta _\Omega^{-2s+1} \right] _{C^{0,1} (B _{\kappa \rho } (x_0))} ^{\alpha} \| \delta _\Omega^{-2s+1} \| _{L^\infty (B _{\kappa \rho } (x_0))} ^{1-\alpha } \leq c \rho ^{-2s+1-\alpha }.
\end{equation*}
Hence, we get
\begin{align*}
\left[ \frac{u}{\delta _\Omega ^{2s-1}} \right] _{C^\alpha (B _{\kappa \rho } (x_0))} & =  \left[ \frac{u}{\delta _\Omega ^{2s-1}} -q(z_0) \right] _{C^\alpha (B _{\kappa \rho } (x_0))} \\
& \leq  \| u -q(z_0)\delta _\Omega^{2s-1} \| _{L^\infty (B _{\kappa \rho } (x_0))}  \left[ \delta _\Omega ^{-2s+1} \right] _{C^\alpha (B _{\kappa \rho } (x_0))} \\
&  \qquad + \left[ u -q(z_0)\delta _\Omega^{2s-1} \right] _{C^\alpha (B _{\kappa \rho } (x_0))}  \left\| \delta _\Omega ^{-2s+1} \right\| _{L^\infty (B _{\kappa \rho } (x_0))} \\
& \leq c \rho ^{2s-1+\alpha} \rho ^{-2s+1} + c \rho ^{2s-1} \rho ^{-2s+1} \leq c
\end{align*}
for some constant $c=c(\Omega ,d,s, \sigma )>0$.
However, this is enough to conclude $ \| u / \delta _\Omega ^{2s-1} \| _{C^\alpha (\Omega )} \leq c$ (see, e.g., \cite[Lem. A.1.4]{RosOton2024Book}).
\end{proof}

\appendix
\section{Auxiliary Lemmas}
\begin{Lemma}[Scaling and Shifting]\label{Lem:ScalingOperator}
Fix $s\in (1/2,1)$, $\sigma \in (0,1]$, and $\epsilon >0$.
Let $u\in C^{2s+\epsilon}_{loc} (\Omega )$ satisfy $\mathcal{L}_{\Omega ,\sigma } u=f$ in $\Omega$ and define $\tilde{u}(x)=u(x_0+r x)$ for some $x_0 \in \R^d$ and $r>0$.
Then 
\begin{equation*}
\mathcal{L}_{r^{-1}(\Omega - x_0) , \sigma} \tilde{ u} (x) = r ^{2s} f(x_0+rx)
\end{equation*}
where $r^{-1}(\Omega - x_0) := \left\{ r^{-1}(y - x_0) \in \R^d \mid y \in \Omega  \right\}$.
This result remains true for weak solutions.
\end{Lemma}
\begin{proof}
Note that
\begin{equation*}
d_{r^{-1}(\Omega - x_0)} (x) = r^{-1} d_{\Omega -x_0}(rx) = r^{-1} d_\Omega (x_0 + rx)
\end{equation*}
for all $x\in r^{-1}(\Omega - x_0)$.
Hence, by setting
\begin{equation*}
\mathcal{B}_\Omega (x,y)= \frac{1}{2} \ind _{\left\{ y\in B_{\sigma d_\Omega (x)}(x) \right\} } + \frac{1}{2} \ind _{\left\{ x\in B_{\sigma d_\Omega (y)}(y) \right\} } ,
\end{equation*}
we also get
\begin{equation*}
\mathcal{B}_{r^{-1}(\Omega - x_0)} (x,y) = \mathcal{B}_\Omega (x_0+rx , x_0 +ry) .
\end{equation*}
By a change of variables, we obtain 
\begin{align*}
\mathcal{L}_{r^{-1}(\Omega - x_0) } \tilde{ u} (x) & = \pv \int _{r^{-1}(\Omega - x_0)} \frac{\tilde{u}(x) - \tilde{u}(y)}{|x-y|^{d+2s}} \mathcal{B}_{r^{-1}(\Omega - x_0)} (x,y) \intd y \\
& =  \pv \int _{r^{-1}(\Omega - x_0)} r^{d+2s} \frac{u(x_0 +rx) - u(x_0+ry)}{|(x_0+rx)-(x_0+ry)|^{d+2s}} \mathcal{B}_\Omega (x_0+rx , x_0 +ry) \intd y \\
& = \pv \int _\Omega r^{2s} \frac{u(x_0 +rx) - u(y')}{|x_0+rx-y'|^{d+2s}} \mathcal{B}_\Omega (x_0+rx ,y') \intd y' \\
& = r^{2s} \mathcal{L}_{\Omega } u (x_0 +rx) .
\end{align*}
For a weak solution, a similar calculation yields
\begin{equation*}
\mathcal{E}_{r^{-1}(\Omega - x_0) } (\tilde{u} , \tilde{\varphi} ) = r^{2s} \int _{r^{-1}(\Omega - x_0)}  f(x_0+rx) \tilde{\varphi } (x) \intd x
\end{equation*}
where $\tilde{\varphi} (x) = \varphi (x_0 +rx)$ for some $\varphi \in H^s_0(\Omega )$.
\end{proof}

\begin{Lemma}[Maximum principle for weak solutions]\label{Lem:WeakMaximumPrinciple}
Fix $s\in (1/2,1)$ and $\sigma \in (0,1]$.
Let $\Omega \subset \R^d$ be a Lipschitz domain and $\Omega ' \subseteq \Omega$ be a bounded domain also with Lipschitz boundary.
Assume that $u$ is a weak supersolution $\mathcal{L}_{\Omega ,\sigma } u \geq 0$ in $\Omega '$, i.e.
\begin{equation*}
\mathcal{E}_{\Omega , \sigma } (u ,\varphi )\geq 0, \qquad \forall \varphi\in H^s_0(\Omega '), \varphi \geq 0 .
\end{equation*}
Furthermore, assume $u\geq 0$ a.e. on $\overline{\Omega } \setminus \Omega '$ in the sense that $u_- \in H^s( \R^d)$ with $u_- = 0$ on $\R^d \setminus \Omega '$.
Then $u\geq 0$ a.e. in $\Omega$.
\end{Lemma}
\begin{proof}
We decompose $u$ into its positive and negative part $u=u_+ - u_-$.
Note that $u_- \in H^s_0 ( \Omega ' )$.
Hence, $\varphi :=u_-$ is an admissible test function.
Using that
\begin{align*}
\mathcal{E}_{\Omega  } (u_+, u_-) & = \frac{1}{2} \int_\Omega \int_\Omega (u_+(x)-u_+(y))(u_-(x)-u_-(y)) \mathcal{K}_\Omega (x,y) \intd y\intd x \\
& = - \frac{1}{2} \int_\Omega \int_\Omega \left( u_+(x)u_-(y)+u_+(y)u_-(x) \right) \mathcal{K}_ \Omega (x,y) \intd y\intd x 
 \leq 0 ,
\end{align*}
we obtain
\begin{equation*}
0 \leq \mathcal{E}_{\Omega } (u, \varphi ) = \mathcal{E}_{\Omega  } (u_+ , u_- ) -\mathcal{E}_{\Omega } (u_- , u_- ) \leq -\mathcal{E}_{\Omega } (u_- , u_- ) \leq 0.
\end{equation*}
This implies
\begin{align*}
0= \mathcal{E}_{\Omega } (u_- , u_-) &=\frac{1}{2}\iint _{\Omega \times \Omega} \frac{\left( u_- (x) -u_-(y) \right) ^2}{|x-y|^{d+2s}} \mathcal{B}_\Omega (x,y) \intd x \intd y  \\
& \geq \frac{1}{2} \iint _{\Omega ' \times \Omega '} \frac{\left( u_- (x) -u_-(y) \right) ^2}{|x-y|^{d+2s}} \mathcal{B}_{\Omega '} (x,y) \intd x \intd y \\
 & =  \mathcal{E}_{\Omega '} (u_- , u_-) \asymp [u_- ]_{H^s(\Omega ')} 
\end{align*}
where for the last step, we use the comparability from \autoref{Cor:EnergiesComparable}.
Hence, $[u_-]_{H^s(\Omega ')}=0$ which implies that $u_- $ is constant a.e. in $\Omega '$.
Since by assumption $u_- \in H^s(\R^d)$ and $u_-=0$ a.e. in $\R^d\setminus \Omega '$, we conclude that $u\geq 0$ a.e. in $\R^d$.
\end{proof}
\begin{Lemma}[Maximum principle for strong solutions]\label{Lem:MaximumPrincipleStrongSol}
Fix $s\in (1/2,1)$, $\sigma \in (0,1]$, and $\epsilon >0$.
Let $\Omega \subset \R^d$ be a Lipschitz domain and $\Omega ' \subseteq \Omega$ be a bounded domain also with Lipschitz boundary.
Assume that $u \in C^{2s+\epsilon}_{loc}(\Omega ') \cap C (\overline{\Omega } )$ is a supersolution in $\Omega '$, i.e.
\begin{equation*}
\mathcal{L}_{\Omega ,\sigma } u (x) \geq 0 \quad \text{for all} \quad x\in \Omega ' .
\end{equation*}
Furthermore, assume $u\geq 0$ on $\overline{\Omega} \setminus \Omega '$.
Then $u\geq 0$ in $\overline{\Omega }$.
\end{Lemma}
\begin{proof}
We will argue by contradiction.
Let $x_0 \in \overline{\Omega'}$ such that $u(x_0)=\min _{x\in \overline{\Omega '}} u(x)$ and assume that $x_0\in \Omega '$ with $u(x_0)<0$.
Then 
\begin{equation} \label{eq:LUSandwichForStrongMaxPrinciple}
0 \leq \mathcal{L}_\Omega u(x_0) = \pv \int _\Omega \frac{u(x_0)-u(y)} {|x_0-y|^{d+2s}} \mathcal{B}_\Omega (x_0,y) \intd y \leq 0
\end{equation}
where for the last inequality we have used that $u(x_0) = \inf _{x\in \Omega} u(x)$, since $u\geq 0 > u(x_0)$ on $\Omega \setminus \Omega '$.
Recall that there exists some $c=c(\sigma )>0$ such that $\mathcal{B} _\Omega (x_0, y ) >0$ for all $y\in  B_{cd_\Omega (x_0)}(x_0)$ (see \eqref{eq:BLocally1AroundPoint}).
Hence using \eqref{eq:LUSandwichForStrongMaxPrinciple}, we conclude $u(y)=u(x_0)<0$ for all $y\in B_{cd_\Omega (x_0)}(x_0)$.

\textbf{Claim}: $u(y)=u(x_0)<0$ for all $y\in \Omega '$.

To see this, let $y_0 \in \Omega '$ and connect $x_0$ with $y_0$ by a path inside $\Omega '$.
Then choose $x_1, \dots, x_n$ on this path such that the balls $B_{cd_\Omega (x_i)}(x_i)$ cover the path.
Repeating the above argument $n$ times yields that $u(y_0)=u(x_0)$.

However, the above claim contradicts the assumption $u\in C(\overline{\Omega })$ with $u\geq 0$ on $\overline{\Omega} \setminus \Omega '$. 
\end{proof}


\printbibliography

\end{document}